\documentclass[11pt]{amsart}
\usepackage{amssymb,latexsym,amsmath,graphicx,graphics,epic,eepic,caption,subcaption}

\theoremstyle{plain}
\newtheorem{theorem}{Theorem}
\numberwithin{theorem}{section}
\newtheorem{lemma}[theorem]{Lemma}
\newtheorem{proposition}[theorem]{Proposition}
\newtheorem{corollary}[theorem]{Corollary}

\theoremstyle{definition}
\newtheorem{definition}[theorem]{Definition}
\newtheorem{example}[theorem]{Example}

\newtheorem{remark}[theorem]{Remark}

\newtheorem{criterion}[theorem]{Criterion}
\newtheorem{principle}[theorem]{Principle}

\newcommand{\C}{{\mathbb C}}
\newcommand{\R}{{\mathbb R}}
\newcommand{\Z}{{\mathbb Z}}
\newcommand{\Q}{{\mathbb Q}}
\renewcommand{\P}{{\mathbb P}}
\newcommand{\s}{{\mathbb S}}

\newcommand{\I}{{\mathcal I}}

             \newcommand{\J}{{\mathcal J}}
              \newcommand{\M}{{\mathcal M}}
              
              \newcommand{\U}{{\mathcal U}}

              \newcommand{\E}{{\mathcal E}}

              \renewcommand{\O}{{\mathcal O}}
              \newcommand{\N}{{\mathcal N}}
              \newcommand{\G}{{\mathcal G}}

               \renewcommand{\S}{{\mathcal S}}
               \newcommand{\p}{{\mathcal P}}

\begin{document}
\title{Symplectic configurations: a homological and computer-aided approach}
\author{Weimin Chen}
\subjclass[2000]{}
\thanks{}
\date{\today}
\maketitle
%\centerline{\bf Working draft -- Please do not circulate without permission!}
\begin{abstract}
Motivated by and extending the technical results in our earlier work \cite{C,C1} on symplectic Calabi-Yau $4$-manifolds, a general and systematic approach for studying certain unions of symplectic embedded surfaces in a rational $4$-manifold $X=\C\P^2\# N\overline{\C\P^2}$ is formulated, which may find applications in a broader 
range of problems. A distinct feature of this method is that it is computer-aided.  We address several fundamental theoretical questions concerning the computational aspect. On the other hand, we also establish a symplectic analog of Cremona transformations from algebraic geometry, which is another fundamental feature and a main technical tool of this method. For an illustration, we give a new proof that a certain line arrangement in $\C\P^2$, called Fano planes, cannot exist in the symplectic category. The nonexistence of Fano planes in the algebraic category follows from a theorem of Hirzebruch \cite{H}, while in the topological category, including the symplectic category, it was first proved by Ruberman and Starkston \cite{RuS}. Our proof for the symplectic category is independent to both, and is by combining the Cremona transformation technique with Gromov's theory of pseudoholomorphic curves. 
\end{abstract}
\tableofcontents

\section{The general scheme and main results}
\subsection{The general problem and an outline of strategy}
Let $X=\C\P^2\# N\overline{\C\P^2}$ be a rational $4$-manifold, let $D=\cup_{k=1}^n F_k$ be a union of smoothly embedded, oriented surfaces in $X$, which obeys the following condition:
\begin{itemize}
\item [{(\dag)}] Any two $F_k,F_l$ in $D$ are either disjoint, or intersect transversely and positively at one point, and no three distinct components of $D$ meet in one point. 
\end{itemize}
We call such $D$ a {\bf symplectic configuration} if there is a symplectic structure $\omega$ on $X$ with respect to which each surface $F_k$ is symplectic. We should point out that we do not assume $D$ is connected 
here as one usually does. Furthermore, for simplicity and without loss of generality, we shall assume that the symplectic structures $\omega$, with respect to which $D$ is symplectic and which are auxiliary and will not be fixed, define the same canonical line bundle up to an isomorphism, which will be denoted by $K_X$. Finally, we shall only be concerned with the embeddings of $D$ up to a smooth isotopy. 

A scheme for analyzing such configurations $D$ naturally emerged in our earlier work \cite{C} where, following the general strategy proposed in \cite{C3}, we reduce the study of symplectic finite group actions on a symplectic Calabi-Yau $4$-manifold to the corresponding question concerning the existence and classification of certain 
symplectic configurations in the corresponding symplectic resolution of the quotient orbifolds, which is a rational $4$-manifold $X=\C\P^2\# N\overline{\C\P^2}$ for the most important case. A technical foundation was laid in \cite{C} for such a study, which was further expanded in \cite{C1}. Building on the technical results in \cite{C,C1} and formalizing the scheme, we shall develop in this paper a systematic approach for studying general symplectic configurations in  a rational $4$-manifold $X=\C\P^2\# N\overline{\C\P^2}$, which aims to achieve the following specific objectives: for a given topological type of a configuration $D$ and a rational $4$-manifold $X=\C\P^2\# N\overline{\C\P^2}$:
\begin{itemize}
\item [{(i)}] To show that no symplectic embeddings of $D$ in $X$ exist. 
\item [{(ii)}] To show that a symplectic embedding of $D$ in $X$ implies the existence of a holomorphic embedding of $D$.
\item [{(iii)}] To construct a symplectic embedding of $D$ in $X$, thus proving existence. 
\item [{(iv)}] To describe the topology of the complement $X\setminus D$, including the classification of symplectic embeddings of $D$ in $X$.
\end{itemize}
These objectives, although being rooted in our project in \cite{C,C1}, are also applicable to a broad range of problems involving symplectic configurations in rational $4$-manifolds. 

\vspace{2mm}

Our method begins by considering the homological expression of the symplectic surfaces $F_k$ in $D$ with respect to a special basis $H,E_1,E_2,\cdots,E_N$ of $H^2(X)$ associated to the symplectic structure $\omega$, called a {\bf reduced basis} of $(X,\omega)$. In \cite{C1} we introduced a symplectic blowing-down procedure associated to each reduced basis, blowing down successively the classes $E_N,E_{N-1},\cdots, E_2, E_1$ to reduce the $4$-manifold $X=\C\P^2\# N\overline{\C\P^2}$ to $\C\P^2$ (for the most general case, to 
$\C\P^2\#\overline{\C\P^2}$), and to transform the symplectic configuration $D$ to a so-called {\bf symplectic arrangement} $\hat{D}$ in $\C\P^2$, which is a union of pseudoholomorphic curves whose singularities and intersection pattern are encoded in the homological expressions of the $F_k$'s with respect to the reduced basis $H,E_1,E_2,\cdots,E_N$. In a nutshell, we reduce the study of $D$ in $X$ to a study of $\hat{D}$ in $\C\P^2$,
which is more amenable to the existing tools, such that Gromov's theory of pseudoholomorphic curves and results or techniques from algebraic geometry. 

With this understood, a key observation we made in \cite{C} is that by specifying the symplectic area of the 
surfaces $F_k$ in $D$, in which the symplectic structure $\omega$ may be altered but the canonical class 
$c_1(K_X)$ remains the same, it is possible to eliminate the ``unwanted" homological expressions of the surfaces $F_k$, as a reduced basis of $(X,\omega)$ is area-constrainted (particularly, in an ideal situation if one can eliminate all possible homological expressions, then no symplectic embeddings of $D$ in $X$ can exist). This allows us to have control over the outcome of the successive symplectic blowing-down procedure applied to $D$, i.e., the symplectic arrangement $\hat{D}$ in $\C\P^2$, as well as the ability to choose between different choices of the successive symplectic blowing-down, which lead to different symplectic arrangements in $\C\P^2$. 

Note that in an analogous situation in the algebraic setting, the complex arrangements in $\C\P^2$ that are resulted from two different successive blowing-downs are related by a Cremona transformation, i.e., a birational automorphism of $\C\P^2$, see \cite{AC}. As the main technical result of this paper, we establish a symplectic analog of Cremona transformations, which allows us to construct from a given symplectic arrangement in $\C\P^2$ another symplectic arrangement which are ``Cremona equivalent" to each other in a certain algebraic sense, after we specify an ``algebraic Cremona equivalence". See Theorem 1.8 or the discussions in Section 3.4 for more details. It turns out that Cremona equivalence between different symplectic arrangements is an important feature of this method and a main technical tool.  In practice, the symplectic Cremona transformations will be used in combination with Gromov's theory of pseudoholomorphic curves as well as results or techniques from algebraic geometry. For an illustration and an application, see  Example 1.9, Theorem 1.11, Corollaries 1.12 and 1.13.

\vspace{2mm}

With the preceding discussions understood, it is clear that the following questions are fundamental issues pertaining to this method:

\begin{itemize}
\item [{(1)}] Under what conditions are there only finitely many possible homological expressions for $D$ (here the symplectic structure is not fixed)? What are the basic structural properties of the set of potential homological expressions?
\item [{(2)}] What is the mechanism to eliminate a homological expression of $D$ by specifying the areas of the surfaces $F_k$ in $D$? What is the set of possible areas for $F_k$ from which we can specify, and what are the basic structural properties?
\end{itemize}
Furthermore, note that our method has a natural, build-in computational component, i.e., under the assumption that there are only finitely many homological expressions of $D$, we will rely on computer programming to generate the set of all homological expressions of $D$, and then use computer programming to select an ``optimal" choice of the areas for the surfaces $F_k$ which we will specify. 

Keeping in mind the computational feature of this method, we formulate in this paper a fairly general and practically applicable condition, and show that under this condition, it suffices to only consider a set of finitely many homological expressions of $D$. Moreover, we formulate a fairly simple and computational-friendly criterion for eliminating a given homological expression of $D$ by a given choice of areas of the surfaces $F_k$. The more
``practical", computational related questions, such as algorithm design, computational efficiency, optimization, etc, 
will be dealt with when we apply the method to some specific (or specific type of) symplectic configurations, as how to handle these issues depends on the symplectic configurations under consideration. 

\vspace{2mm}

With the preceding understood, we shall next give an overview of this method, along the way introducing some fundamental concepts and stating the main results of this paper, which occupies the remaining part of this section.

\subsection{The computational aspect: basic concepts and a finiteness condition}
Let $D=\cup_{k=1}^n F_k$ be a symplectic configuration in $X=\C\P^2\# N\overline{\C\P^2}$, let $\omega$ be a symplectic structure on $X$ with respect to which each $F_k$ is symplectic. A reduced basis $H,E_1,E_2,\cdots,E_N$ of $(X,\omega)$ is a certain basis of $H^2(X)$ which has a standard intersection matrix, such that 
$$
c_1(K_X)=-3H+E_1+E_2+\cdots+ E_N.
$$
See Example 2.1(2) for a precise definition, and see \cite{C}, Section 3, \cite{C1}, Section 4, for further relevant details. Reduced bases always exist (cf. \cite{BP, LW}), and it is known that the symplectic areas of a reduced 
basis $H,E_1,E_2,\cdots,E_N$, i,e., 
$$
\lambda_0:=\omega(H), \lambda_i:=\omega(E_i) \mbox{  for } i=1,2,\cdots,N,
$$ 
determine the symplectic structure $\omega$ up to a symplectomorphism of $X$ (cf. \cite{KK}). Furthermore, for a generic symplectic structure $\omega$, reduced basis of $(X,\omega)$ is unique, see Lemma 2.9.

Let $A_k\in H^2(X)$ be the class of $F_k$. Write each $A_k$ in a reduced basis $H,E_1,E_2,\cdots,E_N$, 
$$
A_k:=a_kH-\sum_{i=1}^N b_{ki}E_i, \;\; a_k, b_{ki}\in\Z.
$$
Then the class $A_k$ determines a vector $\vec{v}_k=(a_k, b_{k1}, b_{k2}, \cdots, b_{kN})$, which is admissible in the sense of Definition 1.1 (cf. \cite{C}, Lemmas 3.3 and 3.4, compare also Lemma 2.3(2) in this paper). The assignment $F_k\mapsto A_k:=a_kH-\sum_{i=1}^N b_{ki}E_i$
is called a {\bf homological expression} of the symplectic configuration $D=\cup_{k=1}^nF_k$ with respect to the reduced basis $H,E_1,E_2,\cdots,E_N$. With this understood, one of the fundamental ideas of our method, the first one, is to study $D$ through its homological expressions. 

\begin{definition}
A vector of integer entries $\vec{v}:=(a, b_{1}, b_{2}, \cdots, b_{N})$ is called {\bf admissible} if the following conditions are satisfied: 
\begin{itemize}
\item [{(1)}] If $a>0$, then $b_i\geq 0$ for each $i=1,2,\cdots,N$.
\item [{(2)}] If $a\leq 0$, then exactly one of the $b_i$'s equals $-(|a|+1)$ and the rest are either $0$ or $1$.
\end{itemize}
\end{definition}

To proceed further, we denote the self-intersection of $F_k$ by $\nu_k$, the genus of $F_k$ by $g_k$, and the intersection number of $F_k,F_l$, where $k\neq l$, by $\nu_{kl}$. Then it is easy to see that for any homological expression $F_k\mapsto A_k:=a_kH-\sum_{i=1}^N b_{ki}E_i$ of $D$, the $n$-tuple of vectors $(\vec{v}_k)$,
where $\vec{v}_k=(a_k, b_{k1}, b_{k2}, \cdots, b_{kN})$, belongs to the set $\Omega(D)$ defined below 
in Definition 1.2.

\begin{definition}
For a given symplectic configuration $D=\cup_{k=1}^nF_k$, with $\nu_k$, $g_k$, $\nu_{kl}$ defined above, we denote by $\Omega(D)$ the set of $n$-tuples of vectors $(\vec{v}_k)$, where each $\vec{v}_k:=(a_k, b_{k1}, b_{k2}, \cdots, b_{kN})$ is admissible,  and the following equations are satisfied:
\begin{itemize}
\item [{(1)}] $a_k^2-\sum_{i=1}^N b_{ki}^2=\nu_k, \;\;\; k=1,2,\cdots,n$.
\item [{(2)}] $-3a_k+\sum_{i=1}^N b_{ki}=2g_k-2-\nu_k, \;\;\;  k=1,2,\cdots,n$.
\item [{(3)}] $a_ka_l-\sum_{i=1}^N b_{ki}b_{li}=\nu_{kl},\;\;\;  k\neq l, \; k,l=1,2,\cdots,n$.
\end{itemize}
Furthermore, for any $n$-tuple $\underline{C}=(C_k)$ of positive constants, we denote by 
$\Omega(D,\underline{C})$ the subset of $\Omega(D)$ which consists of those $(\vec{v}_k)$ such that the first entry $a_k$ in each $\vec{v}_k$ obeys $a_k\leq C_k$. 

Fixing an order of the components of $D$, i.e., $F_1, F_2, \cdots, F_n$,  we associate to each element 
$(\vec{v}_k)\in \Omega(D)$, where $\vec{v}_k=(a_k, b_{k1},b_{k2},\cdots, b_{kN})$, the following 
$n\times (N+1)$-matrix $\I_{(\vec{v}_k)}$, defined by 
$$
\I_{(\vec{v}_k)}=\left (\begin{array}{ccccc}
a_1 & -b_{11} & -b_{12} & -b_{13} & \cdots   -b_{1N} \\
a_2 & -b_{21} & -b_{22} & -b_{23} & \cdots  -b_{2N} \\
\cdots \\
a_k & -b_{k1} & -b_{k2} & -b_{k3} & \cdots  -b_{kN} \\
\cdots \\
a_n & -b_{n1} & -b_{n2} & -b_{n3} & \cdots -b_{nN} \\
\end{array}
\right ).
$$
We call $\I_{(\vec{v}_k)}$ the {\bf associated matrix} of $(\vec{v}_k)\in \Omega(D)$.
\end{definition}

The set $\Omega(D)$ admits some natural groups of symmetries, which come in three types:

\vspace{1mm}

{\bf Permutations of indices $1,2,\cdots,N$:} Let $\sigma\in S_N$, a permutation of indices $1,2,\cdots,N$. For any $(\vec{v}_k)\in\Omega(D)$ (resp. $\Omega(D,\underline{C})$), let $\vec{v}_k^\prime$ be the vector obtained 
from $\vec{v}_k$ by changing the $b_{ki}$-entries in $\vec{v}_k$ according to 
$\sigma$, then $(\vec{v}_k^\prime)\in \Omega(D)$ (resp. $\Omega(D,\underline{C})$). Note that in terms of the
associated matrix $\I_{(\vec{v}_k)}$, this amounts to a permutation of the last $N$ columns of $\I_{(\vec{v}_k)}$
according to $\sigma\in S_N$.

\vspace{1mm}

{\bf Automorphisms of $D$:} Let $\tau\in S_n$, a permutation of indices $1,2,\cdots,n$. Suppose $\tau$ induces an automorphism of $D$, i.e., the data $\{\nu_k, g_k, \nu_{kl}\}$ are preserved under $\tau$. Then for any $(\vec{v}_k)\in\Omega(D)$ (resp. $\Omega(D,\underline{C})$), $(\vec{v}_k^\prime:=\vec{v}_{\tau(k)})\in \Omega(D)$  
(resp. $\Omega(D,\tau(\underline{C}))$) as well. Here $\tau(\underline{C})=(C_{\tau(k)})$ for $\underline{C}=(C_k)$.
In terms of the associated matrix $\I_{(\vec{v}_k)}$, this amounts to a permutation of the rows of $\I_{(\vec{v}_k)}$
according to $\tau\in S_n$.

\vspace{1mm}

{\bf Automorphisms of $H^2(X)$:} The relevant automorphisms of $H^2(X)$ are those which preserve the intersection form on $H^2(X)$ and the canonical class $c_1(K_X)$. We are particularly interested in the automorphisms of $H^2(X)$ which are induced by an orientation-preserving diffeomorphism of $X$. According to \cite{LL} (cf. Theorem 3.1 in \cite{LL}), fixing any standard basis $H,E_1,E_2,\cdots,E_N$ (i.e., $H,E_1,E_2,\cdots,E_N$ has standard intersection matrix, and $c_1(K_X)=-3H+E_1+E_2+\cdots+E_N$, see Section 2), such an automorphism must be a product of reflections along $(-2)$-classes of the form 
$\gamma=E_i-E_j$ or $\gamma=H-E_i-E_j-E_k$. 
Recall that the reflection $R(\gamma)$ along a $(-2)$-class $\gamma$ is defined as follows:
$$
R(\gamma)(A)=A+(\gamma\cdot A)\gamma, \;\; \forall A\in H^2(X). 
$$
In particular, $R(E_i-E_j)$ is simply switching the classes $E_i,E_j$. 

With the preceding understood, the action of $R(\gamma)$ on the set $\Omega(D)$ is defined as follows:
given any $(\vec{v}_k)\in \Omega(D)$, we identify $\vec{v}_k$ with the class
$A_k:=a_kH-\sum_{i=1}^N b_{ki}E_i$ and set $A_k^\prime:=R(\gamma)(A_k)$. Let 
$\vec{v}_k^\prime:=(a_k^\prime, b_{k1}^\prime, b_{k2}^\prime, \cdots, b_{kN}^\prime)$ be the vector 
formed from the coefficients of $A_k^\prime$. Then we define $R(\gamma)(\vec{v}_k)=(\vec{v}_k^\prime)$. 
With this understood, it is easy to see that, if $\gamma=E_i-E_j$,  $R(\gamma)(\vec{v}_k)\in \Omega(D)$ for any 
$(\vec{v}_k)\in \Omega(D)$, as $R(\gamma)$ simply switches the indices $i,j$. 

On the other hand, note that $R(\gamma)$, for $\gamma=H-E_i-E_j-E_k$, may not preserve the admissibility of 
the vectors $\vec{v}_k$. However, when it does preserve the admissibility of each $\vec{v}_k$, then 
$R(\gamma)(\vec{v}_k)\in \Omega(D)$, as the equations (1)-(3) in Definition 1.2 are always satisfied by
$R(\gamma)(\vec{v}_k)$. Finally, note that even if $R(\gamma)(\vec{v}_k)\in \Omega(D)$, 
$(\vec{v}_k)\in \Omega(D,\underline{C})$ does not imply that $R(\gamma)(\vec{v}_k)\in \Omega(D,\underline{C})$, 
because the constraints $a_k\leq C_k$ may not be preserved under $R(\gamma)$. Finally, we should point out 
that the reflections $R(\gamma)$, where $\gamma=H-E_i-E_j-E_k$, are closely related to the so-called quadratic Cremona transformations in algebraic geometry  (cf. \cite{AC}). So in some sense, the correspondence 
$(\vec{v}_k)\mapsto R(\gamma)(\vec{v}_k)$ defines an ``algebraic Cremona equivalence", which, under some further assumptions, can be enhanced to a symplectic Cremona transformation, see Theorem 1.8. 

\vspace{1mm}

From an enumerative point of view, we shall work with the set of orbits of $\Omega(D)$ under the actions of permutations of the indices $1,2,\cdots,N$. We denote the set of orbits of $\Omega(D)$ by $\hat{\Omega}(D)$, 
and the corresponding orbit set of $\Omega(D,\underline{C})$ by $\hat{\Omega}(D,\underline{C})$. Note that 
when we turn an element $(\vec{v}_k)\in \Omega(D)$ into the corresponding homological expression of $D$, i.e., $F_k\mapsto A_k:=a_kH-\sum_{i=1}^N b_{ki}E_i$ where $\vec{v}_k=(a_k, b_{k1},b_{k2},\cdots, b_{kN})$, the homological expression depends only on the orbit of $(\vec{v}_k)$ in $\hat{\Omega}(D)$ as the classes $E_i$ are naturally ordered for a reduced basis. We shall call an element of $\hat{\Omega}(D)$ (or for simplicity a representative $(\vec{v}_k)$ of it) a {\bf homological assignment} of $D$. From a computational point of view, 
we obtain the set of all possible homological expressions of $D$ through the set of homological assignments. 

\vspace{1mm}

As we shall study $D$ through its homological expressions, the first fundamental question is whether the set
$\Omega(D)$ is always finite. It turns out that in general, $\Omega(D)$ is not finite. However, for any 
$\underline{C}$, the set $\Omega(D,\underline{C})$ is always finite (cf. Lemma 2.4), and moreover, when 
$X=\C\P^2\# N\overline{\C\P^2}$ for some $N\leq 8$, $\Omega(D)= \Omega(D,\underline{C})$ for some 
$\underline{C}$ which depends only on $D$; in particular, $\Omega(D)$ is finite if $N\leq 8$ 
(cf. Proposition 2.5). We state the theorem below for the case where $D$ consists of a single surface, which may be of independent interest.

\begin{theorem}
Let $X=\C\P^2\# N\overline{\C\P^2}$ for some $N\leq 8$, and let $\omega$ be a symplectic structure on $X$. 
Fixing any integer $\alpha$ and any non-negative integer $g$, the number of classes $A\in H^2(X)$ which can be realized by an embedded symplectic surface in $(X,\omega)$, with genus $g$ and self-intersection $-\alpha$, 
is finite, bounded from above by a constant depending only on $\alpha$ and $g$. Moreover, if $N=9$ but
$\alpha+2g-2>0$, then the conclusion continues to hold. 
\end{theorem}

Before dealing with the issue of finiteness of the homological assignments of $D$ for the case where $N\geq 9$, we shall first discuss the second fundamental idea of the method, i.e., the freedom of specifying the symplectic areas of the components $F_k$ of the symplectic configuration.  To this end, we need to impose the following additional condition on $D$: 
\begin{itemize}
\item [{(\ddag)}] Introduce the $n\times n$ matrix $Q:=(\nu_{kl})$, $\nu_{kl}:=F_k\cdot F_l$,
where an order for the components $F_k$ is being fixed. Then either $Q$ is negative definite, or $D$ is connected and $Q$ is non-singular and non-negative definite. 
\end{itemize}

Under $(\ddag)$, we shall define a cone $C_\delta$ in $\R^n$ as follows. First, we shall adapt the following notation: for any vector $\vec{x}=(x_1,x_2,\cdots,x_n)^T$, we will write $\vec{x}\geq 0$ (resp. $\vec{x}>0$) if $x_k\geq 0$ 
(resp. $x_k>0$) for any $k=1,2,\cdots,n$. With this understood, we have the following definition for $C_\delta$:

\begin{itemize}
\item if $Q$ is negative definite, then $C_\delta=\{\vec{\delta}\in\R^n|\vec{\delta}\geq 0\}$, 
\item if $D$ is connected and $Q$ is non-singular and non-negative definite, then
$$
C_\delta=\{\vec{\delta}\in\R^n|\vec{\delta}\geq 0 \mbox{ and } Q^{-1}\vec{\delta}\geq 0\}.
$$
\end{itemize}
We remark that the cone $C_\delta$ is invariant under the automorphisms of $D$, as $Q=(\nu_{kl})$ is invariant 
under an automorphism of $D$.

\begin{definition}
Let $\vec{\delta}=(\delta_k)$ be any interior point in the cone $C_\delta$. We denote by $Z(\vec{\delta})$ the set of symplectic structures $\omega$ on $X$ which have the following properties: 
\begin{itemize}
\item $c_1(K_X)$ is the canonical class of $\omega$.
\item $D$ is symplectic with respect to $\omega$.
\item $\omega(F_k)=\delta_k$ for $k=1,2,\cdots,n$. 
\end{itemize}
\end{definition}

With Definition 1.4 understood, what we mean by the freedom of specifying the symplectic areas of the components 
$F_k$ is that $Z(\vec{\delta})\neq \emptyset$ for any interior point $\vec{\delta}$ of the cone $C_\delta$, if 
$D$ indeed is a symplectic configuration in $X$ (see Lemma 2.10 and \cite{C}, Lemma 4.1). In other words, for any interior point $\vec{\delta}$ in $C_\delta$, there is a symplectic structure $\omega$ such that $D$ is symplectic with respect to $\omega$ and the $\omega$-areas of the components $F_k$ are given by the entries of $\vec{\delta}$,
with $c_1(K_X)$ being the canonical class. 

\vspace{2mm}

With the preceding understood, let $H,E_1,E_2,\cdots,E_N$ be any reduced basis of $(X,\omega)$, where 
$\omega\in Z(\vec{\delta})$, and let $F_k\mapsto A_k:=a_kH-\sum_{i=1}^N b_{ki}E_i$ be the corresponding homological expression of $D$. Set $\vec{v}_k=(a_k, b_{k1}, b_{k2}, \cdots, b_{kN})$ for $k=1,2,\cdots,n$. 
Then as we pointed out earlier, 
$(\vec{v}_k)\in \Omega(D)$. We shall say that the element $(\vec{v}_k)\in \Omega(D)$ is {\bf realized under 
$\vec{\delta}$}. For any $(\vec{v}_k)\in \Omega(D)$, if $(\vec{v}_k)$ is not realized under $\vec{\delta}$, we shall say that $(\vec{v}_k)$ {\bf can be eliminated by $\vec{\delta}$}. It is clear that if there exists an interior point 
$\vec{\delta}$ of the cone $C_\delta$, such that every element $(\vec{v}_k)\in \Omega(D)$ can be eliminated 
by $\vec{\delta}$, then we have shown that the symplectic configuration $D$ cannot exist in $X$.

It turns out that there is a very simple criterion for determining whether a given element $(\vec{v}_k)\in \Omega(D)$ 
can be eliminated by a given $\vec{\delta}$ or not. To explain this, we first recall the constraints on the symplectic areas of a reduced basis $H,E_1,E_2,\cdots,E_N$ of $(X,\omega)$. Let $\lambda_0:=\omega(H)$, 
$\lambda_i:=\omega(E_i)$, where $i=1,2,\cdots,N$, denote the areas of the elements of the reduced basis
$H,E_1,E_2,\cdots,E_N$. Then $(\lambda_0,\lambda_1,\lambda_2, \cdots,\lambda_N)$ satisfies the following conditions (i)-(iii) (cf. \cite{C}): 
\begin{itemize}
\item [{(i)}] $\lambda_i\geq \lambda_j>0$ for any $0<i<j$, where $i,j=1,2,\cdots,N$.
\item [{(ii)}] $\lambda_0\geq \lambda_i+\lambda_j+\lambda_k>0$
for any distinct $i,j,k$, where $i,j,k=1,2,\cdots,N$. 
\item [{(iii)}] $\lambda_0^2-\sum_{i=1}^N \lambda_i^2>0$. 
\end{itemize}

With this understood, we shall consider another cone $C_{\lambda}$, which is a cone in $\R^{N+1}$ defined as follows: let $\vec{\lambda}=(\lambda_0,\lambda_1,\lambda_2,\cdots,\lambda_N)^T\in \R^{N+1}$, then
$$
C_{\lambda}:=\{\vec{\lambda}\in \R^{N+1}| \vec{\lambda}\geq 0, \lambda_0-\lambda_i-\lambda_j-\lambda_k\geq 0,  \mbox{ where $i,j,k$ are distinct}\}.
$$
Now suppose $(\vec{v}_k)\in\Omega(D)$ is realized under $\vec{\delta}$, and let $\omega\in Z(\vec{\delta})$ be
the corresponding symplectic structure and $H,E_1,E_2,\cdots,E_N$ be the reduced basis such that $(\vec{v}_k)$
gives the corresponding homological expression of $D$ with respect to $H,E_1,E_2,\cdots,E_N$. 
Let $\lambda_0:=\omega(H)$, $\lambda_i:=\omega(E_i)$ for $i=1,2,\cdots,N$, and let $\I_{(\vec{v}_k)}$ be the associated matrix of $(\vec{v}_k)$. Setting $\vec{\lambda}:=(\lambda_0,\lambda_1,\cdots,\lambda_N)^T\in \R^{N+1}$, it follows easily that 
$$
\I_{(\vec{v}_k)}\vec{\lambda}=\vec{\delta}, \mbox{ where } 
\vec{\lambda}\in C_\lambda, \vec{\lambda}>0, \mbox{ and } \lambda_0^2-\sum_{i=1}^N \lambda_i^2>0.
$$
We remark that there is another constraint on the vector $\vec{\lambda}$, i.e., the condition (i): 
$\lambda_i\geq \lambda_j$ for $0<i<j$. This constraint, on the one hand, is not as convenient 
because it is not invariant under the permutations of the indices $1,2, \cdots,N$. On the other hand,
for any $\vec{\lambda}\in C_\lambda$, condition (i) is always satisfied up to a permutation of 
$1,2, \cdots,N$. This is the reason why we do not impose it in the definition of $C_\lambda$. 

With the preceding understood, we now state the criterion which determines whether a given element 
$(\vec{v}_k)\in \Omega(D)$ can be eliminated by a given $\vec{\delta}\in C_\delta$. 
Given any $(\vec{v}_k)\in\Omega(D)$, let $\I_{(\vec{v}_k)}$ be the associated matrix. We would like to find a set of area vectors $\vec{\delta}\in C_\delta$ which can be used to eliminate $(\vec{v}_k)$. Obviously, if we set 
$$\Delta((\vec{v}_k)):=\text{ closure }(C_\delta\setminus \I_{(\vec{v}_k)}(C_\lambda)),$$ 
then for any interior point $\vec{\delta}\in \Delta((\vec{v}_k))$,
$(\vec{v}_k)$ and any element of $\Omega(D)$ which equals $(\vec{v}_k)$ up to a permutation of indices 
$1,2, \cdots,N$ can be eliminated by $\vec{\delta}$. Furthermore, let $Aut(D)\subseteq S_n$ be the subgroup of 
permutations of the indices $1,2,\cdots,n$ which preserves the configuration $D$. Note that 
$Aut(D)$ acts on the cone $C_\delta$ by permuting the components of the vectors $\vec{\delta}\in C_\delta$. 
With this understood, if $\vec{\delta}\in \bigcap_{\tau\in Aut(D)} \tau(\Delta((\vec{v}_k)))$ is an interior point, then
$(\vec{v}_k)$ and any element of $\Omega(D)$ which equals $(\vec{v}_k)$ up to a permutation of indices 
$1,2, \cdots,N$ or by the action of an element of $Aut(D)$ can be eliminated by $\vec{\delta}$. 

We remark that if $\vec{\delta}\in \Delta((\vec{v}_k))$ (or $\bigcap_{\tau\in Aut(D)} \tau(\Delta((\vec{v}_k)))$) 
is not an interior point, but only an interior point of $C_\delta$, choosing $\vec{\delta}$ to be the areas of the surfaces $F_k$ may still kill $(\vec{v}_k)\in\Omega(D)$ and the elements of $\Omega(D)$ equivalent to it.
The point is that for such a $\vec{\delta}$, even though the vectors 
$\vec{\lambda}\in \I_{(\vec{v}_k)}^{-1}(\vec{\delta})$ may lie in the cone $C_\lambda$ (i.e., on a face of $C_\lambda$), the inequality $\lambda_0^2-\sum_{i=1}^N \lambda_i^2>0$ may fail so that $\vec{\lambda}$ 
still cannot be the area vector from a reduced basis. On the other hand, we should point out that from a computational point of view, for a given $(\vec{v}_k)\in\Omega(D)$, describing the set of $\Delta((\vec{v}_k))$ or 
$\bigcap_{\tau\in Aut(D)} \tau(\Delta((\vec{v}_k)))$ could be a challenging problem combinatorially. So in
practice, we shall rely on computer programming (see \cite{CGM}) to select an ``optimal" choice of 
$\vec{\delta}\in C_\delta$ (e.g. guided by Principle 1.7). We summarize it in the following 

\begin{criterion}
For any $(\vec{v}_k)\in \Omega(D)$, if $\vec{\delta}\in \Delta((\vec{v}_k))$ is an interior point, then $(\vec{v}_k)$ and any element of $\Omega(D)$ which equals $(\vec{v}_k)$ up to a permutation of indices $1,2, \cdots,N$ can be eliminated by $\vec{\delta}$. Furthermore, if $\vec{\delta}\in \bigcap_{\tau\in Aut(D)} \tau(\Delta((\vec{v}_k)))$ is an interior point, then $(\vec{v}_k)$ and any element of $\Omega(D)$ which equals $(\vec{v}_k)$ up to a permutation of indices $1,2, \cdots,N$ or by the action of an element of $Aut(D)$ can be eliminated by $\vec{\delta}$. 
\end{criterion}

Now we return to the issue of finiteness of homological assignments. For the case of $N\geq 9$, we shall impose the following additional assumption on the configuration $D$: 
\begin{itemize}
\item [{(*)}] $c_1(K_X)$ is supported by $D$, i.e., there exist $c_1,c_2,\cdots,c_n\in\Q$, 
such that $c_1(K_X)=\sum_{k=1}^n c_k F_k$. Moreover, for any $k$, if $c_k\geq 0$, then $F_k$ is a
$(-\alpha)$-sphere for some $\alpha=0,1,2$ or $3$.
\end{itemize}

For convenience we introduce the following subsets of indices of $k$:
$$
I_0=\{k|c_k\geq 0 \mbox{ in condition (*)}\}, \;\; I_1=\{k|\mbox{$F_k$ is a
$(-\alpha)$-sphere for $\alpha=0,1,2$ or $3$}\}.
$$
Observe that $I_0\subseteq I_1$. With this understood, fixing any subset $I^\ast$ of indices such that
$I^\ast\subseteq I_1$, we introduce the following two cones $C^\ast_0, C^\ast_1$ in $\R^n$:
$$
C^\ast_0:=\{\vec{\delta}\in \R^n| \delta_k \leq -\sum_{l=1}^n c_l\delta_l, \; \forall k\in I_0\} 
\mbox{ and }
C^\ast_1:=\{\vec{\delta}\in \R^n| 2 \delta_k \leq -\sum_{l=1}^n c_l\delta_l, \; \forall k\in I^\ast\}.
$$
(We remark that for all the symplectic configurations we encountered in the study of symplectic Calabi-Yau 
$4$-manifolds in \cite{C,C1}, the condition (*) is satisfied, with the interior of the cone 
$C^\ast_0 \cap C_\delta$ nonempty.)

We have the following theorem, which is proved in Section 2.

\begin{theorem}
Under the additional assumptions (\ddag) and (*), there exists a $\underline{C}=(C_k)$, where the constants 
$C_k$ can be explicitly determined from $D$ and the coefficients $c_1,c_2,\cdots,c_n$ in the assumption (*), such that for any interior point $\vec{\delta}\in C^\ast_0 \cap C_\delta$, if an element 
$(\vec{v}_k)\in \Omega(D)$ is realized under $\vec{\delta}$, then $(\vec{v}_k)\in \Omega(D,\underline{C})$.
Moreover, if we fix a subset $I^\ast$ of indices such that $I^\ast\subseteq I_1$, and choose an interior point 
$\vec{\delta}\in C^\ast_1 \cap C^\ast_0\cap C_\delta$, then the constant $C_k$ in $\underline{C}=(C_k)$ can be taken to be $3$ for any $k\in I^\ast$.
\end{theorem}

In other words, by Theorem 1.6, under the assumptions (\ddag) and (*) and assuming the interior of the cone 
$C^\ast_0 \cap C_\delta$ is nonempty, if we choose an interior point $\vec{\delta}\in C^\ast_0 \cap C_\delta$ for the areas of the components $F_k$ of $D$, any element of $\Omega(D)$ in the complement of the finite set 
$\Omega(D,\underline{C})$ can be eliminated. 

Consequently, under the assumptions (\ddag) and (*), it suffices to consider only the elements of the finite set 
$\Omega(D,\underline{C})$ as long as we choose an interior point $\vec{\delta}\in C^\ast_0 \cap C_\delta$ for the areas of the $F_k$'s. Since $\Omega(D,\underline{C})$ is finite, it is possible to give an enumeration of the elements of the corresponding set $\hat{\Omega}(D,\underline{C})$ of homological assignments of $D$ via a computer search (see \cite{CGM}). With this understood, it is desirable to choose an interior point 
$\vec{\delta}\in C^\ast_0 \cap C_\delta$ according to the following principle, where we denote by 
$\hat{\Omega}(D,\underline{C},\vec{\delta})$ the subset of $\hat{\Omega}(D,\underline{C})$ consisting of the elements which cannot be eliminated by  $\vec{\delta}$.
(Our experience in \cite{C} shows that $\hat{\Omega}(D,\underline{C},\vec{\delta})$ can be quite sensitive to
the choice of $\vec{\delta}$.)

\begin{principle}
Assume the interior of $C^\ast_0 \cap C_\delta$ is nonempty. Choose an interior point 
$\vec{\delta}\in C^\ast_0 \cap C_\delta$ such that either $\hat{\Omega}(D,\underline{C},\vec{\delta})=\emptyset$, 
or the following are true:
\begin{itemize}
\item [{(i)}] $\hat{\Omega}(D,\underline{C},\vec{\delta})$ has a very small number of elements.
\item [{(ii)}] For each $(\vec{v}_k)\in\Omega(D,\underline{C},\vec{\delta})$, the entry $a_k$ in $\vec{v}_k$ for each
$k=1,2,\cdots,n$ is non-negative and takes very small values relative to the genus $g_k$ of $F_k$, e.g. $0\leq a_k\leq 3$ if $g_k=0$.
\item [{(iii)}] For each $(\vec{v}_k)\in\hat{\Omega}(D,\underline{C},\vec{\delta})$, with respect to the homological expression of $D$ corresponding to $(\vec{v}_k)$, the successive symplectic blowing-down procedure introduced in \cite{C1} can be carried through to the final stage of $\C\P^2$. As a consequence, for each $(\vec{v}_k)\in\hat{\Omega}(D,\underline{C},\vec{\delta})$, the configuration $D$ is transformed under the successive 
symplectic blowing-down to a symplectic arrangement $\hat{D}$ in $\C\P^2$ whose combinatorial type is 
completely determined by $(\vec{v}_k)$. 
\end{itemize}
\end{principle}

It is clear that if there is a $\vec{\delta}$ such that $\hat{\Omega}(D,\underline{C},\vec{\delta})=\emptyset$, then
the symplectic configuration $D$ cannot exist in $X$. In general, when 
$\hat{\Omega}(D,\underline{C},\vec{\delta})\neq \emptyset$ for any interior point 
$\vec{\delta}\in C^\ast_0 \cap C_\delta$,
if we can find a $\vec{\delta}$ according to Principle 1.7, then the study of $D$ is reduced to the problem of understanding the symplectic arrangements in $\C\P^2$ which correspond to the elements in
$\hat{\Omega}(D,\underline{C},\vec{\delta})$. For example, if we can show that the symplectic arrangement 
$\hat{D}$ in $\C\P^2$ which corresponds to $(\vec{v}_k)\in \hat{\Omega}(D,\underline{C},\vec{\delta})$ 
cannot exist, then the element $(\vec{v}_k)$ is further eliminated. On the other hand, if for some element 
$(\vec{v}_k)\in\hat{\Omega}(D,\underline{C},\vec{\delta})$, the corresponding symplectic arrangement $\hat{D}$ 
in $\C\P^2$ can be realized, then by a successive symplectic blowing-up operation in the reversed order 
(cf. Lemma 3.1), one obtains a symplectic embedding of $D$ in $X$, proving the existence. 

With this understood, at this last stage of the method, a central problem is to try to prove that the symplectic arrangements in $\C\P^2$, which correspond to the elements of $\hat{\Omega}(D,\underline{C},\vec{\delta})$ and cannot be eliminated by other means, can be deformed to a complex arrangement in $\C\P^2$
with the same combinatorial type.  A positive solution would have the following implications: If the complex
arrangement is known to not exist (e.g., by results from algebraic geometry), the corresponding symplectic arrangement also cannot exist, therefore the corresponding element in $\hat{\Omega}(D,\underline{C},\vec{\delta})$ can be eliminated. On the other hand, if every symplectic arrangement under consideration can be deformed to a complex arrangement in $\C\P^2$, and some of the complex arrangements do exist, then the symplectic embedding of $D$ in $X$, which exists, is smoothly equivalent to a holomorphic embedding. We remark that for deformation of a symplectic arrangement in $\C\P^2$ to a complex arrangement, one relies on Gromov's theory of pseudoholomorphic curves. See Section 4 for more details. 

\vspace{2mm}

\subsection{Cremona transformations in a symplectic setting and an application} 
After addressing these theoretical issues, we now discuss the main technical theorem of this paper, Theorem 1.8, where a symplectic analog of Cremona transformations from algebraic geometry (cf. \cite{AC}) is 
established. Moreover, for an illustration of how this technique is used in combination with Gromov's theory
of pseudoholomorphic curves and results from algebraic geometry, we give a new proof that a certain symplectic line arrangement, called Fano planes (cf. \cite{RuS}), cannot exists in $\C\P^2$. 

Recall that for any $(-2)$-class $\gamma=H-E_r-E_s-E_t$, where $H,E_1,E_2,\cdots, E_N$ is a standard basis of $H^2(X)$, the reflection $R(\gamma)$ acts on the set $\Omega(D)$ as long as admissibility is preserved, i.e., for any $(\vec{v}_k)\in \Omega(D)$, $(\vec{v}_k^\prime):=R(\gamma)(\vec{v}_k)\in \Omega(D)$ if and only if each $\vec{v}_k^\prime$ is admissible. The reflections $R(\gamma)$, where $\gamma=H-E_r-E_s-E_t$, are closely related to the so-called quadratic Cremona transformations. So along the way, we will also obtain certain conditions under which the reflection $R(\gamma)$ preserves the admissibility of an element $(\vec{v}_k)\in \Omega(D)$ (see Lemma 3.8).

The construction of a symplectic analog of quadratic Cremona transformations requires an extension of the notion of homological expression of $D$ to a virtual setting. To be more precise, recall that in a homological expression
$F_k\mapsto A_k:=a_k H-\sum_{i=1}^N b_{ki} E_i$, the basis $H,E_1,E_2,\cdots,E_n$ is required to be a reduced basis. This condition allows us to successively blow down the classes $E_N,E_{N-1}, \cdots, E_1$, as they can be successively represented by a symplectic $(-1)$-sphere at each stage. Furthermore, in order to ensure the successive blowing-down operation is reversible, certain assumptions which are labelled as (a) and (b) (see Section 3 for more details) are imposed on the homological expression $F_k\mapsto A_k:=a_k H-\sum_{i=1}^N b_{ki} E_i$. Under the successive blowing-down procedure, the configuration $D$ is transformed to a symplectic arrangement 
$\hat{D}$ in $\C\P^2$, which is a union of pseudoholomorphic curves whose singularities and intersection pattern are completely determined by the element $(\vec{v}_k)\in \Omega(D)$, 
where $\vec{v}_k:=(a_k,b_{k1},b_{k2},\cdots,b_{kN})$. 
The type of the singularities and the intersection pattern of the components of $\hat{D}$ together form part of 
the so-called {\bf combinatorial type} of $\hat{D}$. Now the key observation is that the description of the combinatorial type of 
$\hat{D}$ only requires a partial order on the set of $E_i$-classes $E_1,E_2,\cdots,E_N$, which is analogous to the partial order defined by the relation of ``infinitely near" in algebraic geometry (cf. \cite{Bea}), and this partial order
on the set $E_1,E_2,\cdots,E_N$ is completely determined by $(\vec{v}_k)$ as well. With this understood, roughly speaking, if we drop the requirement of the basis $H,E_1,E_2,\cdots, E_N$ being a 
reduced basis in a homological expression of $D$, we get the notion of a {\bf virtual homological expression} of $D$ 
(see Definition 3.5 for a precise explanation). In particular, a virtual homological expression of $D$ also determines a partial order on the set $E_1,E_2,\cdots,E_N$, as well as a {\bf virtual combinatorial type} (see Lemma 3.6). A virtual homological expression of $D$ is said to be {\bf realizable} if its virtual combinatorial type is the combinatorial type of a symplectic arrangement in $\C\P^2$ (cf. Definition 3.7). 

With the preceding understood, we now state the relevant theorem. Let $(\vec{v}_k)\in\Omega(D)$ be an element which is realized by a symplectic structure $\omega\in Z(\vec{\delta})$, such that the successive blowing-down procedure associated to the corresponding homological expression of $D$ can be performed to the final stage 
of $\C\P^2$, resulting in a symplectic arrangement $\hat{D}$ in $\C\P^2$. Let $H,E_1,E_2,\cdots,E_N$ be the 
reduced basis, with respect to which the $a$, $b_i$-coefficients of the class of $F_k$ are given by the entries 
in the vector $\vec{v}_k$. As we mentioned earlier, there is a partial order $\leq$ of infinitely-nearness on the set 
$E_1,E_2,\cdots,E_N$, which depends only on $(\vec{v}_k)$. We mention that since a minimal element $E_i$
under the partial order $\leq$ is always the last to be blown-down, there is a point denoted by $\hat{E}_i$ in
$\C\P^2$ assigned to the minimal class $E_i$ (see Section 3 for more details). Finally, the combinatorial type of $\hat{D}$ also depends only on $(\vec{v}_k)$. 

Let $E_r, E_s,E_t$ be three distinct $E_i$-classes, and let $\gamma:=H-E_r-E_s-E_t$. Set 
$(\vec{v}_k^\prime):=R(\gamma)(\vec{v}_k)$. Then observe that if we let $H^\prime, E_1^\prime,E_2^\prime,\cdots E_N^\prime$ be the image of $H,E_1,E_2,\cdots,E_N$ under the reflection $R(\gamma)$, and write 
$\vec{v}_k^\prime=(a_k^\prime, b_{k1}^\prime,\cdots,b_{kN}^\prime)$, then 
$$
a_kH-\sum_{i=1}^N b_{ki} E_i=a_k^\prime H^\prime-\sum_{i=1}^N b_{ki}^\prime E_i^\prime.
$$
In particular, $\vec{v}_k^\prime$ encodes the coefficients of the class of $F_k$ with respect to the basis $H^\prime, E_1^\prime,E_2^\prime,\cdots E_N^\prime$, which is only a standard basis (see Section 2 for a definition).

\begin{theorem}
Assume the components of $\hat{D}$ is $\hat{J}$-holomorphic where $\hat{J}$ is a compatible almost complex structure on $\C\P^2$. Furthermore, assume $E_r,E_s,E_t$ satisfy one of the following conditions:
\begin{itemize}
\item [{(1)}] $E_r,E_s,E_t$ are minimal with respect to the partial order $\leq$, and the points 
$\hat{E}_r,\hat{E}_s,\hat{E}_t\in \C\P^2$ are not contained in a degree $1$ $\hat{J}$-holomorphic sphere. 
\item [{(2)}] $E_r,E_s$ are minimal, $E_t$ is infinitely near to $E_s$ of order $1$, such that the
point $\hat{E}_t$ is not contained in the proper transform of the degree $1$ $\hat{J}$-holomorphic sphere passing through $\hat{E}_r,\hat{E}_s$.
\item [{(3)}] $E_r$ is minimal, $E_s$ is infinitely near to $E_r$ of order $1$, $E_t$ 
is infinitely near to $E_s$ of order $1$, and $E_t$ is not a satellite class (cf. Section 3). 
\end{itemize}
Then the assignment $F_k\mapsto a_k^\prime H^\prime-\sum_{i=1}^N b_{ki}^\prime E_i^\prime$ is a virtual
homological expression of $D$. Moreover, if the assumptions {\em (a)}, {\em (b)} are satisfied by the virtual homological expression, then there is a symplectic arrangement $\hat{D}^\prime$ in $\C\P^2$ which realizes the virtual combinatorial type of the virtual homological expression $F_k\mapsto a_k^\prime H^\prime-\sum_{i=1}^N b_{ki}^\prime E_i^\prime$. Note that in particular, if $\hat{D}^\prime$ does not exist, neither does $\hat{D}$, and as a
consequence, the element $(\vec{v}_k)\in\Omega(D)$ can be eliminated. 
\end{theorem}

We remark that $(\vec{v}_k^\prime)$ is not necessarily realized by some $\omega\in Z(\vec{\delta})$ for any 
$\vec{\delta}$, and the virtual homological expression  
$F_k\mapsto a_k^\prime H^\prime-\sum_{i=1}^N b_{ki}^\prime E_i^\prime$ is not necessarily a homological expression of $D$. Furthermore, the symplectic arrangement $\hat{D}^\prime$ in $\C\P^2$ is not necessarily 
resulted from a successive blowing-down associated to a homological expression of $D$.

In comparison if this were in the algebraic geometry setting, then $D$ would be a complex configuration in a 
rational algebraic surface $X=\C\P^2 \# N \overline{\C\P^2}$, where the reflection $R(\gamma)$ is associated with 
a quadratic Cremona transformation $\Psi: \C\P^2\dashrightarrow \C\P^2$ (a birational automorphism of $\C\P^2$), and $\hat{D}$ is a complex arrangement in $\C\P^2$ which is the direct image of a successive blowing down from
$X$ to $\C\P^2$, i.e., a birational morphism $\pi: X\rightarrow \C\P^2$. Moreover, there is a birational
morphism $\pi^\prime: X\rightarrow \C\P^2$ such that the Cremona transformation 
$\Psi=\pi^\prime \circ \pi^{-1}$, and the complex arrangement 
$\hat{D}^\prime$ is simply the direct image of $D\subset X$ under $\pi^\prime: X\rightarrow \C\P^2$. With this 
understood, even though in Theorem 1.8 we did not attempt to establish any analog of the Cremona 
map $\Psi$ in the symplectic setting, we were able to show the existence of a symplectic arrangement 
$\hat{D}^\prime$ realizing the virtual combinatorial type resulted from the reflection $R(\gamma)$,
i.e., a correspondence $\Psi_D: \hat{D}\mapsto \hat{D}^\prime$, which is determined by $R(\gamma)$ and the homological expression of $D$. In many situations, this is good enough for applications. 
A proof of Theorem 1.8 is given in Section 3.4.

\begin{example}
(1) Consider a symplectic line arrangement $\hat{D}_1$ in $\C\P^2$, called a Fano plane (cf. \cite{RuS}), which consists of $7$ degree $1$ symplectic spheres $\hat{F}_1, \hat{F}_2,\cdots,\hat{F}_7$ with $7$ triple intersection points $p_1,p_2,\cdots,p_7$, where each $\hat{F}_k$ is $\hat{J}$-holomorphic for some compatible almost complex structure $\hat{J}$. Without loss of generality, we assume the following is the intersection pattern of the $7$ spheres:
\begin{itemize}
\item $\hat{F}_1\cap \hat{F}_2\cap \hat{F}_3=\{p_1\}$, $\hat{F}_1\cap \hat{F}_4\cap \hat{F}_6=\{p_2\}$, 
$\hat{F}_1\cap \hat{F}_5\cap \hat{F}_7=\{p_3\}$, $\hat{F}_2\cap \hat{F}_4\cap \hat{F}_7=\{p_4\}$, 
\item $\hat{F}_2\cap \hat{F}_5\cap \hat{F}_6=\{p_5\}$, 
$\hat{F}_3\cap \hat{F}_4\cap \hat{F}_5=\{p_6\}$, $\hat{F}_3\cap \hat{F}_6\cap \hat{F}_7=\{p_7\}$.
\end{itemize}
We apply Lemma 3.1 (see Section 3.1) to blow up at $p_1,p_2,\cdots,p_7$, and let $E_1,E_2,\cdots,E_7$ be the exceptional $(-1)$-spheres, which has an area $\epsilon$ for a sufficiently small $\epsilon>0$. Let $F_1,F_2,\cdots,F_7$ be the proper transforms of  $\hat{F}_1, \hat{F}_2,\cdots,\hat{F}_7$ in $\C\P^2\#7\overline{\C\P^2}$, which is a disjoint union of $7$ symplectic $(-2)$-spheres, a symplectic configuration we denote by $D_1$. It follows easily that when 
$\epsilon$ is chosen sufficiently small, $H, E_1,E_2,\cdots,E_7$ is a reduced basis of $\C\P^2\#7\overline{\C\P^2}$. With this understood, the following is the corresponding homological expression of $D_1$, $F_k\mapsto A_k$, where
\begin{itemize}
\item $A_1=H-E_{1}-E_{2}-E_{3}$, $A_2=H-E_{1}-E_{4}-E_{5}$, $A_3=H-E_{1}-E_{6}-E_{7}$, 
\item $A_4=H-E_{2}-E_{4}-E_{6}$, $A_5=H-E_{3}-E_{5}-E_{6}$, $A_6=H-E_{2}-E_{5}-E_{7}$, 
\item $A_7=H-E_{3}-E_{4}-E_{7}$.
\end{itemize}
Now we pick a point $p_8\in \C\P^2$ such that $p_6,p_7,p_8$ are not lying in a degree $1$
$\hat{J}$-holomorphic sphere. We blow up at $p_8$ and let $E_8$ be the exceptional $(-1)$-sphere, which also
has area $\epsilon$. Then $H, E_1,E_2,\cdots,E_8$ is a reduced basis of $\C\P^2\#8\overline{\C\P^2}$.
We consider the $(-2)$-class $\gamma=H-E_6-E_7-E_8$, where we note that $E_6,E_7,E_8$ satisfy (1) of Theorem 1.8, as all the classes $E_1,E_2,\cdots,E_8$ are minimal with respect to the partial order of infinitely-nearness
in this case. If we let $H^\prime=R(\gamma)(H)$, and $E_i^\prime=R(\gamma)(E_i)$, $i=1,2,\cdots,8$, then by Theorem 1.8, we obtain the following virtual homological expression of $D_1$, i.e., $F_k\mapsto A_k^\prime$, where
\begin{itemize}
\item $A_1^\prime=2H^\prime-E_{1}^\prime-E_{2}^\prime-E_{3}^\prime-E_6^\prime-E_7^\prime-E_8^\prime$, 
\item $A_2^\prime=2H^\prime-E_1^\prime-E_{4}^\prime-E_{5}^\prime-E_6^\prime-E_7^\prime-E_8^\prime$, 
\item $A_3^\prime=E_8^\prime-E_{1}^\prime$, $A_4^\prime=H^\prime-E_{2}^\prime-E_{4}^\prime-E_{6}^\prime$, 
$A_5^\prime=H^\prime-E_{3}^\prime-E_{5}^\prime-E_{6}^\prime$, 
\item $A_6^\prime=H^\prime-E_{2}^\prime-E_{5}^\prime-E_{7}^\prime$, 
$A_7^\prime=H^\prime-E_{3}^\prime-E_{4}^\prime-E_{7}^\prime$.
\end{itemize}
Furthermore,  there is a symplectic arrangement $\hat{D}_1^\prime$ in $\C\P^2$, which realizes the virtual combinatorial type of the virtual homological expression. 

(2) Consider the following symplectic arrangement $\hat{D}_2$ in $\C\P^2$, which consists of $3$ degree $1$
symplectic spheres $\hat{F}_1, \hat{F}_2,\hat{F}_3$ intersecting at a single point $p_1$, and a degree $2$
symplectic sphere $\hat{F}_4$, which intersects with $\hat{F}_1, \hat{F}_2,\hat{F}_3$ at $3$ distinct points $p_2,p_3,p_4$ other than $p_1$, with a tangency of order $2$, where each $\hat{F}_k$ is $\hat{J}$-holomorphic for some compatible almost complex structure $\hat{J}$. We apply Lemma 3.1 to blow up at $p_1,p_2,p_3,p_4$, and let $E_1,E_2,E_3,E_4$ be the exceptional $(-1)$-spheres, which has an area $2\epsilon$ for a sufficiently small 
$\epsilon>0$. Let $F_1,F_2, F_3,F_4$ be the proper transforms of $\hat{F}_1, \hat{F}_2,\hat{F}_3,\hat{F}_4$. 
We continue to blow up at the intersection of $F_4$ with $F_1,F_2, F_3$, and let $E_5,E_6,E_7$ be the corresponding exceptional $(-1)$-spheres which has an area $\epsilon$. We continue to denote by 
$F_1,F_2, F_3,F_4$ the proper transforms, and let $F_5,F_6,F_7$ be the proper transforms of $E_2,E_3,E_4$. 
Then we get a symplectic configuration $D_2$ in $\C\P^2\#7\overline{\C\P^2}$, which consists of a disjoint union
of $7$ symplectic $(-2)$-spheres $F_1,F_2,\cdots,F_7$. When $\epsilon$ is chosen sufficiently small, $H, E_1,E_2,\cdots,E_7$ is a reduced basis of $\C\P^2\#7\overline{\C\P^2}$, and the following is the corresponding homological expression of $D_2$, $F_k\mapsto A_k$, where
\begin{itemize}
\item $A_1=H-E_{1}-E_{2}-E_{5}$, $A_2=H-E_{1}-E_{3}-E_{6}$, $A_3=H-E_{1}-E_{4}-E_{7}$, 
\item $A_4=2H-E_{2}-E_3-E_{4}-E_5-E_{6}-E_7$, 
\item $A_5=E_{2}-E_{5}$, $A_6=E_{3}-E_{6}$, $A_7=E_{4}-E_{7}$.
\end{itemize}
With this understood, we pick a point $p_8\in \C\P^2$ such that $p_2,p_3,p_8$ are not lying in a degree $1$
$\hat{J}$-holomorphic sphere. We blow up at $p_8$ and let $E_8$ be the exceptional $(-1)$-sphere, which also
has area $\epsilon$. Then $H, E_1,E_2,\cdots,E_8$ is a reduced basis of $\C\P^2\#8\overline{\C\P^2}$.
We consider the $(-2)$-class $\gamma=H-E_2-E_3-E_8$, where we note that $E_2,E_3,E_8$ satisfy (1) of 
Theorem 1.8, as in this case, the classes $E_1,E_2, E_3,E_4,E_8$ are minimal with respect to the partial order 
of infinitely-nearness. If we let $H^\prime=R(\gamma)(H)$, and $E_i^\prime=R(\gamma)(E_i)$, $i=1,2,\cdots,8$, 
then by Theorem 1.8, we obtain the following virtual homological expression of $D_2$, i.e., 
$F_k\mapsto A_k^\prime$, where
\begin{itemize}
\item $A_1^\prime=H^\prime-E_{1}^\prime-E_{2}^\prime-E_{5}^\prime$, 
$A_2^\prime=H^\prime-E_1^\prime-E_{3}^\prime-E_6^\prime$, 
\item $A_3^\prime=2H^\prime-E_1^\prime-E_{2}^\prime-E_{3}^\prime-E_{4}^\prime-E_7^\prime-E_{8}^\prime$,
\item $A_4^\prime=2H^\prime-E_{2}^\prime-E_{3}^\prime-E_{4}^\prime-E_{5}^\prime-E_{6}^\prime-E_{7}^\prime$, 
\item $A_5^\prime=H^\prime-E_{3}^\prime-E_{5}^\prime-E_{8}^\prime$, 
$A_6^\prime=H^\prime-E_{2}^\prime-E_{6}^\prime-E_{8}^\prime$, $A_7^\prime=E_{4}^\prime-E_{7}^\prime$.
\end{itemize}
Furthermore, there is a symplectic arrangement $\hat{D}_2^\prime$ in $\C\P^2$, which realizes the virtual combinatorial type of the virtual homological expression. 
\end{example}

It is easy to see that the virtual homological expressions in Example 1.9(1) and Example 1.9(2) are equivalent,
and the symplectic arrangements $\hat{D}_1^\prime$ and $\hat{D}_2^\prime$ have the same combinatorial type.
We formalize it in the following definition.

\begin{figure}[h]
   \centering
   \includegraphics[width=0.6\textwidth]{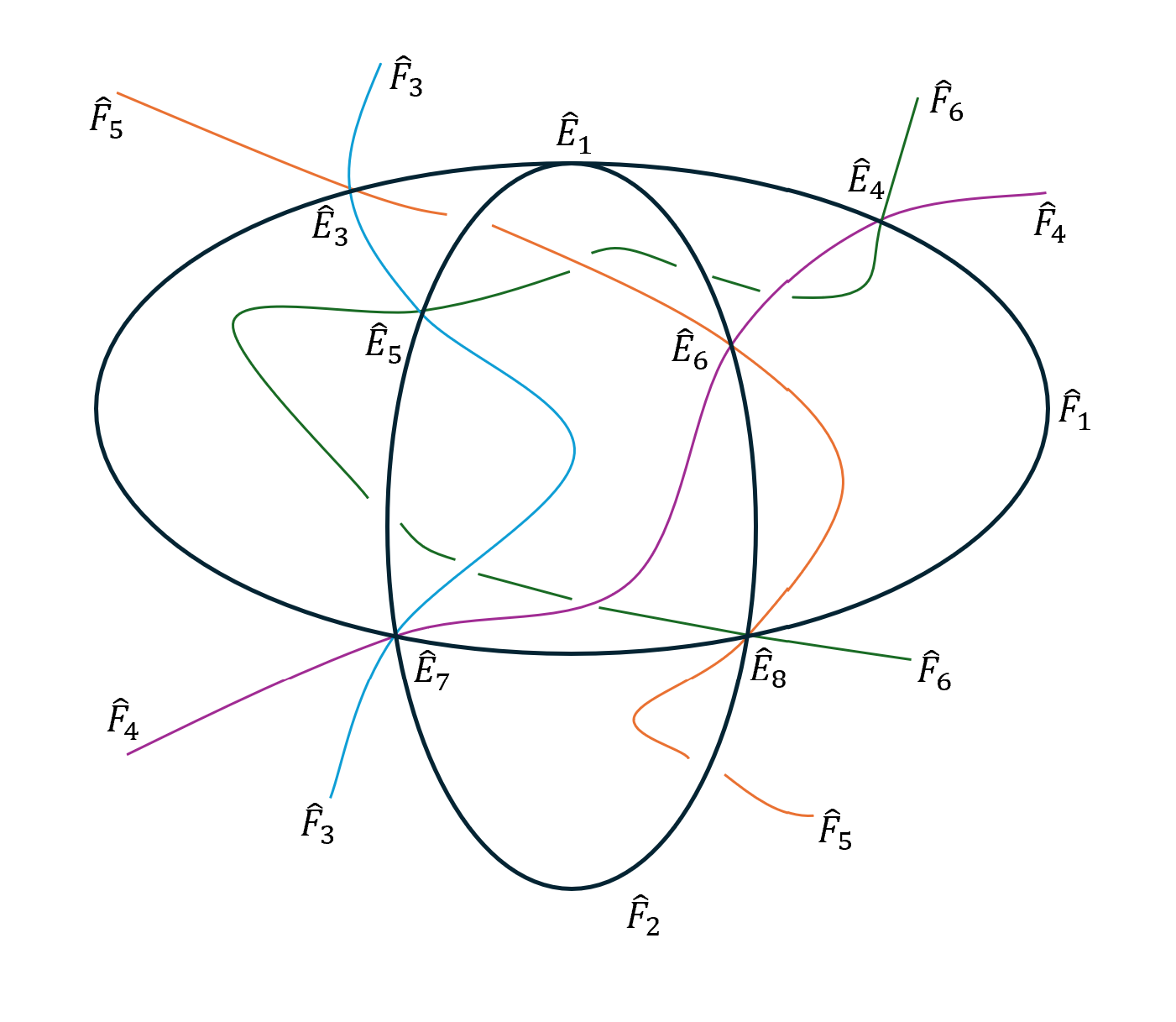}
   \caption*{Figure 1}
\end{figure}

\begin{definition}
Let $\hat{D}$ denote a symplectic arrangement in $\C\P^2$, which consists of $2$ degree $2$ symplectic spheres
$\hat{F}_1,\hat{F}_2$, and $4$ degree $1$ symplectic spheres $\hat{F}_3,\hat{F}_4,\hat{F}_5,\hat{F}_6$, realizing  the virtual combinatorial type of the following virtual homological expression of a disjoint union of $7$
symplectic $(-2)$-spheres in $\C\P^2\#8\overline{\C\P^2}$ (see Figure 1):
\begin{itemize}
\item $A_1=2H-E_{1}-E_{2}-E_{3}-E_{4}-E_{7}-E_{8}$,
\item $A_2=2H-E_{1}-E_{2}-E_{5}-E_{6}-E_{7}-E_{8}$,
\item $A_3=H-E_{3}-E_{5}-E_{7}$, $A_4=H-E_{4}-E_{6}-E_{7}$,
\item $A_5=H-E_{3}-E_{6}-E_{8}$, $A_6=H-E_{4}-E_{5}-E_{8}$, $A_7=E_{1}-E_{2}$.
\end{itemize}
We note that $\hat{F}_1,\hat{F}_2$ are tangent at $\hat{E}_1$. All other intersection points are transversal. 
\end{definition}

With the preceding understood, the following theorem is proved in Section 4.

\begin{theorem}
Assume there is a symplectic arrangement $\hat{D}$ in $\C\P^2$ (with respect to a K\"{a}hler form $\omega$) 
as defined in Definition 1.10 such that $\hat{D}$ is $\hat{J}$-holomorphic for some $\omega$-tame 
almost complex structure $\hat{J}$. Then there exists a smooth path $J_t$ of $\omega$-tame almost complex structures, $t\in [0,1]$, where $J_1=\hat{J}$ and $J_0$ is integrable, such that one of the following two cases 
must occur: 
\begin{itemize}
\item [{(i)}] There exists a smooth isotopy of $J_t$-holomorphic arrangement $\hat{D}_t$, for $t\in [0,1]$, such that 
$\hat{D}_t=\hat{D}$ at $t=1$ and at $t=0$, $\hat{D}_t$ is a complex arrangement with the same combinatorial type
(as shown in Figure 1). 
\item [{(ii)}] There is a smooth isotopy of $J_t$-holomorphic arrangement $\hat{D}_t$ for
$t\in (0,1]$ such that $\hat{D}_t=\hat{D}$ at $t=1$, and $\lim_{t\rightarrow 0} \hat{D}_t$ is a complex arrangement 
with a combinatorial type as shown in either Figure 2(1) or Figure 2(2).
\end{itemize}
In particular, there is a complex arrangement in $\C\P^2$ which has a combinatorial type as shown in 
either Figure 1, or Figure 2(1), or Figure 2(2).
\end{theorem}

\begin{figure}[h]
   \centering
   \begin{subfigure}[b]{0.45\textwidth}
      \centering
      \includegraphics[height=5.5cm]{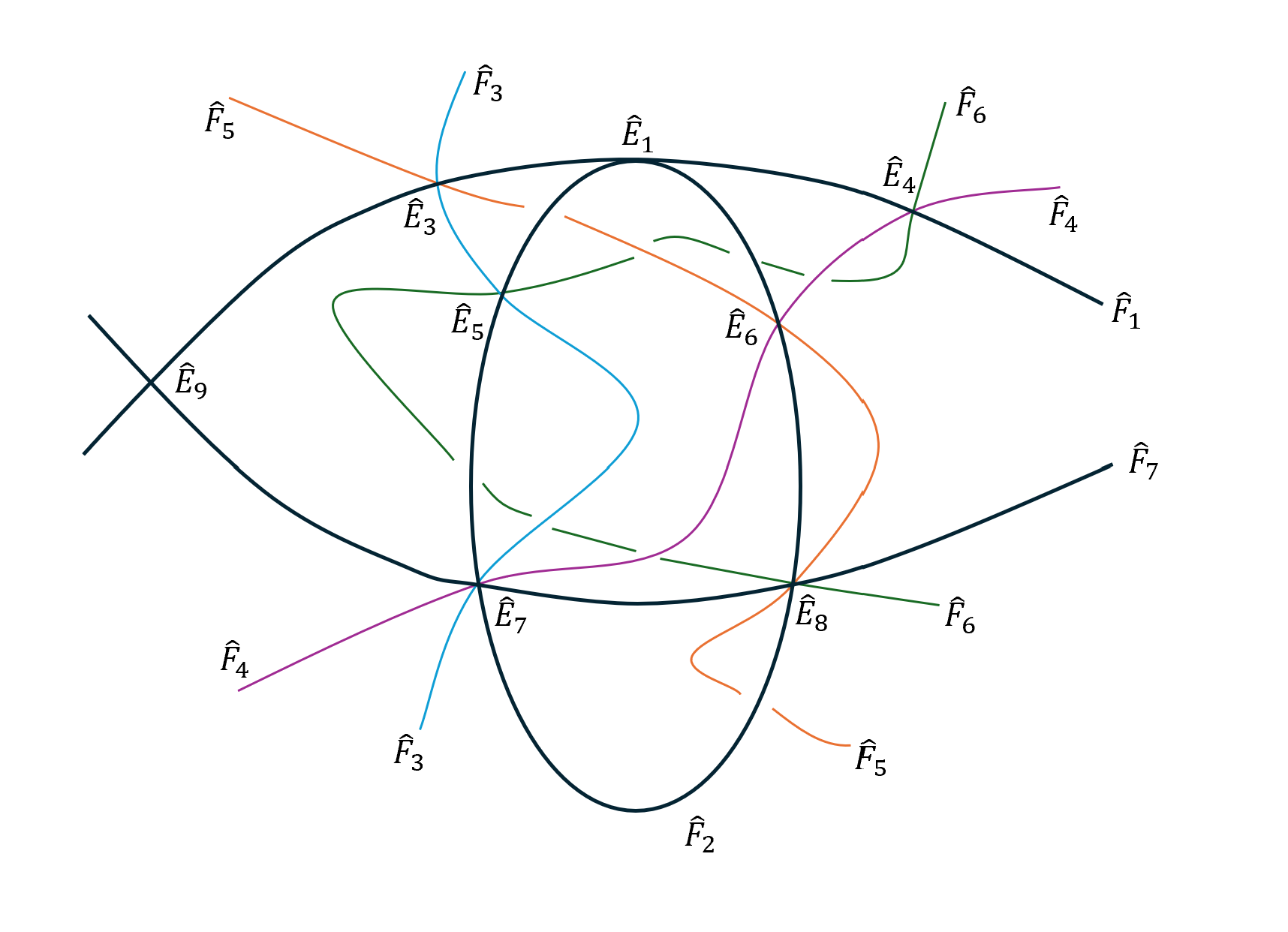}
      \caption*{Figure 2(1)}
   \end{subfigure}
   \hfill
   \begin{subfigure}[b]{0.45\textwidth}
      \centering
      \includegraphics[height=5.5cm]{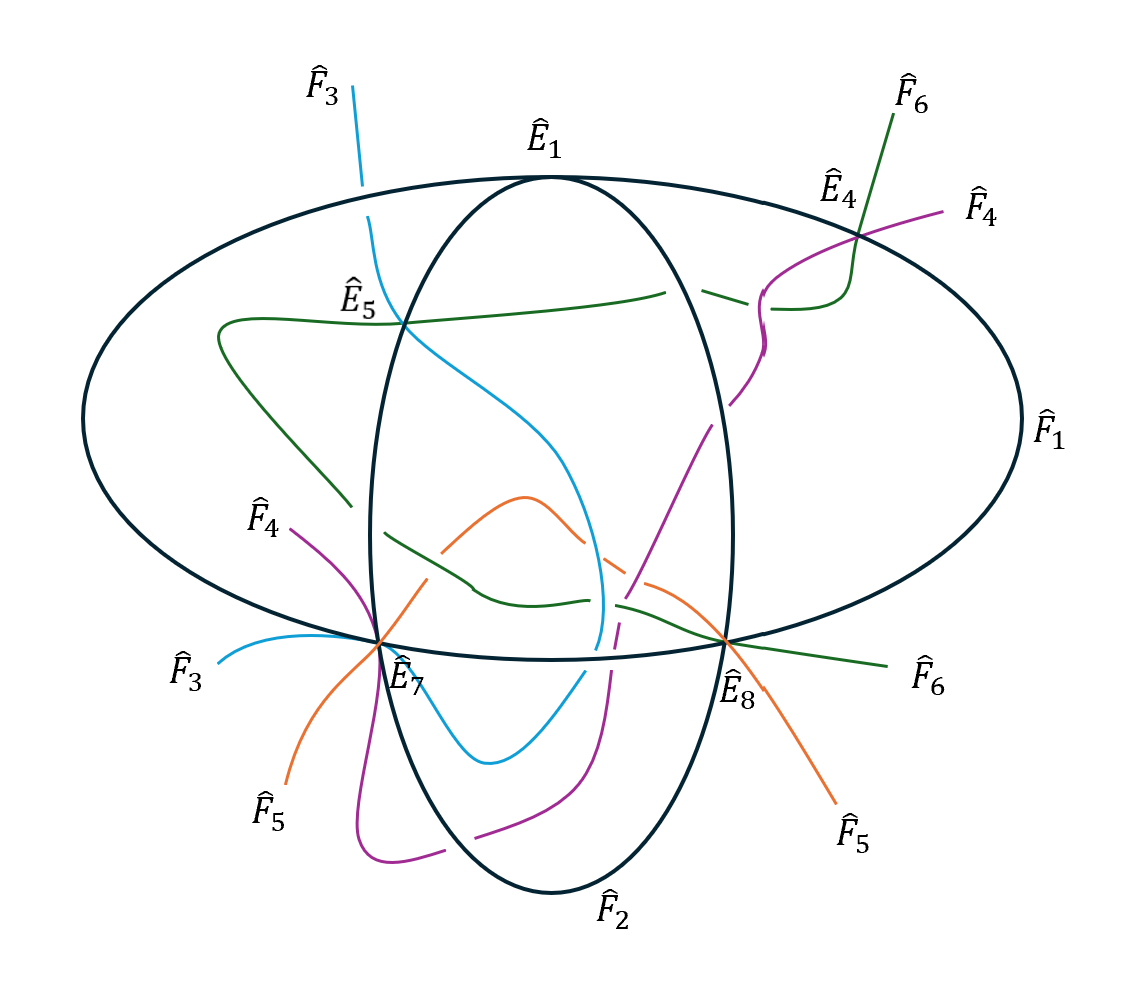}
      \caption*{Figure 2(2)}
   \end{subfigure}
\end{figure}

We observe the following corollary of Theorem 1.11.

\begin{corollary}
There exists no symplectic arrangement in $\C\P^2$ which has the combinatorial type of $\hat{D}$ 
in Definition 1.10.
\end{corollary}

\begin{proof}
Assume to the contrary that there is a symplectic arrangement which has the combinatorial type of $\hat{D}$. 
Then by Theorem 1.11, there is a complex arrangement in $\C\P^2$, denoted by $\hat{D}_0$, which has a combinatorial type as shown in either Figure 1, or Figure 2(1), or Figure 2(2). We will show that such a complex arrangement cannot exist by applying a suitable quadratic Cremona transformation $\Psi: \C\P^2\dashrightarrow \C\P^2$ to $\hat{D}_0$ to reach a contradiction, which finishes the proof of Corollary 1.12.

Consider first the case where the complex arrangement $\hat{D}_0$ has the combinatorial type as depicted in Figure 2(1). We apply the quadratic Cremona transformation $\Psi: \C\P^2\dashrightarrow \C\P^2$ to 
$\hat{D}_0$, which has three proper base points $\hat{E}_3, \hat{E}_7,\hat{E}_8$ (i.e., the one which 
corresponds to the reflection $R(\gamma)$, $\gamma=H-E_3-E_7-E_8$). As such, $\Psi$ is obtained by blowing up at $\hat{E}_3,\hat{E}_7,\hat{E}_8$ (note that $\hat{E}_3,\hat{E}_7,\hat{E}_8$ do not lie on a line
because otherwise, the line would intersect the line $\hat{F}_1$ in $2$ distinct points $\hat{E}_3, \hat{E}_9$, which is a contradiction), then blowing down the proper transforms of the 
$3$ lines passing through each pair of points $\hat{E}_3,\hat{E}_7$, $\hat{E}_3,\hat{E}_8$ and 
$\hat{E}_7,\hat{E}_8$ (see Chapter 5 in \cite{Moe}, or Figure 3(1) in Section 3.4). 
Note that in fact, the $3$ lines, i.e., 
those containing the pairs $\hat{E}_3,\hat{E}_7$, $\hat{E}_3,\hat{E}_8$ and $\hat{E}_7,\hat{E}_8$, are actually the components $\hat{F}_3,\hat{F}_5$ and $\hat{F}_7$ of $\hat{D}_0$. With this understood, it follows easily that under the birational automorphism $\Psi: \C\P^2\dashrightarrow \C\P^2$, $\hat{F}_2$ remains to be a conic,
$\hat{F}_1$, $\hat{F}_4$, and $\hat{F}_6$ remain to be a line, with the $3$ lines intersecting at $\hat{E}_4$, 
and each of the lines $\hat{F}_3,\hat{F}_5$ and $\hat{F}_7$ contracts to a point.  
Moreover, the image of $\hat{D}_0$ under 
$\Psi: \C\P^2\dashrightarrow \C\P^2$ is a complex arrangement consisting of $3$ lines $L_1,L_2,L_3$ intersecting at a single point $\hat{E}_4$, and a conic $S$ intersecting each of $L_1,L_2,L_3$ at a single point, 
$\hat{E}_1,\hat{E}_5,\hat{E}_6$, with a tangency of order $2$, which in fact has the combinatorial type of 
$\hat{D}_2$ in Example 1.9(2). We will show that such a complex arrangement 
does not exist by a rather elementary argument. 

Without loss of generality, we may assume that the intersection points  $\hat{E}_1,\hat{E}_5,\hat{E}_6$ and
$\hat{E}_4$ are all contained in the affine part $\C^2$, and moreover, $\hat{E}_4$ is the origin of  $\C^2$ and
the $3$ lines $L_1,L_2,L_3$ are given by equations $y=w_i x$, for $i=1,2,3$, where $w_i\in\C$, and $x,y$ are the coordinates of $\C^2$. With this understood, the conic $S$ is given by the zero set of a quadratic irreducible polynomial 
$$
Ax^2+Bxy+Cy^2+Dx+Ey+F=0,
$$
where $A,B,C,D,E,F\in\C$ and $F\neq 0$ because $S$ does not contain the origin $(0,0)$. To derive a contradiction, let $L$ be a line defined by $y=wx$, which has slope $w\in \C$. Then the intersection $S\cap L$
consists of points $(x,wx)$ where $x$ is a solution of the following quadratic equation
$$
(A+Bw+Cw^2)x^2+ (D+Ew)x +F=0. 
$$
In particular, if $(x,wx)$ is an intersection point of $S$ and $L$ with a tangency of order $2$, then the slope $w$
of $L$ must obey the following quadratic equation
$$
(D+Ew)^2-4(A+Bw+Cw^2)F=0. 
$$
In particular, the slopes $w_1,w_2,w_3$ of $L_1,L_2,L_3$ are $3$ distinct solutions of the above equation,
implying that the equation must be trivial, which is equivalent to 
$$
E^2-4CF=D^2-4AF=2DE-4BF=0. 
$$
Now we fix a square root $f$ of $F$, and choose square roots $a,c$ of $A,C$ such that $E=2cf$ and $D=2af$.
Then it follows easily that $B=2ac$ because $f\neq 0$. It follows that
$$
Ax^2+Bxy+Cy^2+Dx+Ey+F=a^2x^2+2acxy+c^2y^2+2afx+2cfy+f^2=(ax+cy+f)^2,
$$
contradicting the irreducibility of the polynomial. Hence a complex arrangement $\hat{D}_0$ 
with a combinatorial type as in Figure 2(1) cannot exist. 

Next assume the complex arrangement $\hat{D}_0$ has a combinatorial type as shown in Figure 1. We 
apply the same Cremona transformation $\Psi: \C\P^2\dashrightarrow \C\P^2$ which has three proper base 
points $\hat{E}_3, \hat{E}_7,\hat{E}_8$ to $\hat{D}_0$. In this case, the component $\hat{F}_2$ remains to 
be a conic but the conic $\hat{F}_1$ becomes a line. Moreover, $\hat{F}_4$ and $\hat{F}_6$ remain to be a line, with the $3$ lines intersecting at $\hat{E}_4$, and each of the lines $\hat{F}_3,\hat{F}_5$ contracts to a point.
It is easy to see that under the Cremona transformation $\Psi: \C\P^2\dashrightarrow \C\P^2$, $\hat{D}_0$ is transformed to a complex arrangement consisting of $3$ lines $L_1,L_2,L_3$ intersecting at a single point 
$\hat{E}_4$, and a conic $S$ intersecting each of $L_1,L_2,L_3$ at a single point, 
$\hat{E}_1,\hat{E}_5,\hat{E}_6$, with a tangency of order $2$. This is a contradiction, hence the complex arrangement $\hat{D}_0$ with a combinatorial type as shown in Figure 1 also cannot exist. 

Finally, assume $\hat{D}_0$ has a combinatorial type as shown in Figure 2(2). This time we apply a quadratic 
Cremona transformation to $\hat{D}_0$ which has proper base points $\hat{E}_7,\hat{E}_8$, with the third base point being infinitely near to $\hat{E}_7$. More concretely, we blow up at $\hat{E}_7,\hat{E}_8$, and then blow up at the intersection of the proper transform of $\hat{F}_3$ with the exceptional $(-1)$-sphere. It is easy to see that under this Cremona transformation, $\hat{F}_2$ remains to be a conic but $\hat{F}_1$ becomes a line, 
$\hat{F}_4$ and $\hat{F}_6$ remain to be a line, with the $3$ lines intersecting at $\hat{E}_4$, and each of the lines $\hat{F}_3,\hat{F}_5$ contracts to a point (see Chapter 5 in \cite{Moe}, or Figure 3(2) in Section 3.4). Again 
$\hat{D}_0$ is transformed to a complex arrangement consisting of $3$ lines $L_1,L_2,L_3$ intersecting 
at a single point, plus a conic $S$ intersecting each of $L_1,L_2,L_3$ at a single point with a tangency of 
order $2$. The contradiction implies that the complex arrangement $\hat{D}_0$ with a combinatorial type as shown in Figure 2(2) cannot exist. 

\end{proof}

It follows immediately that Theorem 1.8 (with Example 1.9) and Theorem 1.11, with Corollary 1.12, 
together imply the following corollary. 

\begin{corollary}
{\em(1)} The symplectic line arrangement $\hat{D}_1$ in $\C\P^2$ as described in Example 1.9(1), i.e., a Fano plane, does not exist. 

{\em(2)} The symplectic arrangement $\hat{D}_2$ which consists of three degree $1$ symplectic spheres and one
degree $2$ symplectic sphere in $\C\P^2$, as described in Example 1.9(2), does not exist. 
\end{corollary}

The fact that a complex line arrangement in $\C\P^2$ which is a Fano plane does not exist follows from a theorem of Hirzebruch in \cite{H}. On the other hand, we have just seen in the proof of Corollary 1.12 that a complex arrangement with the combinatorial type of $\hat{D}_2$ also cannot exist. If one could show that the symplectic arrangements $\hat{D}_1$, $\hat{D}_2$ can be deformed to a complex arrangement in the fashion of Theorem 1.11 (as in case (i)), then Corollary 1.13 would follow from the above two facts immediately. However, it is not clear that 
$\hat{D}_1$, $\hat{D}_2$ can be deformed to a complex arrangement; in fact, there exists a symplectic line arrangement which cannot be deformed to a complex line arrangement (cf. \cite{RuS}). The point we wish to make here is that by applying a Cremona transformation (including its symplectic analog), one could get around of the issue of proving it directly. Finally, we should point out that a line arrangement which is a Fano plane cannot exist even in the topological category, which is a result due to Ruberman and Starkston \cite{RuS}. However, part (2) of
Corollary 1.13 is a new result. The nonexistence of a symplectic Fano plane played a crucial role in 
our work \cite{C} on symplectic Calabi-Yau $4$-manifolds.

\vspace{3mm}

{\bf Acknowledgements:} I am indebted to Cagri Karakurt for useful discussions who once was a partner on the project but later voluntarily withdrew from it. With this said, I am responsible for the entire contents of this paper, including any possible errors or omissions. Part of the work was carried out during my visit at MPIM-Bonn in the summer of 2022. I am grateful for the excellent working environment as well as the hospitality and financial support from the institute.

\section{Homological assignments: finiteness and elimination by specifying areas} 
\subsection{Admissibility and successive blowing down: a new proof}
Let $X=\C\P^2\# N\overline{\C\P^2}$, which is either equipped with a symplectic structure or a complex structure. The canonical line bundle of $X$ is denoted by $K_X$. A basis $H, E_1, E_2,\cdots,E_N$ of $H^2(X)$ is called {\bf standard} if the following holds:
\begin{itemize}
\item $H^2=1$, $E_i^2=-1$ and $H\cdot E_i=0$, $\forall i$, and $E_i\cdot E_j=0$, $\forall i\neq j$.
\item $c_1(K_X)=-3H+E_1+E_2+\cdots+E_N$.
\end{itemize}
A standard basis $H, E_1, E_2,\cdots,E_N$ is {\bf ordered} if we fix the natural order of the $E_i$-classes $E_1, E_2,\cdots,E_N$. 

\begin{example}
In this paper, there are two primary examples of standard bases that will be considered, both of which are naturally ordered. 

(1) Suppose $X$ is a complex surface which is a successive blowing-up of $\C\P^2$. Let $E_i$, 
$1\leq i\leq N$, be the total transform of the exceptional divisor of the $i$-th blowing-up in $X$, and let $H$ be the total transform of a line $L\subset \C\P^2$. Then $H, E_1,E_2,\cdots,E_N$ is a standard basis, which we will call the standard basis associated to the successive blowing-up of $\C\P^2$. We note that the basis 
$H, E_1,E_2,\cdots,E_N$ is naturally ordered. 

(2) Let $\omega$ be a symplectic structure on $X$. Denote by $\E_\omega$ the set of classes $E\in H^2(X)$ such that $E$ can be represented by a smooth $(-1)$-sphere in $X$ and 
$c_1(K_X)\cdot E=-1$. A standard basis $H,E_1,E_2,\cdots,E_N$ is called a 
{\bf reduced basis} of $(X,\omega)$ if in addition, $E_i\in \E_\omega$ for each $i$, and the following area conditions are satisfied: $\omega(E_2)\leq \omega(E_1)$, and when $N\geq 3$,
$\omega(E_N)=\min_{E\in\E_\omega}\omega(E)$, and for any $2<i<N$, 
$\omega(E_i)=\min_{E\in\E_i}\omega(E)$, where for each $i<N$, 
$\E_i:= \{E\in \E_\omega|E\cdot E_j=0, \forall j>i\}$. We note that a reduced basis 
$H,E_1,E_2,\cdots,E_N$ is naturally ordered. 
\end{example}

Let $H,E_1,E_2,\cdots,E_N$ be a standard basis of $H^2(X)$, and let $A\in H^2(X)$. Then 
$A=aH-\sum_{i=1}^N b_i E_i$, where $a,b_i\in\Z$. We will call $a, b_i$ {\bf the $a$-coefficient and
$b_i$-coefficients} of $A$ with respect to the standard basis. Moreover, we define the {\bf virtual genus}
of $A$ to be the following integer
$$
g(A):=\frac{1}{2}(A^2+c_1(K_X)\cdot A)+1.
$$

\begin{definition}
(1) We say a class $A\in H^2(X)$ is {\bf admissible with respect to a standard basis} $H,E_1,E_2,\cdots,E_N$ if the vector formed by the $a$-coefficient and $b_i$-coefficients of $A$ is admissible in the sense of Definition 1.1.

(2) Assume in addition, $H,E_1,E_2,\cdots,E_N$ is ordered. Then an admissible class $A$ is called {\bf positive} with respect to the order of $H,E_1,E_2,\cdots,E_N$ if when the $a$-coefficient of $A$ is less than or equal to $0$,
the $b_i$-coefficients of $A$ satisfy the following condition: for any $b_i\neq 0, b_j\neq 0$, $i<j$ if $b_i<b_j$. 
\end{definition}
We remark that in the case where the $a$-coefficient of $A$ is less than or equal to $0$, i.e., $a\leq 0$, it follows easily that $g(A)=0$ must be true, and moreover, $2a\geq 1+A^2$. In particular, $A^2<0$. 
See \cite{C}, Lemma 3.4, for more details.

\begin{lemma}
In the following situations, the homology class $A$ is admissible with respect to the corresponding standard basis,
which is also positive with respect to the natural order of the standard basis:

{\em (1)} {\em (}The holomorphic case{\em)} $X$ is a successive blowing-up of $\C\P^2$,
$H, E_1,E_2,\cdots,E_N$ is the standard basis associated to the successive blowing-up, 
and $A$ is the class of an irreducible curve $C$ in $X$. In this case, the $a$-coefficient is always non-negative.

{\em (2)}  {\em (}The symplectic  case{\em)} $H, E_1,E_2,\cdots,E_N$ is a reduced basis of $(X,\omega)$ and $A$ is the class of a $J$-holomorphic curve $C$ in $X$, where $J$ is some $\omega$-compatible almost complex structure {\em(}e.g. $C$ is a smoothly embedded symplectic surface in $X${\em)}.
\end{lemma}

We remark that Lemma 2.3(2) (i.e., the symplectic case) is already known when $A$ is the class of a smoothly embedded symplectic surface (cf. \cite{C}, Lemmas 3.3 and 3.4). The proof given below, which explores (in both situations) the fact that there is a successive blowing-down process associated to the standard basis 
$H, E_1,E_2,\cdots,E_N$, is an alternative proof, independent of the proof in \cite{C}, and is more in line with the
thinking of this paper. 

\begin{proof}
We write $X=X_N$, $C=C_N$, and $A=A_N$.

(1) Let $\pi_N: X_N\rightarrow X_{N-1}$ be the blowing-down of 
the exceptional divisor $E_N$, and let $C_{N-1}$ be the direct image of $C_N$ under $\pi_N$. If $C_{N-1}=0$, then $A_N=E_N$, and we are done in this case. If $C_{N-1}\neq 0$, then 
$C_{N-1}$ is an irreducible curve in $X_{N-1}$. Let $\pi_N^\ast C_{N-1}$ be the total transform of
$C_{N-1}$ in $X_N$. Then $C_N=\pi_N^\ast C_{N-1}-b_N E_N$, where $b_N:=C_N\cdot E_N\geq 0$.
Furthermore, we note that $b_N=0$ if and only if $E_N$ and $C_N$ are disjoint, and $b_N=1$ if and only if $E_N$ and $C_N$ intersect transversely at a single point, and that only in these cases,
$C_{N-1}$ continues to be nonsingular when $C_N$ is nonsingular. If we let $A_{N-1}$ be the class of $C_{N-1}$
in $X_{N-1}$, then $A_N=\pi_N^\ast A_{N-1}-b_N E_N$. The lemma follows easily by induction on $N$. Note that we always have the $a$-coefficient of $A$ non-negative in this case. 

(2) Assume first that $N\geq 3$. We let $J$ be a compatible almost complex structure such that $C_N$ is $J$-holomorphic. Since $N\geq 3$, and since $E_N$ has the minimal symplectic area, $E_N$ can be represented by a $J$-holomorphic $(-1)$-sphere $S_N$ by \cite{KK}. If $C_N=S_N$, then $A_N=E_N$, and we are done in this case. If $C_N\neq S_N$, we set $b_N:=C_N\cdot S_N$. Then $b_N\geq 0$ since both $C_N$ and $S_N$ are $J$-holomorphic. Moreover, as we showed in \cite{C1}, Section 4, we can slightly perturb $S_N$ if necessary, so that it intersects $C_N$ transversely and positively, and we can then symplectically blow down $X_N$ to $X_{N-1}$ along the (perturbed) symplectic $(-1)$-sphere $S_N$, such that $C_N$ descends to a (generally singular) symplectic surface $C_{N-1}$ in $X_{N-1}$, where $C_{N-1}$ is smoothly embedded if and only if $b_N\leq 1$ and $C_N$ is smoothly embedded. Furthermore, $H, E_1,E_2, \cdots, E_{N-1}$ descends to a reduced basis of $X_{N-1}$ (cf. \cite{C1}, Lemma 4.2), and $C_{N-1}$ can be made $J$-holomorphic for some compatible $J$ on $X_{N-1}$. 
If we let $A_{N-1}$ be the class of $C_{N-1}$ in $X_{N-1}$, then $A_N=A_{N-1}-b_N E_N$.
With this understood, if $N-1\geq 3$, we can continue this process and run an induction on $N$. 

Hence it remains to consider the case where $N\leq 2$. Assume $N=2$ first. In this case, there are
three $(-1)$-classes $E_1,E_2, H-E_1-E_2$ which can be represented by symplectic spheres, and 
there are two possibilities which we shall discuss separately. 

First, consider the case where the class $H-E_1-E_2$ has the minimal symplectic area. We fix a $J$ such that $C_N$ is $J$-holomorphic. Then by \cite{KK}, $H-E_1-E_2$ can be represented by a $J$-holomorphic $(-1)$-sphere $S_N$. If $C_N=S_N$, then $A_N=H-E_1-E_2$ and we are done. Suppose $C_N\neq S_N$. Setting $b_N:=C_N\cdot S_N\geq 0$, we slightly perturb $S_N$ so that it intersects $C_N$
transversely and positively, then we symplectically blow down $X_N$ to $\hat{X}$ along the 
perturbed $(-1)$-sphere $S_N$, where $C_N$ descends to $\hat{C}$ in $\hat{X}$. We point out 
that $\hat{X}=\s^2\times \s^2$, and $\hat{C}$ is $\hat{J}$-holomorphic with respect to some compatible almost complex structure $\hat{J}$ on $\hat{X}$. 

With this understood, let $e_1, e_2\in H^2(\hat{X})$ be the descendant of $E_1,E_2$ respectively. 
(Correspondingly, $H-E_2$, $H-E_1$ are the total transform of $e_1$, $e_2$ in $X_N$ respectively.)
Then $e_1,e_2$ form a basis of $H^2(\hat{X})$, such that $e_1\cdot e_2=1$, and $e_1\cdot e_1=e_2\cdot e_2=0$. Furthermore, $c_1(K_{\hat{X}})=-2e_1-2e_2$. Finally, since the area of $E_1$ is greater than or equal to the area of $E_2$, we note that $e_2$ has the minimal area among 
$e_1, e_2$.
Now we apply Lemma 2.4 of \cite{C2} to the classes $e_1$ and $e_2$. It follows easily that $e_2$ is represented by a $\hat{J}$-holomorphic sphere $\hat{S}_2$. In fact, $\hat{X}$ is foliated by a $\s^2$-family of such $\hat{J}$-holomorphic spheres which contains $\hat{S}_2$. 
Moreover, there is a $\hat{J}$-holomorphic sphere 
$\hat{S}_1$, such that $e_1$ is represented by $\hat{S}_1+m \hat{S}_2$ for some $m\geq 0$. 
Note that $\hat{S}_1^2=-2m<0$ if $m\neq 0$. 

With the preceding understood, we next examine the possible scenarios of $\hat{C}$ in $\hat{X}$. 
First, if $\hat{C}$ is one of the $\hat{J}$-holomorphic spheres representing $e_2$, then 
$b_N=C_N\cdot S_N$ must be equal to $0$ or $1$ as $\hat{C}$ is smoothly embedded. In this case, 
$$
A_N=H-E_1-b_N(H-E_1-E_2)=(1-b_N)H-(1-b_N)E_1 +b_NE_2,
$$
so we are done in this case. Secondly, suppose $\hat{C}=\hat{S}_1$. Then $b_N=0$ or $1$ as well, as $\hat{C}$ is smoothly embedded. Note that
the total transform of $\hat{S}_1$ in $X_N$ is $H-E_2-m(H-E_1)$, so that in this case, we have
$$
A_N=H-E_2-m(H-E_1)-b_N(H-E_1-E_2)= (-m+1-b_N) H+(m+b_N) E_1-(1-b_N)E_2.
$$
We are done in this case as well, as $b_N=0$ or $1$. 
Finally, we consider the case $\hat{C}\neq \hat{S}_1$ and $\hat{C}$ is not one of the $\hat{J}$-holomorphic spheres representing $e_2$. Then $\hat{C}\cdot \hat{S}_2>0$ and 
$\hat{C}\cdot \hat{S}_1\geq 0$. In this case, we need to recall the fact that $\hat{C}$ contains 
a point $p$ such that in a small neighborhood $U$ of $p$, $\hat{C}\cap U$ consists of $b_N$ many
embedded disks intersecting transversely at $p$ (cf. \cite{C1}, Section 4). With this understood, 
since $\hat{X}$ is foliated 
by $\hat{J}$-holomorphic spheres representing $e_2$, it follows easily that $\hat{C}\cdot e_2\geq b_N$.
Now if we write $\hat{C}=ue_1+ve_2$, then $u=\hat{C}\cdot e_2\geq b_N$ and 
$v=\hat{C}\cdot e_1=\hat{C}\cdot \hat{S}_1+m \hat{C}\cdot e_2\geq mb_N$. 
It follows easily that if $m>0$, the class $A$ is admissible with respect to $H,E_1,E_2$, as 
$$
A_N=u(H-E_2)+v(H-E_1)-b_N(H-E_1-E_2)=(u+v-b_N)H-(v-b_N)E_1-(u-b_N)E_2.
$$
Now let $m=0$. Then $\hat{S}_1^2=0$, so that $\hat{X}$ is also foliated by $\hat{J}$-holomorphic spheres representing $e_1$. If $\hat{C}$ is one of the $\hat{J}$-holomorphic spheres representing $e_1$, we have 
$$
A_N=H-E_2-b_N(H-E_1-E_2)=(1-b_N)H+b_NE_1-(1-b_N)E_2,
$$
where $b_N=0$ or $1$, and we are done. Otherwise, we have $\hat{C}\cdot e_1\geq b_N$ instead. In this case, we have $v\geq b_N$ as well, and the lemma also follows. This finishes the discussion when $H-E_1-E_2$ has the minimal area.

Next, we consider the remaining case for $N=2$, where $E_2$ has the minimal area among the three classes $E_1,E_2, H-E_1-E_2$. In this case we can represent $E_2$ by a $J$-holomorphic 
$(-1)$-sphere $S_N$. If $C_N=S_N$, we have $A_N=E_2$ and we are done. If $C_N\neq S_N$, we set
$b_N:=C_N\cdot S_N\geq 0$. We symplectically blow down $X_N$ to 
$X_1=\C\P^2\# \overline{\C\P^2}$ along $S_N$, where $C_N$ descends to a $J_1$-holomorphic curve $C_1$ in $X_1$ for some compatible $J_1$, and the proof is reduced to the case of $N=1$. 

Let $A_1$ be the class of $C_1$. Then we apply Lemma 2.3 of \cite{C2} to conclude that either 
$E_1$ is represented by a $J_1$-holomorphic $(-1)$-sphere, or $X_1$ is foliated by $J_1$-holomorphic spheres $S$ representing the class $H-E_1$, together with a $J_1$-holomorphic section $\tilde{C}$ such that $E_1=\tilde{C}+m S$ for some $m>0$. (Note that, for $N\geq 2$, $E_N$ can always be represented by a $J$-holomorphic 
$(-1)$-sphere for any given $J$ as long as $E_N$ has minimal area, on the contrary, for $N=1$,
$E_N$ cannot always be represented by a $J$-holomorphic $(-1)$-sphere if $J$ is not chosen generic.) In the former case, it follows easily that $A_1$ is admissible with respect to $H,E_1$, from which it follows easily that 
$A_N$ is admissible with respect to $H,E_1, E_2$. In the latter case, there are several possibilities. If $C_1$ is one of the $J_1$-holomorphic spheres $S$ representing the class $H-E_1$, then $A_1=H-E_1$, so that $A_N=H-E_1-b_NE_2$ where $b_N=0$ or $1$ as $C_1$ is embedded. If $C_1=\tilde{C}$, then 
$A_1=E_1-m(H-E_1)=-mH+(m+1)E_1$, and $A_N=-mH+(m+1)E_1-b_NE_2$, where $b_N=0$ or $1$ as $C_1=\tilde{C}$ is embedded. 
Finally, if none of the above is true, we have $C_1\cdot S>0$ and $C_1\cdot \tilde{C}\geq 0$. If we write $A_1=aH-b_1E_1$. Then $b_1=C_1\cdot E_1=C_1\cdot (\tilde{C}+mS)>0$, and $a>0$ because $H$, $E_1$ and $C_1$ all have positive areas. It follows easily that, in this case, $A_N$ is also admissible with respect to $H,E_1, E_2$. This concludes the discussion for $N=2$, where in the process the case of $N=1$ is also proved. The case $N=0$ is trivial, so the proof of the lemma is completed. 

\end{proof}

\subsection{Bounding the number of admissible classes}
Note that by the adjunction inequality, the virtual genus of the class of a pseudoholomorphic curve
is always non-negative. On the other hand, in light of Lemma 2.3, we shall be mainly concerned with admissible homology classes. With this understood, the following lemma establishes a fundamental finiteness condition. 

\begin{lemma}
Fix any standard basis $H,E_1,E_2,\cdots,E_N$. For any $\alpha\in\Z$ and any constant $C>0$,
the number of classes $A$ admissible with respect to $H,E_1,E_2,\cdots,E_N$, such that
$A^2=-\alpha$, $g(A)\geq 0$, with the $a$-coefficients of $A$ bounded from above by $C$, is finite.
More precisely, when the $a$-coefficient of $A$ is positive, the $b_i$-coefficients of $A$ are also bounded from above by $C$, and when it's non-positive, the $a$-coefficient of $A$ is bounded from below
by $\frac{1}{2}(1-\alpha)$, and the $b_i$-coefficients of $A$ by $-\frac{1}{2}(1+\alpha)$. 
\end{lemma}

\begin{proof}
Let $A=aH-\sum_{i=1}^N b_iE_i$. Then we have 
$$
a^2-\sum_{i=1}^N b_{i}^2=-\alpha,\;\; g(A)=\frac{1}{2}(-\alpha-3a+\sum_{i=1}^N b_{i})+1.
$$
It follows easily that $\sum_{i=1}^N b_i(b_i-1) +2g(A)=(a-1)(a-2)$. Note that $b_i(b_i-1)\geq 0$ 
for each $i$. On the other hand, $g(A)\geq 0$ by assumption. It follows easily that for each $i$, $b_i(b_i-1) \leq (a-1)(a-2)$. Now suppose $a\leq C$. If $a>0$, then it is easy to see that for each 
$i$, $0\leq b_i\leq C$. If $a\leq 0$, then $A$ being admissible implies that 
$|a|\leq \frac{1}{2}(\alpha-1)$, and $|b_i|\leq |a|+1$ for each $i$ (see \cite{C}, Lemma 3.4, for more details). It is clear that there are only finitely many such classes $A$.

\end{proof}

It turns out that an upper bound for the $a$-coefficients can be established when $N$ is relatively small. More precisely, we have the following

\begin{proposition}
Assume $N\leq 8$, and fix any standard basis $H,E_1,E_2,\cdots,E_N$. Then for any integers 
$\alpha$ and $g\geq 0$, there exists a constant $C:=C(\alpha,g,N)>0$ depending on $\alpha, g$ and $N$ alone, 
such that for any $A\in H^2(X)$ with $A^2=-\alpha$, $g(A)=g$, the $a$-coefficient of $A$ is bounded from above by $C$. As a consequence, the number of classes $A$ admissible with respect to $H,E_1,E_2,\cdots,E_N$ such that $A^2=-\alpha$, $g(A)=g$ is finite. Moreover, if $N=9$ but $\alpha+2g-2>0$, then the conclusion continues to hold. 

\end{proposition}

Note that Theorem 1.3 follows immediately from Lemma 2.3 and Proposition 2.5. 

We remark that Proposition 2.5 is not true if $N>9$ or $N=9$ but $\alpha+2g-2\leq 0$. For example, for any integer $t\geq 0$, the class
$$
A_t:=(3t+1)H-(t+1)E_1-(t+1)E_2-(t+1)E_3-tE_4-\cdots-tE_9
$$
is admissible, with $A_t^2=-2$, $g(A_t)=0$, but the $a$-coefficient of $A_t$ is unbounded. 

Proposition 2.5 will follow from the following two lemmas. 

\begin{lemma}
Fix any standard basis $H,E_1,E_2,\cdots,E_N$. Let $A=aH-\sum_{i=1}^N b_iE_i$ 
be any class (not necessarily admissible) such that $A^2=-\alpha$, $g(A)=g\geq 0$. Let $M$ be the number of non-zero $b_i$-coefficients of $A$. Moreover, set
$$
\delta:=\max \{0, \; 1-(\alpha+2g-2)\}. 
$$
Suppose $\alpha+2g-2\geq -2$. Then $M\geq 10-\delta$ if $a>3$.
\end{lemma}

This is an extension of Lemma 3.5 in \cite{C}, with the same proof strategy.

\begin{proof}
Suppose to the contrary, $M\leq 9-\delta$ where $a>3$.

We first note that the following {\bf Claim} in the proof of Lemma 3.5 of \cite{C} continues to hold, i.e., 

\vspace{1mm}

{\bf Claim:} {\it There are distinct indices $i,j,k$ such that (i) $b_i,b_j,b_k$ are positive, and (ii) $a-(b_i+b_j+b_k)<0$.
}

\vspace{1mm}

To see this, note that $(a-1)(a-2)=\sum_{i=1}^N b_i(b_i-1)+2g$. Since $g\geq 0$, it follows that, as in the proof of
Lemma 3.5 of \cite{C}, we have $b_i\leq a-1$ for any $b_i>0$ (note that we assumed $a>3$). Consequently, if there are at most two
positive $b_i$'s, then $\sum_{i=1}^N b_i\leq 2(a-1)$, which implies $-3a+\sum_{i=1}^N b_i\leq -a-2\leq -6$. But $-3a+\sum_{i=1}^N b_i=\alpha+2g-2\geq -2$, which is a contradiction. Hence there are at least three distinct 
$b_i$'s which are positive.

Next, we observe that the condition $\alpha+2g-2\geq -2$ implies that $\delta$ only takes values 
$0,1,2,3$. With this understood, if for any distinct $i,j,k$, $b_i+b_j+b_k\leq a$, then it follows easily, 
with $M\leq 9-\delta$ and observing that there are at least $\delta$ many positive $b_i$'s, that
$$
\alpha+2g-2=-3a+\sum_{i=1}^N b_i\leq -3a+3a-\delta=-\delta.
$$
But $\delta:=\max \{0, \; 1-(\alpha+2g-2)\}$, which is a contradiction. Hence there must be distinct $i,j,k$, such that $b_i+b_j+b_k> a$. It is clear that we may assume that $b_i,b_j,b_k$ are positive. Hence the {\bf Claim}. 

The argument in the proof of Lemma 3.5 of \cite{C} continues to hold. (We will use the same notations from 
Lemma 3.5 of \cite{C} here.) In particular, with $a>3$, we have $\tilde{a}=2$ or $3$. We need to examine the class 
$\tilde{A}=R_{ijk}(A)$ according to the value of $\tilde{a}$, as we did in \cite{C}. 

Suppose $\tilde{a}=2$. Note that $(\tilde{a}-1)(\tilde{a}-2)\geq 2g$, which implies that $g=0$ must be true. With this understood, the proof proceeds in the same way as in \cite{C}. 

Suppose $\tilde{a}=3$. Then $(\tilde{a}-1)(\tilde{a}-2)\geq 2g$ implies $g=0$ or $1$. If $g=0$, the lemma follows as in \cite{C}. The new case occurs when $g=1$, where $\tilde{b}_i=0$ or $1$, which is easily seen from the identity $(\tilde{a}-1)(\tilde{a}-2)=2g+\sum_{i=1}^N \tilde{b}_i(\tilde{b}_i-1)$. It follows that
$$
\tilde{A}=3H-E_{j_1}-E_{j_2}-\cdots-E_{j_{9+\alpha}}.
$$
There are $9+\alpha$ many non-zero $b_i$-coefficients in the expression of $\tilde{A}$, contradicting the assumption that $M\leq 9-\delta$, because $9+\alpha>9-\delta$ as $g=1$. This finishes the proof. 

\end{proof}

\begin{lemma}
Fix any standard basis $H,E_1,E_2,\cdots,E_N$, and assume $N\leq 8$. Let 
$A=aH-\sum_{i=1}^N b_i E_i$ be any class (not necessarily admissible) with $A^2=-\alpha$, $g(A)=g\geq 0$. Assume $a>0$. 
\begin{itemize}
\item [{(1)}] If $\alpha+2g-2\geq 0$, then $\alpha>0$ must be true. Moreover, $a\leq \sqrt{8\alpha}$. 
\item [{(2)}] If $\alpha\leq 0$, then $a\leq 6|\alpha+2g-2|$ if $N=8$, $a\leq 3|\alpha+2g-2|$ if $N=7$,
and  $a\leq 2|\alpha+2g-2|$ if $N\leq 6$.
\item [{(3)}] If $\alpha>0$ and $\alpha+2g-2<0$, then $a\leq 7$. 
\end{itemize}
\end{lemma}

\begin{proof}
Note that $A^2=-\alpha$ and $g(A)=g$ give rise to
$$
\sum_{i=1}^N b_i=3a+(\alpha+2g-2), \;\; \sum_{i=1}^N b_i^2=a^2+\alpha.
$$
Let $M$ be the number of non-zero $b_i$'s. 

(1) Assume $\alpha+2g-2\geq 0$. Then 
$$
\frac{3a}{M}\leq \frac{1}{M} \sum_{i=1}^N b_i\leq (\frac{1}{M} \sum_{i=1}^N b_i^2)^{1/2}=(\frac{a^2+\alpha}{M})^{1/2}.
$$
Since $a>0$, it follows easily that $9a^2\leq M(a^2+\alpha)$. With $M\leq N\leq 8$, we must have $\alpha>0$. Moreover,
$$
a\leq (\frac{M\alpha}{9-M})^{1/2}\leq \sqrt{8\alpha}. 
$$

(2) Assume $\alpha\leq 0$. Then by (1) above, $\alpha+2g-2<0$. If $3a\leq |\alpha+2g-2|$, we are done. Otherwise,
let $0<\epsilon<3$ be any real number such that $\epsilon a>|\alpha+2g-2|$. Then 
$$
\frac{3a-\epsilon a}{M}<\frac{3a-|\alpha+2g-2|}{M}
\leq \frac{1}{M} \sum_{i=1}^N b_i\leq (\frac{1}{M} \sum_{i=1}^N b_i^2)^{1/2}=(\frac{a^2+\alpha}{M})^{1/2},
$$
which gives $(3-\epsilon)^2a^2< M(a^2+\alpha)$. Since $\alpha\leq 0$, we will 
arrive at a contradiction if $(3-\epsilon)^2-N>0$. If $N=8$, then $(3-\epsilon)^2-N>0$ if we choose 
$\epsilon=1/6$. This implies that $\epsilon a\leq |\alpha+2g-2|$ for $\epsilon=1/6$. It follows that if $N=8$, $a\leq \frac{1}{\epsilon}|\alpha+2g-2|=6|\alpha+2g-2|$. By the same argument, if $N=7$, we may choose $\epsilon=1/3$, and if $N\leq 6$, we may choose $\epsilon=1/2$, so that 
$a\leq 3|\alpha+2g-2|$ if $N=7$ and $a\leq 2|\alpha+2g-2|$ if $N\leq 6$.

(3) Assume $\alpha>0$ and $\alpha+2g-2<0$. Since $g\geq 0$, we must have $g=0$, $\alpha=1$ in this case. In particular, $\alpha+2g-2=-1$. 

Let $0<\epsilon<3$ be any real number such that $\epsilon a> |\alpha+2g-2|=1$. Then
$$
\frac{3a-\epsilon a}{M}<\frac{3a-|\alpha+2g-2|}{M}
\leq \frac{1}{M} \sum_{i=1}^N b_i\leq (\frac{1}{M} \sum_{i=1}^N b_i^2)^{1/2}=(\frac{a^2+1}{M})^{1/2},
$$
which gives $(3-\epsilon)^2a^2< 8(a^2+1)$. Equivalently, 
$$
(1-6\epsilon +\epsilon^2)a^2<8. 
$$
Now choose $\epsilon=1/7$. Then the assumption $\epsilon a>1$ means $a>7$, $\epsilon^2a^2>1$.
It follows from $(1-6\epsilon +\epsilon^2)a^2<8$ that $a^2/7+1<8$, which contradicts $a>7$. Thus the 
assumption $\epsilon a>1$, with $\epsilon=1/7$, cannot be true. Hence $a\leq 7$. 

\end{proof}

We summarize Lemmas 2.6 and 2.7 into the following corollary, from which Proposition 2.5 follows easily
(with the help of Lemma 2.4). Moreover, it also gives an explicit description of the constant $C(\alpha,g,N)$ in Proposition 2.5.

\begin{corollary}
Fix any standard basis $H,E_1,E_2,\cdots,E_N$, and let $A=aH-\sum_{i=1}^N b_i E_i$ be any class (not necessarily admissible) with $A^2=-\alpha$, $g(A)=g\geq 0$.  Assume $a>0$. 
\begin{itemize}
\item [{(1)}] If $\alpha+2g-2>0$ and $N\leq 9$, then $a\leq 3$. 
\item [{(2)}] If $\alpha+2g-2=0$ and $N\leq 8$, then $a\leq 3$.
\item [{(3)}] Suppose $\alpha+2g-2=-1$. If $N\leq 7$, then $a\leq 3$, and if $N\leq 8$, then $a\leq 7$. 
\item [{(4)}] Suppose $\alpha+2g-2\leq -2$ and $N\leq 8$. Then $a\leq 6|\alpha+2g-2|$ if $N=8$, 
$a\leq 3|\alpha+2g-2|$ if $N=7$, and $a\leq 2|\alpha+2g-2|$ if $N\leq 6$.
\end{itemize}
\end{corollary}

\subsection{Uniqueness of reduced bases up to a symplectomorphism}
To proceed further, we include here a lemma addressing the issue of uniqueness of reduced bases for a given
symplectic structure $\omega$ on $X$. In particular, it shows that a reduced basis $H,E_1,E_2, \cdots,E_N$ 
of $(X,\omega)$ is unique if and only if  for any distinct indices $i,j,k$, $\omega(H-E_i-E_j-E_k)>0$. In general,
it is only unique up to a symplectomorphism.

To this end, we recall that $(X,\omega)$ is called {\bf monotone} if $c_1(K_\omega)=-\lambda [\omega]$ 
for some constant $\lambda>0$. It is easy to see that under this assumption, every symplectic $(-1)$-sphere 
in $(X,\omega)$ has the same area, from which it follows easily that every standard basis 
$H,E_1,E_2,\cdots,E_N$ is a reduced basis of $(X,\omega)$. Furthermore, for any $i=1,2,\cdots, N$, 
$\omega(H)=3\omega(E_i)$, and each $E_i$ has the minimal symplectic area in $(X,\omega)$. Consequently, 
for any two standard bases $H,E_1,E_2,\cdots,E_N$ and $H^\prime,E_1^\prime,E_2^\prime,\cdots,E_N^\prime$
of $(X,\omega)$, there is an automorphism $\tau$ of $H^2(X)$ which is a product of finitely many reflections $R(\gamma_s)$, where $\gamma_s=E_i-E_j$ or $H-E_i-E_j-E_k$, and  $\omega(\gamma_s)=0$, $\forall s$, 
such that $H^\prime,E_1^\prime,E_2^\prime,\cdots,E_N^\prime$ is transformed to $H,E_1,E_2,\cdots,E_N$ 
under $\tau$. See \cite{CLW}, Section 2.2, for a more detailed discussion. 

With the preceding understood, we have 

\begin{lemma}
Let $H,E_1,E_2,\cdots,E_N$ and $H^\prime,E_1^\prime,E_2^\prime,\cdots,E_N^\prime$ be two
distinct reduced bases of $(X,\omega)$. Then they must be related by an automorphism of $H^2(X)$
which is a product of finitely many reflections $R(\gamma_s)$, where $\gamma_s=E_i-E_j$ or 
$H-E_i-E_j-E_k$, and $\omega(\gamma_s)=0$ for all $s$. Moreover, one of the $(-2)$-classes $\gamma_s$
must be of the form $H-E_i-E_j-E_k$. As a consequence, $H,E_1,E_2,\cdots,E_N$ and $H^\prime,E_1^\prime,E_2^\prime,\cdots,E_N^\prime$ are related by a symplectomorphism of
$(X,\omega)$.
\end{lemma}

\begin{proof}
Our strategy of proof is by an induction on $N$, assuming $N\geq 3$. More precisely, if $E_N^\prime=E_N$, we symplectically blow down $(X,\omega)$ along $E_N$. Then $H,E_1,E_2,\cdots,E_{N-1}$ and $H^\prime,E_1^\prime,E_2^\prime,\cdots,E_{N-1}^\prime$ 
naturally descend to reduced bases of the 
blow-down manifold, which are obviously distinct as well (cf. Lemma 4.2 in \cite{C1}).

Suppose $E_N^\prime\neq E_N$. Consider the first possibility that $E_N^\prime=E_m$ for some 
$m<N$. In this case, since  $E_N^\prime$ has the minimal area, we have $E_m, E_{m+1}, \cdots,E_N$
all have the same area. If we apply the reflection $R(E_m-E_N)$, where $\omega(E_m-E_N)=0$, 
to the reduced basis $H,E_1,E_2,\cdots,E_N$, the classes $E_m$ and $E_N$ are switched. In this way, we arrive at the condition $E_N^\prime=E_N$ to run the induction process.

For the remaining possibility where $E_N^\prime\neq E_m$ for any $m$, with $N\geq 3$,
we recall Lemma 2.6 of \cite{CLW}, which says that either $(X,\omega)$ is monotone, or otherwise there is a $j>1$ such that $E_N^\prime=H-E_1-E_j$. If $(X,\omega)$ is monotone, then the lemma is trivially true. Assuming the latter case, we note that if $j<N$, we can apply the reflection 
$R(H-E_1-E_j-E_N)$, observing that $\omega(H-E_1-E_j-E_N)=0$, to the reduced basis $H,E_1,E_2,\cdots,E_N$, so that the last class $E_N$ is changed to $H-E_1-E_j=E_N^\prime$. If $j=N$, we shall apply $R(H-E_1-E_{N-1}-E_N)$ to the reduced basis $H,E_1,E_2,\cdots,E_N$. 
(Note that $\omega(H-E_1-E_{N-1}-E_N)=0$, 
because $\omega(H-E_1-E_N)=\omega(E_N^\prime)\leq \omega(E_{N-1})$.)
The resulting reduced basis has the last two classes being $H-E_1-E_N$ and $H-E_1-E_{N-1}$
respectively. With this understood, we apply $R(E_{N-1}-E_N)$ to it to switch $H-E_1-E_N$ with 
$H-E_1-E_{N-1}$ (note that $\omega(E_{N-1}-E_N)=0$). Then the resulting reduced basis has 
the last class being $H-E_1-E_N=H-E_1-E_j=E_N^\prime$. Hence if $(X,\omega)$ is not monotone, we can always arrange so that $E_N^\prime=E_N$ to run the induction on $N$.

Now suppose $N=3$, and after some arrangement, $E_N^\prime=E_N$. Blowing down $X$ along $E_N$, we get $\C\P^2\# 2\overline{\C\P^2}$, with a pair of bases $H,E_1,E_2$ and $H^\prime,E_1^\prime,E_2^\prime$. We claim $H,E_1,E_2$ and $H^\prime,E_1^\prime,E_2^\prime$ are the same up to switching the order of $E_1,E_2$.
To see this, the key observation is that for $\C\P^2\# 2\overline{\C\P^2}$, the only $(-1)$-spheres 
whose intersection with the canonical class equals $-1$ are $E_1,E_2$, and $H-E_1-E_2$, and
moreover, $H-E_1-E_2$ intersects with both $E_1,E_2$ nontrivially. It follows easily that
$E_1^\prime,E_2^\prime$ must be $E_1,E_2$ up to a change of order, and $H=H^\prime$. 

Finally, we note that each of the classes $\gamma$ can be represented by a Lagrangian
sphere $L$ (cf. \cite{LW}). Moreover, the Dehn twist $\tau_L$ associated to $L$ is a symplectomorphism realizing $R(\gamma)$ (cf. \cite{Seidel}). This finishes the proof of the lemma.

\end{proof}

\subsection{Freedom to specify the areas and finiteness of homological assignments}

Next, we assume $D=\cup_{k=1}^n F_k\subset X$ is symplectic with respect to a symplectic structure $\omega_0$ on $X$, which has canonical class $c_1(K_X)$. We will address the issue of freedom of choosing the areas of the surfaces $F_k$ in $D$ (by altering the symplectic structure $\omega_0$ if necessary), under the additional assumption (\ddag) in Section 1. Also recall the cone $C_\delta$ from Section 1, where $Q=(\nu_{kl})$ with 
$\nu_{kl}:=F_k\cdot F_l$:
\begin{itemize}
\item if $Q$ is negative definite, then $C_\delta=\{\vec{\delta}\in\R^n|\vec{\delta}\geq 0\}$, 
\item if $D$ is connected and $Q$ is non-singular and non-negative definite, then
$$
C_\delta=\{\vec{\delta}\in\R^n|\vec{\delta}\geq 0 \mbox{ and } Q^{-1}\vec{\delta}\geq 0\}.
$$
\end{itemize}

\begin{lemma}
Under the assumption $(\ddag)$,  for any interior point $\vec{\delta}=(\delta_k)\in C_\delta$, there exists a symplectic structure $\omega$ on $X$, with respect to which $D$ is symplectic, such that the canonical class of
$\omega$ is $c_1(K_X)$ and $\omega(F_k)=\delta_k$ for $k=1,2,\cdots,n$. In particular, $Z(\vec{\delta})\neq \emptyset$ (see Definition 1.4 for the definition of $Z(\vec{\delta})$).
\end{lemma}

\begin{proof}
The case where $Q$ is negative definite is essentially proved in Lemma 4.1 of \cite{C}. It is shown 
there that for any interior point $\vec{\delta}=(\delta_k)\in C_\delta$, there exists an
$\epsilon_0>0$ sufficiently small, and a symplectic structure $\omega$ on $X$ such that the canonical class of 
$\omega$ is $c_1(K_X)$ and $\omega(F_k)=\epsilon_0\delta_k$ for $k=1,2,\cdots,n$. 
Simply change $\omega$ to $\epsilon_0^{-1}\omega$. 

For the case where $D$ is connected and $Q$ is non-singular and not negative definite, the proof is similar in strategy. By our assumption, $D\subset X$ is symplectic with respect to a symplectic structure $\omega_0$. By Theorem 1.3 of \cite{LM}, one can deform 
$\omega_0$ to a symplectic structure $\omega_1$ such that $D$ is symplectic with respect to 
$\omega_1$, and there is a regular neighborhood $U$ of $D\subset X$ such that $\partial U$ 
is a concave contact boundary of $(U,\omega_1)$. Now given any interior point 
$\vec{\delta}=(\delta_k)\in C_\delta$, i.e., $\vec{\delta}>0$ and $Q^{-1}\vec{\delta}>0$, 
it is shown in \cite{LM} (see Sec. 2.1.1 of \cite{LM}) that there is a regular neighborhood 
$U^\prime$ of $D$ and a symplectic structure $\omega^\prime$ on $U^\prime$ such that 
$\partial U^\prime$ is a concave contact boundary of $(U^\prime,\omega^\prime)$ and 
$\omega^\prime(F_k)=\delta_k$ for each $k$. Moreover, by Theorem 1.7 of \cite{LM}, the contact structures on $\partial U$ and $\partial U^\prime$ are contactomorphic. Since the contact boundaries 
$\partial U$ and $\partial U^\prime$ are concave, there exists a $C_0>0$ sufficiently large, such that
one can remove $U$ from $X$ and then glue back $U^\prime$ contactomorphically, to obtain a
symplectic structure $\omega$ on $X$, such that $\omega(F_k)=C_0\delta_k$ for 
$k=1,2,\cdots,n$. In addition, as we argued in the proof of Lemma 4.1 in \cite{C}, 
one has the canonical class of $\omega, \omega_1$ being the same, which is $c_1(K_X)$.
To finish the proof, one simply replace $\omega$ by $C_0^{-1}\omega$. 

\end{proof} 

\begin{definition}
A homological assignment $(\vec{v}_k)\in\hat{\Omega}(D)$ is called {\bf area-robust} if for 
any interior point $\vec{\delta}$ of ${C}_\delta$, there is a $\vec{\lambda}=(\lambda_0,\lambda_1,\cdots,\lambda_N)^T\in\R^{N+1}$, such that
$$
\I_{(\vec{v}_k)}\vec{\lambda}=\vec{\delta}, \mbox{ where }\vec{\lambda}\in C_\lambda, \; \vec{\lambda}>0, 
\mbox{ and } \lambda_0^2-\sum_{i=1}^N\lambda_i^2>0.
$$
Here $\I_{(\vec{v}_k)}$ is the associated matrix of $(\vec{v}_k)$.
\end{definition}

It is easy to see that for an area-robust homological assignment $(\vec{v}_k)\in\hat{\Omega}(D)$, no matter how to choose the areas $\vec{\delta}\in {C}_\delta$, it is always possible that $(\vec{v}_k)$ is realized under 
$\vec{\delta}$. In other words, $(\vec{v}_k)$ cannot be eliminated by specifying the areas of the $F_k$'s. 

The following is a useful criterion for area-robustness. 

\begin{lemma}
Let $\I$ be the associated matrix of a homological assignment. If there is a vector 
$\vec{x}=(x_0,x_1,\cdots,x_N)^T\in \R^{N+1}$ in the null space of $\I$ such that $\vec{x}$ lies in the interior of the cone $C_\lambda$ and $x_0^2-\sum_{i=1}^N x_i^2>0$, then the homological assignment must be area-robust. 
\end{lemma}

\begin{proof}
Note that under the assumption (\ddag), the intersection matrix $Q$ of $D$ is non-singular, which implies that the matrix $\I$ must be of rank $n$. Hence for any $\vec{\delta}\in\R^n$, there is a $\vec{\eta}\in \R^{N+1}$ such that
$\I\vec{\eta}=\vec{\delta}$. Now choose a constant $C>0$ sufficiently large, we have 
$\I(\vec{\eta}+C\vec{x})=\vec{\delta}$, $\vec{\eta}+C\vec{x}$ lies in the interior of
the cone $C_\lambda$, and the entries of 
$$
\vec{\eta}+C\vec{x}=(\lambda_0,\lambda_1,\lambda_2,\cdots,\lambda_N)^T
$$ 
obey the constraint $\lambda_0^2-\sum_{i=1}^N\lambda_i^2>0$. This proves the area-robustness of the homological assignment. 

\end{proof}

\begin{example}
(1) Let $X=\C\P^2\# 12 \overline{\C\P^2}$, and let $F_1,F_2,\cdots,F_9$ be $9$ symplectic 
$(-3)$-spheres disjointly embedded in $X$. We fix a reduced basis $H,E_1,E_2,\cdots,E_{12}$.

Consider the following potential homological expression for $F_1,F_2,\cdots,F_9$:
\begin{itemize}
\item $H-E_{i}-E_{r}-E_{s}-E_t$, $H-E_{i}-E_u-E_v-E_w$, $H-E_{i}-E_x-E_y-E_z$,
\item $H-E_{j}-E_{r}-E_{u}-E_x$, $H-E_{j}-E_s-E_v-E_y$, $H-E_{j}-E_t-E_w-E_z$,
\item $H-E_{k}-E_{r}-E_{v}-E_z$, $H-E_{k}-E_s-E_w-E_x$, $H-E_{k}-E_t-E_u-E_y$,
\end{itemize}
and let $(\vec{v}_k)$ be the corresponding homological assignment. It is easy to see that 
$\vec{x}=(4,1,1,1,\cdots,1)$ is in the null space of the associated matrix $\I_{(\vec{v}_k)}$ of $(\vec{v}_k)$, 
as the corresponding homology class in $H^2(X)$, i.e., $4H-E_1-E_2-\cdots-E_{12}$, intersects trivially with the homological expression of each $F_k$. Furthermore, $\vec{x}$ lies in the interior of the cone $C_\lambda$, and satisfies the inequality $x_0^2-\sum_{i=1}^N x_i^2>0$ (which is $4^2-12>0$). This shows that the homological assignment $(\vec{v}_k)$ is area-robust. In other words, the potential homological expression for $F_1,F_2,\cdots,F_9$ cannot be eliminated by any choice of the areas of the $F_k$'s.

(2) Let $X=\C\P^2\# 7 \overline{\C\P^2}$. Suppose there are $7$ symplectic $(-2)$-spheres 
$F_1,F_2,\cdots,F_7$ disjointly embedded in $X$. We fix a reduced basis $H,E_1,E_2,\cdots,E_{7}$.

Consider the following potential homological expression of $F_1,F_2,\cdots,F_7$:
\begin{itemize}
\item $H-E_{l_1}-E_{l_2}-E_{l_3}$, $H-E_{l_1}-E_{l_4}-E_{l_5}$, $H-E_{l_1}-E_{l_6}-E_{l_7}$, 
\item $H-E_{l_2}-E_{l_4}-E_{l_6}$, $H-E_{l_3}-E_{l_5}-E_{l_6}$, $H-E_{l_2}-E_{l_5}-E_{l_7}$, 
\item $H-E_{l_3}-E_{l_4}-E_{l_7}$.
\end{itemize}
It is easy to see that $\vec{x}=(3,1,1,\cdots,1)$ is in the null space of the associated matrix $\I$,
as the corresponding homology class in $H^2(X)$, i.e., $3H-E_1-E_2-\cdots-E_{7}$, equals 
$-c_1(K_X)$, hence intersects trivially with each symplectic $(-2)$-sphere $F_k$. 
Note that $\vec{x}$ satisfies the inequality 
$x_0^2-\sum_{i=1}^N x_i^2>0$ (which is $3^2-7>0$), and $\vec{x}$ lies in the cone $C_\lambda$. However, 
$\vec{x}$ does not lie in the interior of $C_\lambda$. It turns out that this potential homological expression of $F_1,F_2,\cdots,F_7$ can be eliminated by a certain choice of the areas of the $F_k$'s.

\end{example}

We should point out that when $N\leq 8$, the area-robustness of $(\vec{v}_k)$ is reduced to the condition
${C}_\delta\subseteq \I_{(\vec{v}_k)}(C_\lambda)$, as the constraint $\lambda_0^2-\sum_{i=1}^N\lambda_i^2>0$ becomes redundant by the following lemma.

\begin{lemma}
Suppose $N\leq 9$. Then for any $\vec{\lambda}=(\lambda_0,\lambda_1,\cdots,\lambda_N)^T\in C_\lambda$ where $\vec{\lambda}>0$, one has $\lambda_0^2-\sum_{i=1}^N \lambda_i^2\geq 0$,
with $"="$ if and only if $N=9$ and $\vec{\lambda}=(\lambda_0,\lambda_0/3, \cdots, \lambda_0/3)^T$.
\end{lemma}

\begin{proof}
Without loss of generality, we assume $\lambda_1\geq \lambda_2\geq \cdots\geq \lambda_N$. Then
as $\lambda_0\geq \lambda_1+\lambda_2+\lambda_3$ and $N\leq 9$, we have
$$
\lambda_0^2\geq \lambda_1^2+\lambda_2^2+\lambda_3^2+2\lambda_1\lambda_2+2\lambda_1\lambda_3+2\lambda_2\lambda_3\geq \sum_{i=1}^N \lambda_i^2.
$$
Furthermore, it is easy to see that the equality holds in the above inequalities if and only if $N=9$ and 
$\vec{\lambda}=(\lambda_0,\lambda_0/3, \cdots, \lambda_0/3)^T$.

\end{proof}

Finally, we address the issue of finiteness of the set $\Omega(D)$, which is not guaranteed when
$N\geq 9$. In particular, we shall give a proof for Theorem 1.6.

The following lemma allows us to trade some freedom of choosing the areas of the $F_k$'s
for an upper bound on the $a$-coefficients of the homology classes of the $F_k$'s. 

\begin{lemma}
Let $H,E_1,E_2,\cdots,E_N$ be a reduced basis of $(X,\omega)$, and let 
$A=aH-\sum_{i=1}^N b_iE_i$ be the class of an embedded symplectic surface of genus $g$ 
and self-intersection $-\alpha$ in $(X,\omega)$, where $-2\leq \alpha+2g-2\leq 1$. We denote by $K_\omega$
the canonical line bundle associated to $\omega$.

{\em(1)} Assume $\omega(A)<-c_1(K_\omega)\cdot [\omega]$ and $a>3$. Then $g=0$, and $A$ must be of the following form
$$
A=aH-(a-1)E_{j_1}-E_{j_2}-\cdots- E_{j_{2a+\alpha}}.
$$
In particular, $a\leq \frac{1}{2}(N-\alpha)$. 

{\em (2)} Assume $2\omega(A)<-c_1(K_\omega)\cdot [\omega]$. Then $a\leq 3$.
\end{lemma}

Note that in the above lemma, with $g=0$, the condition $-2\leq \alpha+2g-2\leq 1$ is equivalent to  
$\alpha=0,1,2,3$. On the other hand, we note that when $g=0$ and $a\leq 3$, the class $A$ is automatically in the form specified in (1), i.e., one of the $b_i$-coefficient of $A$ equals $a-1$ and the rest are either $1$ or $0$, even without imposing the area condition $\omega(A)<-c_1(K_\omega)\cdot [\omega]$.

\begin{proof}
Part (1) of the lemma is an extension of Lemma 3.6 in \cite{C}, with the same proof strategy. (We will use the same notations here.) First, the following key estimate continues to hold:
$$
\sum_{i=1}^N(b^{+}_i-1)\leq 3(a-3)
$$ 
where $b^{+}_i:=\max (1,b_i)$. To see this, note that the assumption that 
$-2\leq \alpha+2g-2\leq 1$ implies that Lemma 2.6 is applicable here, and moreover, 
$$
\delta:=\max \{0, \; 1-(\alpha+2g-2)\}=1-(\alpha+2g-2)
$$ 
in Lemma 2.6. With this understood, we have 
$$
M\geq 10-\delta=9+(\alpha+2g-2),
$$
where $M$ is the number of non-zero $b_i$-coefficients in $A$. It follows easily that
$$
\sum_{i=1}^N(b^{+}_i-1)=\sum_{i=1}^N b_i-M=3a+(\alpha+2g-2)-M\leq 3(a-3)
$$
as claimed, where $\sum_{i=1}^N b_i=3a+(\alpha+2g-2)$ by the adjunction formula. We also used the fact that
$b_i\geq 0$ because $a>0$ and $A$ is admissible (cf. Lemma 2.3). 
Now by the same argument as in Lemma 3.6 of \cite{C}, the assumption
$\omega(A)<-c_1(K_\omega)\cdot [\omega]$ implies that there is a $b_i$ such that $b_i=a-1$.
With $(a-1)(a-2)=\sum_{i=1}^N b_i(b_i-1)+2g$, it follows easily that $g=0$, and the rest of the 
$b_i$'s are either $0$ or $1$. The rest of the proof is the same as in Lemma 3.6 of \cite{C}.

For part (2), assume to the contrary that $a\geq 4$. Note that $2\omega(A)<-c_1(K_\omega)\cdot [\omega]$ implies that $\omega(A)<-c_1(K_\omega)\cdot [\omega]$, so that the conclusion of part
(1) of the lemma holds true. With this understood, we note that
$$
2A+c_1(K_\omega)=(2a-3)H-(2a-3)E_{j_1}-E_{j_2}-\cdots- E_{j_{2a+\alpha}}+E_{j_{2a+\alpha+1}}+\cdots+E_{j_N}.
$$
Furthermore, observe that $2a+\alpha-1\leq 2(2a-3)$ as $\alpha\leq 3$ and $a\geq 4$,
which implies that $2A+c_1(K_\omega)$ can be written as a sum of terms of the form 
$H-E_i-E_j-E_k$ or $E_i$. It follows that $2\omega(A)+c_1(K_\omega)\cdot [\omega]\geq 0$,
which is a contradiction.

\end{proof}

\noindent{\bf Proof of Theorem 1.6:}

\vspace{2mm}

To ease the notations, let $I=\{1,2,\cdots,n\}$ be the index set for the index $k$. Furthermore, recall from Section 1 that
$$
C^\ast_0:=\{\vec{\delta}\in \R^n| \delta_k \leq -\sum_{l=1}^n c_l\delta_l, \; \forall k\in I_0\} 
\mbox{ and }
C^\ast_1:=\{\vec{\delta}\in \R^n| 2 \delta_k \leq -\sum_{l=1}^n c_l\delta_l, \; \forall k\in I^\ast\}.
$$
With this understood, 
let $(\vec{v}_k)\in \Omega(D)$, where $\vec{v}_k=(a_k, b_{k1}, b_{k2}, \cdots, b_{kN})$, be an element which is realized under an interior point $\vec{\delta}\in C^\ast_0 \cap C_\delta$. What this means is that there is
an $\omega\in Z(\vec{\delta})$, such that for a reduced basis $H,E_1,E_2,\cdots,E_N$ of $(X,\omega)$, the assignment $F_k\mapsto A_k:=a_kH-\sum_{i=1}^N b_{ki}E_i$ is a homological expression of $D$. 
As $\vec{\delta}\in C^\ast_0 \cap C_\delta$ is an interior point, Lemma 2.15(1) implies that, for any index $k\in I_0$, $a_k\leq \max(3, \frac{1}{2}(N+F_k^2))$ must be true. 

On the other hand, as the $a$-coefficient of $c_1(K_X)$ equals $-3$, it follows easily that
$$
\sum_{k\in I\setminus I_0} -c_k a_k\leq 3+\sum_{k\in I_0} c_k \cdot \max(3,\frac{N+F_k^2}{2}),
$$
as $c_k\geq 0$ for $k\in I_0$. We observe that $c_k<0$ for $k\in I\setminus I_0$, and moreover, 
for any $k$, if $a_k<0$, then $|a_k|\leq \frac{1}{2}(-F_k^2-1)$ (cf. Lemma 2.4), because by assumption, 
$\vec{v}_k=(a_k, b_{k1}, b_{k2}, \cdots, b_{kN})$ is admissible. It follows easily that for each index $k\in I\setminus I_0$, $a_k$ is bounded from above by a constant $C_k>0$ depending only on $N$, the self-intersections of $F_1,F_2,\cdots, F_n$, and the constants $c_1,c_2,\cdots,c_n$. Finally, for each $k\in I_0$, we set $C_k:= \max(3, \frac{1}{2}(N+F_k^2))$. It follows immediately that $(\vec{v}_k)\in \Omega(D, \underline{C})$ where $\underline{C}=(C_k)$. 
This proves the first part of the theorem. 

Next assume that we fix a subset $I^\ast$ such that $I^\ast\subseteq I_1$, and we choose an interior point 
$\vec{\delta}\in C^\ast_1\cap C^\ast_0 \cap C_\delta$. Then it follows easily from Lemma 2.15(2) that $a_k\leq 3$
for any $k\in I^\ast$, so that we can set the bound $C_k:=3$ in $\underline{C}=(C_k)$ for any $k\in I^\ast$. 
This finishes the proof of Theorem 1.6.

%\vspace{2mm}
\begin{remark}
We note that, in obtaining an estimate for the upper bound $C_k$, $k\in I\setminus I_0$ or $k\in I\setminus I^\ast$, there is at most one $F_k$ (here $k\in I$) such that $a_k<0$ (cf. \cite{C}, Lemma 4.2(1)). On the other hand, depending on the concrete situation, one has the flexibility to impose some extra constraint 
$C^\ast_1=\{\vec{\delta}\in \R^n| 2\delta_k \leq -\sum_{l=1}^n c_l\delta_l, \; \forall k\in I^\ast\}$ in selecting the area
vector $\vec{\delta}$, for some chosen subset $I^\ast\subseteq I_1$, in order to improve the upper bound $C_k$ to $C_k=3$ for $k\in I^\ast$.
\end{remark}

\section{Symplectic Cremona transformations}

\subsection{Successive symplectic blowing-down revisited}
Suppose $D=\cup_{k=1}^n F_k$ is a symplectic configuration in $(X_N,\omega_N)$, where $X_N=\C\P^2\# N\overline{\C\P^2}$. 

Let $F_k\mapsto A_k=a_k H-\sum_{i=1}^N b_{ki} E_i$ be a given homological expression of $D$, where $H,E_1,E_2,\cdots, E_N$ is a reduced basis of 
$(X_N,\omega_N)$. In \cite{C1} we introduced a successive symplectic blowing-down procedure, which, under suitable assumptions on the homological expression $F_k\mapsto A_k=a_k H-\sum_{i=1}^N b_{ki} E_i$ and the symplectic structure $\omega_N$, successively and symplectically blows down $X_N$ to $X_1=\C\P^2\#\overline{\C\P^2}$, and under
additional assumptions, further blows down $X_1$ to $\C\P^2$: 
$$
(X_N,\omega_N)\rightarrow (X_{N-1},\omega_{N-1})\rightarrow \cdots \rightarrow (X_m,\omega_m)
\rightarrow \cdots 
$$
(We shall say that the procedure is at stage $m$ if we reach $(X_m,\omega_m)$ under the successive blowing-down.) In the process, it transforms the configuration $D$ into a so-called {\bf symplectic arrangement} $\hat{D}$ in $X_1=\C\P^2\#\overline{\C\P^2}$ or $\C\P^2$, where $\hat{D}$ is a union of pseudoholomorphic curves, whose singularities and intersection pattern are canonically determined by the homological expression $F_k\mapsto A_k=a_k H-\sum_{i=1}^N b_{ki} E_i$.  Furthermore, this procedure is reversible, meaning that there is a successive blowing-up procedure with reversing order, which recovers the configuration $D$ from the symplectic arrangement $\hat{D}$ up to a smooth isotopy. We remark that even though this is purely a symplectic operation and there is no holomorphic analog of it, in analogy if the successive blowing-down of $X_N$ to $X_1=\C\P^2\#\overline{\C\P^2}$ or $\C\P^2$ were given by a birational morphism and $D$ is a configuration of irreducible curves in $X_N$, $\hat{D}$ would correspond to the direct image of $D$. (Compare the proof of Lemma 2.3(1) and (2).)

%4\vspace{2mm}

First, we shall give an overview of the procedure, explaining its main points and features. The starting point is the fact that the configuration $D$ and its descendant at each stage of the blowing-down can be made $J$-holomorphic for some compatible almost complex structure $J$, while for each 
$2\leq m\leq N$, the class $E_m$ at stage $m$ can always be represented by a $J$-holomorphic 
$(-1)$-sphere for any given $J$. (Here to include the case of $m=2$, we have to impose a technical condition that the symplectic structure $\omega_N$ is odd, meaning that the area $\omega_N(H-E_1-2E_2)\geq 0$.) The main issue is how to construct the descendant of $D$ at the next stage
after blowing down $E_m$. 

To explain this, we let $D_m\subset X_m$ be the descendant of $D$ at stage $m$, which is $J_m$-holomorphic, and $C_m$ be the $J_m$-holomorphic $(-1)$-sphere representing $E_m$. Recall from \cite{C1} that
in order to blow down $(X_m,\omega_m)$ symplectically, we cut $X_m$ open along $C_m$ and insert a standard symplectic $4$-ball of appropriate size, to be denoted by $B(\hat{E}_m)$ where 
$\hat{E}_m$ stands for the center of the ball. With this understood, constructing
the descendant of $D$ at the next stage, i.e., stage $m-1$, boils down to the question of 
how to extend $D_m\setminus C_m$ across the $4$-ball $B(\hat{E}_m)$, as $D_m$ may intersect $C_m$, and the answer depends on whether $C_m$ is part of $D_m$ or not.

If $C_m$ is not part of $D_m$, we shall slightly perturb $C_m$ if necessary ($C_m$ continues to be a smoothly embedded symplectic $(-1)$-sphere), so that it intersects $D_m$ only at its nonsingular locus, with transverse and positive intersections. With this understood, we extend $D_m\setminus C_m$ to $B(\hat{E}_m)$ by adding to each puncture of $D_m\setminus C_m$ a complex linear disk in 
the $4$-ball $B(\hat{E}_m)$. (Note that the disks only intersect at the center $\hat{E}_m$.)

Now assume $C_m$ is part of $D_m$. Note that this occurs if and only if there is one and unique 
component of $D$, denoted by $S$, whose homological expression takes the form 
$$
S=E_m-E_{l_1}-E_{l_2}-\cdots-E_{l_\alpha}, \mbox{ where $m<l_s$ for all $s$}. 
$$
(We call $E_m$ the {\bf leading class} of $S$.) In this case, we can no longer perturb $C_m$ before blowing it down, in order for this procedure to be reversible. With this understood, it is necessary 
that $D_m$ is described by a certain symplectic model near each intersection point of $C_m$ with other components of $D_m$. More concretely, let $x$ be such an intersection point. Then the 
model is as follows: in a Darboux neighborhood of 
$(X_m,\omega_m)$ centered at $x$, there are complex linear coordinates $w_1,w_2$, such that $C_m$ is given by $w_2=0$ and any other component of $D_m$ is given by one of the following equations, $w_1=0$, or $w_2=aw_1$ for some $0\neq a\in\C$, or $w_2^p=aw_1^q$ where $0\neq a\in\C$ and $pq>1$. With this understood, the extension of the corresponding component of $D_m\setminus C_m$ in the $4$-ball $B(\hat{E}_m)$, after blowing down $C_m$, is given, respectively, by $z_1=0$, or $z_1=bz_2^2$ for some $0\neq b\in\C$, or $z_1^q=bz_2^{p+q}$ for some $0\neq b\in\C$, where $z_1,z_2$ are some complex linear coordinates on the $4$-ball $B(\hat{E}_m)$ (cf. \cite{C1}, Lemma 4.4). We remark that the equations of type $w_2^p=aw_1^q$, where $pq>1$, are not preserved under a general linear transformation of $w_1,w_2$, so in general, in the symplectic model above, the axes $w_1=0$, $w_2=0$ (resp. $z_1=0$, $z_2=0$) are uniquely determined up to order. 

With the preceding understood, suppose there is a component of $D$, called $\tilde{S}$, which has zero $a$-coefficient and contains the class $E_m$ in its homological expression. It is easy to see
that the descendant of $\tilde{S}$ in $D_m$ must intersect the $4$-ball $B(\hat{E}_m)$ in a complex linear disk. On the other hand, if $E_{\tilde{m}}$ is the leading class of $\tilde{S}$, then we note that 
$\tilde{m}<m$, and at the stage $\tilde{m}$ of the blowing-down, the $(-1)$-sphere $C_{\tilde{m}}$
representing the class $E_{\tilde{m}}$ will be part of $D_{\tilde{m}}$; in fact, it is the descendant of
$\tilde{S}$ in $D_{\tilde{m}}$. In order to apply the aforementioned symplectic model when we blow down the $(-1)$-sphere $C_{\tilde{m}}$, it is clear that the descendant of $\tilde{S}$ in $D_m$ has to be given by one of the two axes $z_1=0$, $z_2=0$ in $B(\hat{E}_m)$. With this understood, one can show that there are at most two such components $\tilde{S}$ for each given $S$ (cf. \cite{C1}, Lemma 4.5). Nevertheless, this requirement puts certain restrictions on how other components of $D_m$ are allowed to intersect $C_m$. Under the following assumptions (for each given $S$, and assuming
$\omega_N$ is odd), it is shown in \cite{C1} that one can maintain the symplectic models at each stage of the blowing-down, and as a result, successively blows down
$X_N$ to $X_1=\C\P^2\#\overline{\C\P^2}$:
\begin{itemize}
\item [{(a)}] Suppose there are two symplectic spheres $S_1,S_2\subseteq D$ whose $a$-coefficients equal zero and whose homological expressions contain the leading class $E_m$ of $S$. Then for any class $E_{l_s}$ which appears in $S$, but appears in neither $S_1$ nor $S_2$, there is at most one component $F_k$ of $D$ other than $S$, whose homological expression contains $E_{l_s}$ with
$F_k\cdot E_{l_s}=1$.
\item [{(b)}] Suppose there is only one symplectic sphere $S_1\subseteq D$ whose $a$-coefficient 
equals zero and whose homological expression contains the leading class $E_m$ of $S$. 
Then there is at most one class $E_{l_s}$ in $S$, which does not appear in $S_1$, but either appears in the expressions of more than one components $F_k\neq S$, or appears in the expression of only one component $F_k\neq S$ but with $F_k\cdot E_{l_s}>1$.
\end{itemize}
Furthermore, under one of the following additional assumptions, one can further blow down the class $E_1$ and reduce $X_1=\C\P^2\#\overline{\C\P^2}$ to $\C\P^2$:
\begin{itemize}
\item [{(c)}] The classes $E_1,E_2$ have the same area, i.e., $\omega_N(E_1)=\omega_N(E_2)$.
\item [{(d)}] The class $E_1$ is the leading class of a component of $D$. 
\item [{(e)}] There is a component of $D$ with homological expression $aH-b_1E_1-\sum_{i>1}b_i E_i$ where $2b_1<a$.
\end{itemize}
(See \cite{C1}, Theorem 4.3.) End of the review. 

\vspace{2mm}

The successive blowing-up procedure, which recovers $D$ from $\hat{D}$ up to a smooth isotopy, is
based on the following construction, adapted from \cite{CFM} (see also \cite{C3}).

\begin{lemma}
Let $(M,\omega)$ be a symplectic $4$-manifold, $D$ be a union of $J$-holomorphic curves in $M$ where $J$ is $\omega$-compatible. Let $p\in D$, and suppose in a neighborhood $U$ of $p$, $J$ is
integrable and $\omega$ is K\"{a}hler. Let $\pi: \tilde{M}\rightarrow M$ be the complex blow-up at $p$ defined using the complex structure $J$ near $p$. Denote by $C\subset \tilde{M}$ the exceptional 
$(-1)$-sphere and by
$\tilde{D}$ the proper transform of $D$ in $\tilde{M}$. Then there exist a symplectic structure 
$\tilde{\omega}$ on $\tilde{M}$ and a $\tilde{\omega}$-compatible almost complex structure $\tilde{J}$ such that (i) $\tilde{J}$ is integrable and $\tilde{\omega}$ is K\"{a}hler near the exceptional sphere $C$, (ii) $(\tilde{\omega},\tilde{J})=(\omega,J)$ on $\tilde{M}\setminus \pi^{-1}(U)$, and (iii) 
$C$ and $\tilde{D}$ are $\tilde{J}$-holomorphic. Moreover, the area of the exceptional $(-1)$-sphere, 
$\tilde{\omega}(C)$, can take any value which is sufficiently small. 
\end{lemma}

\begin{proof}
Without loss of generality, we assume $U$ is a small ball centered at $p$, such that $D$ is embedded in $U\setminus \{p\}$. We let $J_0$ be the complex structure on $\pi^{-1}(U)$ (note that $J_0=J$ on
$\pi^{-1}(U)\setminus C$), and we fix a K\"{a}hler
form $\Omega$ on it. Then note that there is a $1$-form $\gamma$ on $\pi^{-1}(U)\setminus C$
such that $\Omega=d\gamma$. We pick a cut-off function $\rho$ on $\pi^{-1}(U)$,
which equals $1$ in a neighborhood $V$ of $C$ and equals $0$ near the boundary of $\pi^{-1}(U)$, and let $\epsilon>0$ be sufficiently small. Then on $\pi^{-1}(U)$, we define $\tilde{\omega}$ as follows:
$\tilde{\omega}=\pi^\ast\omega+\epsilon \Omega$ on $V$, which is K\"{a}hler with respect to
$J_0$, and $\tilde{\omega}:=\pi^\ast\omega+\epsilon d(\rho\gamma)$ on $\pi^{-1}(U)\setminus V$,
which is symplectic (as $\epsilon$ is sufficiently small) 
and equals $\pi^\ast\omega$ near the boundary of $\pi^{-1}(U)$, hence extends 
naturally to a symplectic structure on $\tilde{M}$, equalling $\omega$ on 
$\tilde{M}\setminus \pi^{-1}(U)$. 

To define $\tilde{J}$, we note that over the region where $\rho$ is non-constant, 
$J_0$ is $\tilde{\omega}$-tame and $\tilde{\omega}|_D>0$, and $D$ is embedded 
in $U\setminus \{p\}$. Let $h_0$ be the K\"{a}hler metric associated to $\tilde{\omega}$ and
$J_0$ whenever $J_0$ is $\tilde{\omega}$-compatible, and define $h$ by 
$h(X,Y):=\frac{1}{2}(\tilde{\omega}(X,J_0Y)+\tilde{\omega}(Y,J_0X))$. Then $h$ is a metric and 
$h=h_0$ whenever $h_0$ is defined. In particular, the $\tilde{\omega}$-compatible almost complex
structure determined by the metric $h$ equals $J_0$ whenever $J_0$ is $\tilde{\omega}$-compatible.
The problem is that the tangent bundle $TD$ may not be invariant under it. To deal with this issue,
we first define $\tilde{J}$ as an $\tilde{\omega}$-compatible almost complex structure on $D$ using the metric $h$, then extends it to the normal bundle of $D$ (still $\tilde{\omega}$-compatible) 
such that it equals $J_0$ outside a neighborhood of $\text{supp }\rho^\prime$. Let $h^\prime$ be the metric along $D$ defined by $h^\prime(X,Y):=\tilde{\omega}(X, \tilde{J}Y)$, and let $\tilde{h}$ be a metric which is an interpolation of $h^\prime$ and $h$. Then we can extends $\tilde{J}$ to the rest 
of $\pi^{-1}(U)$ using the 
$\tilde{\omega}$-compatible almost complex structure determined by the metric $\tilde{h}$. It is clear
that the conditions (i)-(iii) are satisfied by $(\tilde{\omega},\tilde{J})$. Finally, note that $\tilde{\omega}(C)=
\epsilon \Omega(C)$, which can take any value which is sufficiently small.

\end{proof}

With the preceding understood, we shall next remove the condition that $\omega_N$ is odd from
the assumptions of the successive blowing-down, by introducing a modified version of the assumptions (a) and (b). Recall that $\omega_N$ is odd if $\omega_N(H-E_1-2E_2)\geq 0$,
which means that when $\omega_N$ is even (i.e. not odd), the $(-1)$-class $H-E_1-E_2$
has the minimal area among the three classes $E_1$, $E_2$, and $H-E_1-E_2$. 

To state the modified version of (a) and (b), we first note that there is at most one component 
$F_k$ of $D$ with the following significance: the $a$-coefficient of $F_k$ equals $1$ and its homological expression contains both classes $E_1$ and $E_2$ (this is because $F_k\cdot F_l\geq 0$ for $k\neq l$). We shall denote such a component of $D$ by $\Sigma_0$ if it exists. With this understood, we observe that the same argument for the proof of Lemma 4.5 in \cite{C1} shows that there are at most two components $F_k$ of $D$ such that the leading class $E_m$ of $S$ is contained in the homological expression of $F_k$ and either $F_k=\Sigma_0$ or $F_k$ has 
$a$-coefficient $0$. With this understood, here is the modified version of (a) and (b).

\begin{itemize}
\item [{(a')}] Suppose there are two components $S_1,S_2\subset D$ where the homological expressions of $S_1$, $S_2$ contain the leading class $E_m$ of $S$ and either $S_1$, $S_2$ are 
$\Sigma_0$ and a symplectic sphere whose $a$-coefficient equals zero or both $S_1$, $S_2$ are 
a symplectic sphere whose $a$-coefficient equals zero. Then for any class $E_{l_s}$ which appears in $S$, but appears in neither $S_1$ nor $S_2$, there is at most one component $F_k$ of $D$ other than $S$, whose homological expression contains $E_{l_s}$ with $F_k\cdot E_{l_s}=1$.
\item [{(b')}] Suppose there is only one component $S_1\subset D$ where the homological expression of $S_1$ contains the leading class $E_m$ of $S$ and either $S_1=\Sigma_0$ or $S_1$ is a symplectic sphere with $a$-coefficient $0$. Then there is at most one class $E_{l_s}$ in the homological expression of $S$, which does not appear in $S_1$, but either appears in the expressions of more than one components $F_k\neq S$, or appears in the expression of only one component $F_k\neq S$ but with $F_k\cdot E_{l_s}>1$.
\end{itemize}

Note that the assumptions (a') and (b') imply the assumptions (a) and (b).

\begin{lemma}
Theorem 4.3 of \cite{C1} continues to be true without the assumption that $\omega_N$ is odd if either
$E_2$ is the leading class of a component of $D$ or the assumptions (a) and (b) are replaced by (a') and (b'). 
\end{lemma}

\begin{proof}
For simplicity, we shall first consider the case where the component $\Sigma_0$ does not exist in $D$. It is easy to see that in this case, the assumptions (a') and (b') boil down to (a) and (b), and we can simply proceed as in \cite{C1}, until we reach the stage $X_2=\C\P^2\# 2\overline{\C\P^2}$ of the successive blowing-down. We need to explain how to blow down the class $E_2$ 
when $\omega_N$ is even. (We shall continue to use the notations introduced in \cite{C1}, 
Section 4.)

Recall that the descendant $D_2$ of $D$ in $(X_2,\omega_2)$ is a union of $J_2$-holomorphic curves, where $J_2$ is some $\omega_2$-compatible almost complex structure. The key issue for blowing down the class $E_2$ is to represent it by a $J_2$-holomorphic $(-1)$-sphere (note that $E_2$ does not have the minimal area as 
$\omega_N$ is even, so the relevant result in \cite{KK} does not apply here). 
Once this is achieved, the rest is the same as in \cite{C1}. With this understood, if $E_2$ is the leading class of a component $S$ of $D$, then the descendant of $S$ in $D_2$ is a $J_2$-holomorphic 
$(-1)$-sphere representing $E_2$, and we are done in this case. 

Assuming $E_2$ is not the leading class of any component of $D$, we shall proceed as follows. First, since the class $H-E_1-E_2$ has the minimal area (as $\omega_N$ is even), we can represent it by a $J_2$-holomorphic 
$(-1)$-sphere $C$ (cf. \cite{KK}). Since we assume that $\Sigma_0$ does not exist in $D$, it follows easily that $C$ is not a component of $D_2$. With this understood, we can perturb $C$ slightly so that it intersects with each component of $D_2$ transversely and positively, and remains to be a symplectic $(-1)$-sphere. We symplectically blow down $X_2$ along $C$, and denote the resulting symplectic $4$-manifold by $(\hat{X},\hat{\omega})$. As we have already seen in the proof of Lemma 2.3, $\hat{X}$ is diffeomorphic to $\s^2\times \s^2$. Note that $D_2$ descends to
a union of $\hat{J}$-holomorphic curves in $(\hat{X},\hat{\omega})$, which is denoted by $\hat{D}$. 
Let $B(p)$ be the standard symplectic $4$-ball in $(\hat{X},\hat{\omega})$ resulted from
the symplectic blowing-down, with its center denoted by $p$.

As in the proof of Lemma 2.3, let $e_1, e_2\in H^2(\hat{X})$ be the descendant of $E_1,E_2$ respectively. Applying Lemma 2.4 of \cite{C2} to the class $e_2$, and with $\hat{\omega}(e_2)\leq \hat{\omega}(e_1)$, it follows easily that $e_2$ is represented by a $\hat{J}$-holomorphic sphere. 
In fact, $\hat{X}$ is foliated by a $\s^2$-family of such $\hat{J}$-holomorphic spheres. We denote 
by $\hat{C}_2$ the one which passes through the point $p$, i.e., the center of the symplectic 
$4$-ball $B(p)$ in $\hat{X}$. 

Now we apply Lemma 3.1 and holomorphically blow up $\hat{X}$ at $p$. We denote by $X_2^\prime$ the resulting manifold, and $J_2^\prime$ the almost complex structure. We let $D_2^\prime$ denote the proper transform of $\hat{D}$ in $X_2^\prime$, which is $J_2^\prime$-holomorphic. As we pointed out earlier, one can naturally identify $(X_2,D_2)$ with 
 $(X_2^\prime,D_2^\prime)$ smoothly. With this understood, we shall replace $(X_2,D_2)$ by
 $(X_2^\prime,D_2^\prime)$ in our argument. As a consequence, if we let $C_2^\prime$ be the proper transform of $\hat{C}_2$ in $X_2^\prime$, then $C_2^\prime$ is the $J_2^\prime$-holomorphic
$(-1)$-sphere representing the class $E_2$ that we are looking for. 

Under one of the additional assumptions (c), (d), or (e), one can further blow down $X_1$ to $\C\P^2$, in analogy to \cite{C1}. More concretely, the cases (d) and (e) are the same as in \cite{C1}; for case (c) where the classes $E_1$ and $E_2$ have the same area, we note that 
$\hat{\omega}(e_1)= \hat{\omega}(e_2)$, so that we can apply Lemma 2.4 of \cite{C2} to the class $e_1$ as well. This gives rise to a $\hat{J}$-holomorphic sphere $\hat{C}_1$ representing the class $e_1$ which contains the point $p$. The proper transform of $\hat{C}_1$ in $X_2^\prime$ is a $J_2^\prime$-holomorphic $(-1)$-sphere, denoted by $C_1^\prime$, which represents the class $E_1$. The $(-1)$-spheres $C_1^\prime$, $C_2^\prime$ are disjoint, so they can be blown down at the same time to reach the final stage $\C\P^2$. 

In the case where the component $\Sigma_0$ does exist in $D$, the idea of the proof is the same, with the argument slightly modified. More precisely, since $\Sigma_0$ is a component of $D$, the $J_2$-holomorphic $(-1)$-sphere $C$ which represents the class $H-E_1-E_2$ will be part of $D_2$, hence we can no longer perturb $C$ before blowing it down. However, the assumptions (a') and (b') ensure that we can still blow down $X_2$ along $C$, in the fashion explained in \cite{C1}, to reach 
to $\hat{X}=\s^2\times\s^2$. Let $\hat{D}$ be the descendant of $D$ in $\hat{X}$. We apply 
Lemma 3.1 to recover $(X_2,D_2)$ from $(\hat{X},\hat{D})$ in the fashion we explained in the earlier case where
$\Sigma_0$ does not exist. 
With this understood, the rest of the proof is the same as in the earlier case. 

\end{proof}

\subsection{A partial order of infinitely-nearness} 
For the rest of this section, we will focus on the case where the final stage of the 
successive blowing-down is $\C\P^2$. We will denote the symplectic structure on $\C\P^2$ by $\hat{\omega}$.
Then the symplectic arrangement $\hat{D}$ is a union of $\hat{J}$-holomorphic curves for some 
$\hat{\omega}$-compatible almost complex structure $\hat{J}$. (Both $\hat{\omega}$
and $\hat{J}$ are naturally resulted from the successive blowing-down, cf. \cite{C1}.)

The successive blowing-up procedure, which recovers $D$ from $\hat{D}$ up to a smooth isotopy,  
is simply an application of Lemma 3.1 at the points $\hat{E}_i$, $1\leq i\leq N$, successively and 
in a reversing order. (As for the notation, recall that $\hat{E}_i$ is the center of the standard symplectic $4$-ball $B(\hat{E}_i)$ equipped with the standard complex structure, which is inserted into $X_i\setminus C_i$ when we blow down $X_i$ along the $(-1)$-sphere $C_i$ representing the class $E_i$; in particular, 
$\hat{E}_i\in B(\hat{E}_i)\subset X_{i-1}$.)

More concretely, we apply Lemma 3.1 to $(\C\P^2,\hat{\omega})$ at the point
$\hat{E}_1$. We denote the resulting blow-up manifold by $(\tilde{X}_1,\tilde{\omega}_1)$ and the
$\tilde{\omega}_1$-compatible almost complex structure by $\tilde{J}_1$, and let $C_1$ be the
exceptional $(-1)$-sphere in $\tilde{X}_1$. With this understood, we define $\tilde{D}_1\subset 
\tilde{X}_1$ to be the proper transform of $\hat{D}$ if $E_1$ is not the leading class of a component of $D$, and define $\tilde{D}_1$ to be the union of $C_1$ with the proper transform 
of $\hat{D}$ if $E_1$ is the leading class of a component of $D$. After identifying $(X_1,D_1)$ with
$(\tilde{X}_1,\tilde{D}_1)$ smoothly, we apply Lemma 3.1 to $(\tilde{X}_1,\tilde{D}_1)$ at the point
$\hat{E}_2$, and define $(\tilde{X}_2,\tilde{\omega}_2)$, $\tilde{J}_2$, and $\tilde{D}_2$ in the same fashion. Inductively, we obtain $(\tilde{X}_N,\tilde{\omega}_N)$, $\tilde{J}_N$, and $\tilde{D}_N$.
For simplicity, we shall write $\tilde{J}$, $\tilde{D}$ for $\tilde{J}_N$, $\tilde{D}_N$. Then 
$(X_N,D)$ can be identified with $(\tilde{X}_N,\tilde{D})$ smoothly. Note that $\tilde{\omega}_N$ is different from $\omega_N$ in general, but the canonical classes are the same under the identification of $X_N$ with $\tilde{X}_N$.

With the preceding understood, observe that there is a natural smooth map 
$\pi: \tilde{X}_N\rightarrow \C\P^2$ which is smoothly equivalent to a holomorphic blowing-up. This allows us to introduce a notion of ``infinitely near" amongst the points $\hat{E}_i$, $1\leq i\leq N$, in the same way as in the complex analytic setting. However, we shall formulate it instead in terms of the classes $E_i$, $1\leq i\leq N$. 

\begin{definition}
Given a homological expression $F_k\mapsto A_k=a_k H-\sum_{i=1}^N b_{ki} E_i$ of $D$, where we assume that the final stage of the corresponding successive blowing-down procedure is $\C\P^2$, we can associate a partial order of infinitely-nearness on the $E_i$-classes to the homological expression 
$F_k\mapsto A_k=a_k H-\sum_{i=1}^N b_{ki} E_i$ as follows. 

For any class $E_i$, $1\leq i\leq N$, we say a class $E_j$ is {\bf infinitely near to $E_i$ of order $1$} if the point $\hat{E}_j$ is lying on the exceptional $(-1)$-sphere $C_i$ in the $i$-th blowup $\tilde{X}_i$.
Inductively, we say $E_j$ is {\bf infinitely near to $E_i$ of order $r$}, for $r>1$, if there is a class
$E_k$ which is infinitely near to $E_i$ of order $(r-1)$ and $E_j$ is infinitely near to $E_k$ of order $1$. 
When there is no need to mention the order $r$, we shall simply say that $E_j$ is infinitely near to $E_i$. 

It follows easily that the notion of ``infinitely near" defined above gives rise to a partial order $\leq$ on the classes $E_i$, $1\leq i\leq N$, where $E_i\leq E_j$ if either $E_i=E_j$, or $E_j$ is infinitely near to $E_i$, for which we 
write $E_i<E_j$. We remark that the partial order $\leq$ on the $E_i$'s is consistent with the natural order of the reduced basis $H,E_1,E_2,\cdots,E_N$. 
\end{definition}

The minimal elements $E_i$ with respect to the partial order $\leq$, which correspond to the points 
$\hat{E}_i$ that are lying in $\C\P^2$ (these are the so-called proper base points of $\C\P^2$ in the complex analytic setting, cf. \cite{AC, Bea}), are given below (cf. \cite{C1}, Theorem 4.3),
$$
\E(D):=\{E_i|\mbox{there is no $F_k\subseteq D$ with zero $a$-coefficient such that $E_i\cdot F_k>0$}\}.
$$
In particular, if $E_i$ is non-minimal, then there must be a component of $D$ of zero $a$-coefficient, whose homological expression contains $E_i$ as a non-leading class.

\begin{lemma}
{\em (1)} Suppose $E_i$ is infinitely near to $E_m$ of order $1$, then $E_i$ must be contained in the homological expression of the component of $D$ of leading class $E_m$. On the other hand, 
let $S$ be a component of $D$ of zero $a$-coefficient, and let $E_m$ be its leading class. 
If the homological expression of $S$ contains $E_i$ as a non-leading class, then $E_m<E_i$, i.e., $E_i$ is infinitely near to $E_m$. 

{\em (2)} The maximal elements with respect to the partial order $\leq$ consist of those classes $E_i$,
where either $E_i$ is not the leading class of any component of $D$, or $E_i$ is the leading class 
of a component of $D$ which is a $(-1)$-sphere. 

{\em (3)} Let $E_i$ be a non-minimal class. Then $E_i$ is contained in the homological expression, as a non-leading class, of either one or two components of $D$ of zero $a$-coefficeint. Moreover, let $E_m$, or in the latter case, $E_m, E_n$, be the leading classes respectively. Then $E_i$ is infinitely near to $E_m$ of order $1$, and in the latter case, $E_n<E_m$.

{\em (4)} Let $F_k$ be any component of $D$ with positive $a$-coefficient. Let $b_{ki}$ be the $b_i$-coefficients of $F_k$. Then for any $0<i,j\leq N$, $b_{ki}\geq b_{kj}$ if $E_i\leq E_j$.
\end{lemma}

\begin{proof}
(1) First, let $E_i$ be infinitely near to $E_m$ of order $1$. Then since the point $\hat{E}_i$ lies on
the $(-1)$-sphere $C_m$ representing the class $E_m$, $C_m$ must be part of the descendant $D_m$ of
$D$, as otherwise, one would have perturbed $C_m$ to a general position to avoid the point 
$\hat{E}_i$. Consequently, there is a component of $D$ with leading class $E_m$. Moreover, it is
easy to see that its homological expression must contain $E_i$, as $\hat{E}_i$ is lying on $C_m$.

On the other hand, suppose $S$ is a component of $D$ with leading class $E_m$, whose homological
expression contains $E_i$ as a non-leading class. In particular, $E_i$ is a non-minimal class. If $E_i$ is infinitely near to $E_m$ of order $1$, then we are done. Otherwise, there must be another component $\tilde{S}$ with leading class
$E_{\tilde{m}}$, such that $E_i$ is infinitely near to $E_{\tilde{m}}$ of order $1$. Now observe that the homological expressions of both $S$ and $\tilde{S}$ contain $E_i$ as a non-leading class, and with $S\cdot \tilde{S}\geq 0$, it follows easily that $E_{\tilde{m}}$ must be contained in the homological expression of $S$ as a non-leading class. Then note that $|\tilde{m}-m|<|i-m|$, which allows us to 
show $E_m<E_{\tilde{m}}$ by induction. Hence $E_m<E_i$. 

(2) Suppose $E_i$ is not the leading class of any component of $D$. To see it must be maximal,
we note that, in the definition of the successive blowing-down procedure, the class $E_i$ is represented by a symplectic $(-1)$-sphere $C_i$ which intersects transversely and positively with the corresponding descendant of $D$ when we blow down the class $E_i$. In particular, there are no points 
$\hat{E}_j$ lying on $C_i$, so that there are no classes $E_j$ which are infinitely near to $E_i$ of order $1$. This proves that $E_i$ is maximal. If $E_i$ is the leading class of a component $S$ of $D$. Then it follows easily from part (1) that $E_i$ is maximal if and only if $S$ is a $(-1)$-sphere. 

(3) Since $E_i$ is non-minimal, there must be a class $E_m$ such that $E_i$ is infinitely near to
$E_m$ of order $1$. Moreover, by part (1) $E_m$ is the leading class of a component $S$ of $D$ whose
homological expression contains $E_i$. If $S^\prime$ is another component of $D$ of zero 
$a$-coefficient whose homological expression contains $E_i$ as a non-leading class, then 
as we have seen in the proof of part (1), $E_m$ must be contained in the homological expression of 
$S^\prime$ as a non-leading class. Since $S^\prime$ contains both $E_i$ and $E_m$, it follows easily
that there can be at most one such component of $D$, because any two distinct such componentsc have a non-negative intersection by the condition (\dag). Finally, if $E_n$ is the leading class of $S^\prime$, then $E_n<E_m$ by part (1). 

(4) First, let $E_i,E_j$ be any two classes where $E_j$ is infinitely near
to $E_i$ of order $1$. We claim that $b_{ki}\geq b_{kj}$ must be true. To see this, let $S$ be the component of $D$ with leading class $E_i$, and we write
$$
S=E_i-E_{j_1}-E_{j_2}-\cdots -E_{j_m}.
$$
Then $S\cdot F_k=b_{ki}-b_{kj_1}-b_{kj_2}-\cdots-b_{kj_m}$. By part (1), $E_j$ is one of the 
$E_{j_s}$'s. With this understood, it follows easily that $b_{ki}\geq b_{kj}$, as $S\cdot F_k\geq 0$ 
and $b_{kj_s}\geq 0$, $\forall s$ (here we use the assumption that $F_k$ has positive $a$-coefficient). Inductively, we conclude that for any two classes $E_i$, $E_j$, if $E_i\leq E_j$, then 
$b_{ki}\geq b_{kj}$.

\end{proof}

We shall distinguish the two cases in Lemma 3.4(3). Borrowing the relevant terminology from algebraic geometry 
(cf. \cite{AC}, Definition 1.1.21), we call a non-minimal class $E_i$ {\bf free} if there is only one component of $D$ containing it as a non-leading class; otherwise, we call $E_i$ a {\bf satellite} class. 

On the other hand, we also note that, as a corollary of Lemma 3.4, for each 
non-minimal class $E_i$, there is a unique class $E_j$ such that $E_i$ is infinitely near to $E_j$
of order $1$. It follows easily that for any non-minimal class $E_i$, there is a uniquely determined linear
chain of classes $E_{j_1}$, $E_{j_2}, \cdots, E_{j_m}$, such that $E_i$ is infinitely near to $E_{j_1}$
of order $1$, for any $s$, $E_{j_s}$ is infinitely near to $E_{j_{s+1}}$ of order $1$, and the last class
$E_{j_m}\in\E(D)$ (i.e., $E_{j_m}$ is minimal). We shall call it the {\bf linear chain associated to} $E_i$. Note that if $E_m$ is a class such that $E_m<E_i$, then $E_m$ must be one of the $E_{j_s}$'s
in the linear chain associated to $E_i$. 

\subsection{Combinatorial type and virtual combinatorial type}
With the preceding understood, we shall next describe the {\bf combinatorial type} of the symplectic arrangement 
$\hat{D}\subset \C\P^2$, resulted from the successive blowing-down procedure associated to a given 
a homological expression $F_k\mapsto A_k=a_k H-\sum_{i=1}^N b_{ki} E_i$ of $D$.

First of all, observe that each component $F_k$ in $D$ has a 
non-negative $a$-coefficient, as the final stage of the successive blowing-down is $\C\P^2$.
Secondly, only each of those $F_k$ with positive $a$-coefficient descends to an irreducible component in $\hat{D}$, which we denote by $\hat{F}_k$, i.e.,
$$
\hat{D}=\cup_{\{k| a_k>0\}} \hat{F}_k.
$$
With this understood, as part of the combinatorial type of $\hat{D}$ we assign each $\hat{F}_k$ with a pair of integers $(a_k,g_k)$, where $a_k$ is the $a$-coefficient of $F_k$ and $g_k$ is the genus of $F_k$. It is easy to 
see that $a_k$ is the degree of $\hat{F}_k$ in $\C\P^2$, and $\hat{F}_k$ can be parametrized by a  $\hat{J}$-holomorphic map from a genus $g_k$ surface into $\C\P^2$. The rest of the combinatorial type is concerned with the singularities of each $\hat{F}_k$ as well as how the components of $\hat{D}$ intersect with each other. 

Any intersection point in $D$ between $F_k,F_l$, where $a_k,a_l>0$, carries over to $\hat{D}$. The new intersections and the singularities of the components $\hat{F}_k$ in 
$\hat{D}$ all occur at the points $\hat{E}_i$ where $E_i\in\E(D)$. To describe this part of the combinatorial type of $\hat{D}$, we recall that for each $E_i\in\E(D)$, there is a $4$-ball $B(\hat{E}_i)$ in $\C\P^2$, with center $\hat{E}_i$, such that $(\hat{\omega},\hat{J})$ is the standard linear structure on $B(\hat{E}_i)$.

\vspace{2mm}

\noindent{\bf Description of $\hat{D}$ near the points $\hat{E}_i$, where $E_i\in\E(D)$:}
\begin{itemize}
\item Suppose $E_i\in\E(D)$ is maximal. Then each component 
$\hat{F}_k\subset \hat{D}$ in the $4$-ball $B(\hat{E}_i)$ consists of $b_{ki}=E_i\cdot F_k$ many
complex linear disks. 
\item Suppose $E_i\in\E(D)$ is non-maximal. Let $S_i$ be the component of $D$ whose leading class is $E_i$, and let $\{E_{i_\alpha}\}$ be the set of classes such that each $E_{i_\alpha}$
is infinitely near to $E_i$ of order $1$. Then each $\hat{F}_k\cap B(\hat{E}_i)$ consists of a union of holomorphic disks intersecting at $\hat{E}_i$, $\hat{F}_k\cap B(\hat{E}_i)=\cup_j \U_{kj}$, where $j$ is running over the classes $E_j$ such that $E_i\leq E_j$, according to the following rules: 
\begin{itemize}
\item For the case of $E_j=E_i$, if $F_k$ intersects $S_i$, then there is a complex line $L_k$ in
$B(\hat{E}_i)$ and $\U_{kj}$ consists of one linear disk lying on $L_k$. If $F_k$ does not intersect $S_i$, then $\U_{kj}$ is empty. 
\item Each $E_{i_\alpha}$ is assigned with a complex line $L_\alpha$ in $B(\hat{E}_i)$, distinct from each $L_k$, and a complex coordinate system $(z_1,z_2)$ on $B(\hat{E}_i)$ such that $L_\alpha$
is given by $z_1=0$. Moreover, (i) if $E_{i_\alpha}$ is maximal, then $\U_{ki_\alpha}$
consists of $b_{ki_\alpha}=E_{i_\alpha}\cdot F_k$ many embedded disks, each of which is either given by $z_1=0$ or $z_1=cz_2^2$ for some $0\neq c\in\C$.
(ii) If $E_{i_\alpha}$ is not maximal, then for the case $E_j=E_{i_\alpha}$, $\U_{kj}$ is either empty or consists of one embedded disk given by $z_1=0$ or $z_1=cz_2^2$, depending on whether the component of $D$ with leading class $E_{i_\alpha}$ intersects $F_k$ or not. For any other $j$ where 
$E_{i_\alpha}<E_j$, if $E_j$ is maximal, then $\U_{kj}$ consists of $b_{kj}=E_j\cdot F_k$ many 
disks all given by the same equation $z_1^q=cz_2^{p+q}$ (with distinct $c$). 
The relatively prime integer pairs $(p,q)$ appearing in the equation defining the disks in $\U_{kj}$ can be computed,
in a canonical way, from the homological expression of the components of $D$ whose leading class appears in the linear chain associated to the maximal class $E_j$. If $E_j$ is not maximal, then $\U_{kj}$ is either empty or consists of only one disk (embedded or singular of form $z_1^q=cz_2^{p+q}$), depending on whether the component of $D$ with leading class $E_j$ intersects $F_k$ or not. 
(See the construction of the successive blowing-down in \cite{C1}, Section 4.)
\end{itemize}
\end{itemize}

We observe that the description of combinatorial type above only involves the partial order $\leq$ of infinitely-nearness on the set of $E_i$-classes and the corresponding homological assignment $(\vec{v}_k)$ in 
$\hat{\Omega}(D)$. (Note that even the genus $g_k$ of $F_k$ is determined by $\vec{v}_k$ via the adjunction formula.) 
Furthermore, it is easy to see that the description can be extended to a virtual setting in a
fairly straightforward way. But first, we need to formulate it properly. 

\begin{definition}
Let  $H,E_1,E_2,\cdots,E_N$ be a standard basis which is ordered. Let $(\vec{v}_k)\in \Omega(D)$ be an element 
satisfying the following conditions:
\begin{itemize}
\item The first entry $a_k$ in each $\vec{v}_k$ is non-negative.
\item Each class $A_k:=a_k H-\sum_{i=1}^N b_{ki}E_i$ is positive with respect to the ordered basis $H,E_1,E_2,\cdots,E_N$ in the sense of Definition 2.2(2), where $a_k,b_{ki}$ are the entries of $(\vec{v}_k)$.
\end{itemize}
We call the assignment $F_k\mapsto A_k$ a {\bf virtual homological expression} of $D$.
\end{definition}

We remark that the main difference between a virtual homological expression and a homological expression is that in a virtual homological expression, the standard basis $H,E_1,E_2,\cdots,E_N$, which is ordered, is not required to satisfy any area conditions.
With this understood, we remark that the assumptions (a) and (b) in the successive blowing-down procedure make perfect sense for a virtual homological expression. 

The following lemma is self-evident, whose proof is left to the reader. 

\begin{lemma}
Let $F_k\mapsto A_k=a_k H-\sum_{i=1}^N b_{ki} E_i$ be a virtual homological expression of $D$. Then there is a well-defined partial order of infinitely-nearness on the set of $E_i$-classes. Furthermore, the virtual homological expression $F_k\mapsto A_k=a_k H-\sum_{i=1}^N b_{ki} E_i$ determines a well-defined combinatorial type, which
is called {\bf the virtual combinatorial type} associated to $F_k\mapsto A_k=a_k H-\sum_{i=1}^N b_{ki} E_i$.
\end{lemma}

It is easy to see that a homological expression of $D$, which gives rise to a symplectic arrangement $\hat{D}$ in
$\C\P^2$ under the successive blowing-down, is a virtual homological expression of $D$. Moreover, the combinatorial type of $\hat{D}$ coincides with the virtual combinatorial type. This is consistent with the following definition.

\begin{definition}
We say a virtual homological expression of $D$ is {\bf realizable} if there is a symplectic arrangement in $\C\P^2$ which realizes the virtual combinatorial type associated to the virtual homological expression. Otherwise, it is called {\bf nonrealizable}. 
\end{definition}

Here in Definition 3.7, we do not require the symplectic arrangement is actually resulted from the successive blowing-down procedure associated to a homological expression of $D$.

\subsection{Quadratic Cremona transformations}
Recall that for any $(-2)$-class $\gamma=H-E_r-E_s-E_t$, the reflection $R(\gamma)$ acts on the set
$\Omega(D)$ as long as the admissibility of the corresponding classes are preserved (cf. Section 1).
In what follows, we shall give some conditions under which the admissibility is preserved under $R(\gamma)$. 
The reflections $R(\gamma)$ are closely related to quadratic Cremona transformations in algebraic geometry. 
Under some further assumptions, we shall define a symplectic analog of quadratic Cremona transformations
associated to such a reflection $R(\gamma)$.

We first give a brief review on quadratic Cremona transformations in algebraic geometry, see e.g. \cite{AC, Moe} 
for a more comprehensive discussion. 

Recall that a quadratic Cremona transformation is a degree $2$ birational automorphism of $\C\P^2$, $\Psi: \C\P^2\dashrightarrow \C\P^2$. As such, it has $3$ base points counted with multiplicity. It follows easily that there is a rational surface $X$, together with a pair of birational morphisms $\pi, \pi^\prime: X\rightarrow \C\P^2$, each of which is a successive blowing-up at $3$ points, such that $\pi^\prime=\Psi\cdot \pi$. 
With this understood, if we denote by $H,E_1,E_2,E_3$ (resp. $H^\prime,E_1^\prime,E_2^\prime,E_3^\prime$) the standard basis associated to the successive blowing-up $\pi: X\rightarrow \C\P^2$ (resp. $\pi^\prime$), and let 
$\gamma=H-E_1-E_2-E_3$, then the reflection associated to the Cremona transformation 
$\Psi: \C\P^2\dashrightarrow \C\P^2$ is $R(\gamma)$. In particular, note that 
$H^\prime,E_1^\prime,E_2^\prime,E_3^\prime$ is the image of $H,E_1,E_2,E_3$ under $R(\gamma)$; more concretely, 
$$
H^\prime=2H-E_1-E_2-E_3,E_1^\prime=H-E_2-E_3, E_2^\prime=H-E_1-E_3, E_3^\prime=H-E_1-E_2.
$$

If $\tilde{X}$ is any rational surface with birational morphisms $\tilde{\pi}, \tilde{\pi}^\prime: \tilde{X}\rightarrow \C\P^2$ such that $\tilde{\pi}^\prime=\Psi\circ \tilde{\pi}$, then there is a birational morphism $\eta: \tilde{X}\rightarrow X$ such that $\tilde{\pi}, \tilde{\pi}^\prime$ can factor through $\eta$, i.e, 
$\tilde{\pi}=\pi\circ \eta$ and $\tilde{\pi}^\prime=\pi^\prime\circ\eta$. It is easy to see 
that the reflection $R(\gamma)$ extends uniquely to a corresponding reflection on $H^2(\tilde{X})$, and 
$\Psi=\pi^\prime\circ \pi^{-1}=\tilde{\pi}^\prime\circ \tilde{\pi}^{-1}$, i.e., the Cremona transformation $\Psi$ is obtained by a successive blowing up $\pi^{-1}$ (resp. $\tilde{\pi}^{-1}$) followed by a successive blowing down 
$\pi^\prime$ (resp. $\tilde{\pi}^{\prime}$).

In order to give a more concrete description of $\Psi$, we let $p_1,p_2,p_3$ (resp. $p_1^\prime,p_2^\prime,p_3^\prime$) be the corresponding base points. Then there are three scenarios for the successive blowing-up $\pi$ (resp. $\pi^\prime$):
\begin{itemize}
\item [{(1)}] $p_1,p_2,p_3$ (resp. $p_1^\prime,p_2^\prime,p_3^\prime$) are all proper base points, and $p_1,p_2,p_3$ (resp. $p_1^\prime,p_2^\prime,p_3^\prime$) are not contained in a line in $\C\P^2$,
\item [{(2)}] $p_1,p_2$  (resp. $p_1^\prime,p_2^\prime$) are proper base points, 
$p_3$ (resp. $p_3^\prime$) is infinitely near to $p_2$  (resp. $p_2^\prime$), such that $p_3$ (resp. $p_3^\prime$)
does not lie in the proper transform of the line passing through  $p_1,p_2$  (resp. $p_1^\prime,p_2^\prime$),
\item [{(3)}] $p_1$  (resp. $p_1^\prime$) is proper, $p_2$  (resp. $p_2^\prime$) is infinitely near 
to $p_1$  (resp. $p_1^\prime$),  $p_3$  (resp. $p_3^\prime$) is infinitely near 
to $p_2$  (resp. $p_2^\prime$), and $p_3$  (resp. $p_3^\prime$) is not a satellite point. 
\end{itemize}

Correspondingly, we have the following description of the Cremona transformation $\Psi: \C\P^2\dashrightarrow \C\P^2$ in each of the three cases above (cf. \cite{Moe}, Chapter 5).

\vspace{1mm}

Case  (1): Since $p_1,p_2,p_3$ are proper base points and are not contained in a line, for each of the three pairs of points, i.e., $p_1,p_2$, $p_2,p_3$, and $p_1,p_3$, there is a unique line containing them, which will be denoted by
$L_3$, $L_1$, and $L_2$ respectively. Now blow up at $p_1,p_2$, and $p_3$, and let $E_1,E_2$, and $E_3$
be the corresponding exceptional $(-1)$-spheres. Then observe that the proper transform of $L_3$, $L_1$, and $L_2$ has class $E_3^\prime=H-E_1-E_2$, $E_1^\prime=H-E_2-E_3$ and $E_2^\prime=H-E_1-E_3$ 
respectively, which are $(-1)$-spheres in $X$. With this understood, the birational morphism 
$\pi^\prime: X\rightarrow \C\P^2$ blows down these $(-1)$-spheres to the base points $p_3^\prime,p_1^\prime,p_2^\prime$ in $\C\P^2$. Note that if we let $\pi^\prime(E_i)=L_i^\prime$ for $i=1,2,3$,
then $L_1^\prime, L_2^\prime, L_3^\prime$ are distinct lines, and  $L_1^\prime$ contains $p_2^\prime$ and 
$p_3^\prime$, $L_2^\prime$ contains $p_1^\prime$ and 
$p_3^\prime$, and $L_3^\prime$ contains $p_1^\prime$ and $p_2^\prime$ respectively. See Figure 3(1).

\begin{figure}[h]
   \centering
   \includegraphics[width=0.95\textwidth]{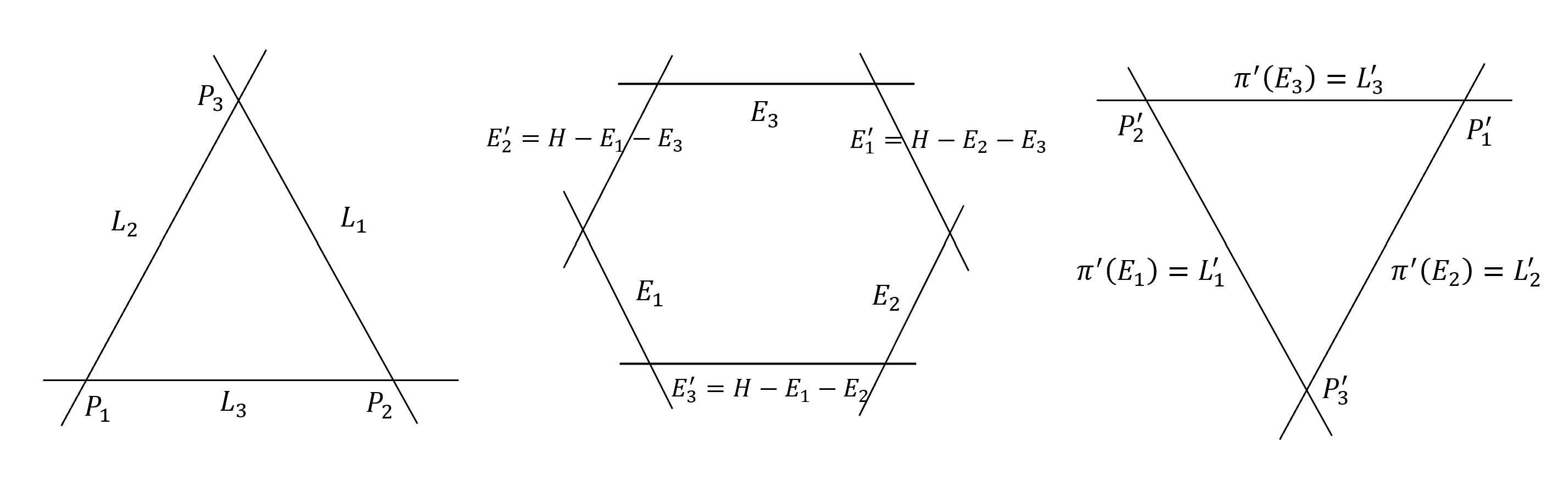}
   \caption*{Figure 3(1)}
\end{figure}

Case  (2): Since $p_1,p_2$ are proper base points, there is a unique line $L_3$ containing both of them. Let 
$L_1$ be the unique line passing through $p_2$ whose proper transform contains the base point $p_3$, as $p_3$
is infinitely near to $p_2$. Note that by the assumption in (2), it follows easily that $L_1,L_3$ are distinct. Now 
after blowing up at $p_1,p_2$ and $p_3$ and letting $E_1,E_2$ and $E_3$ be the corresponding exceptional 
$(-1)$-spheres, we observe that the proper transform of $L_3$ and $L_1$ has class 
 $E_3^\prime=H-E_1-E_2$ and $E_1^\prime=H-E_2-E_3$ respectively, and observe that 
 $E_2^\prime-E_3^\prime=E_2-E_3$, which is the class of the proper transform of $E_2$ 
 after blowing up at $p_3$. With this understood, the birational morphism $\pi^\prime: X\rightarrow \C\P^2$ blows down $E_3^\prime, E_2^\prime$ and $E_1^\prime$ successively. It is clear that for the corresponding base points, $p_3^\prime$ is infinitely near to $p_2^\prime$, and $p_1^\prime,p_2^\prime$ are proper base points. Moreover,  $p_1^\prime,p_2^\prime$ are contained in the line $L_3^\prime:=\pi^\prime(E_3)$, and the line $L_1^\prime:=\pi^\prime(E_1)$ contains the base point $p_2^\prime$ and its proper transform contains the base point $p_3^\prime$, and  $L_1^\prime\neq L_3^\prime$. See Figure 3(2).

\begin{figure}[h]
   \centering
   \includegraphics[width=0.95\textwidth]{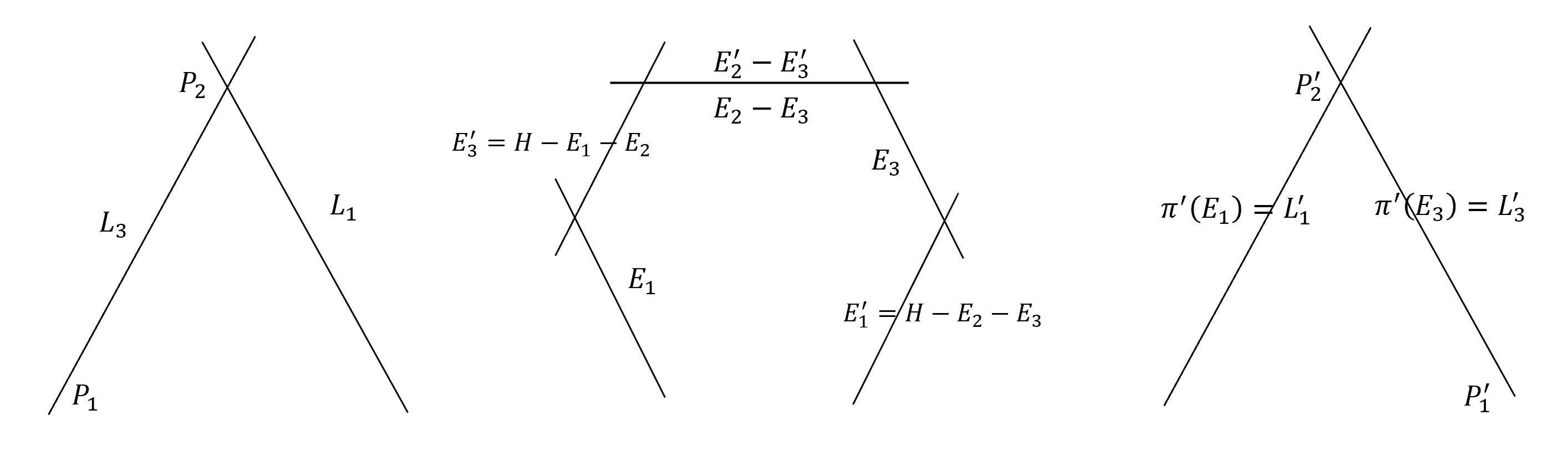}
   \caption*{Figure 3(2)}
\end{figure}

Case  (3): Let $L_3$ be the unique line passing through the proper base point $p_1$ such that its proper transform contains the base point $p_2$ after blowing up at $p_1$, as $p_2$ is infinitely near to $p_1$. Furthermore, since 
the base point $p_3$ is infinitely near to $p_2$ but is not a satellite point, it is not contained in the proper transform of the exceptional $(-1)$-sphere $E_1$, which has class $E_1-E_2$. Note that $p_3$ is also not contained in the proper transform of  $L_3$ after blowing up successively at $p_1$ and $p_2$, as it has class
$E_3^\prime=H-E_1-E_2$. It follows easily that after blowing up at  $p_3$, the proper transforms of $E_1$ and $L_3$ remain the same, which has class $E_1-E_2=E_1^\prime-E_2^\prime$ and $E_3^\prime=H-E_1-E_2$
respectively. Note that the proper transform of $E_2$ has class $E_2-E_3=E_2^\prime-E_3^\prime$. Now the
birational morphism $\pi^\prime: X\rightarrow \C\P^2$ blows down $E_3^\prime, E_2^\prime$ and $E_1^\prime$ successively. If we let $L_3^\prime=\pi^\prime(E_3)$, then it is easy to see that $L_3^\prime$ contains the base point $p_1^\prime$, which is a proper base point, the proper transform of $L_3^\prime$ contains the base point $p_2^\prime$, and the base point $p_3^\prime$ is not a satellite point. See Figure 3(3).

\vspace{2mm}

\begin{figure}[h]
   \centering
   \includegraphics[width=0.95\textwidth]{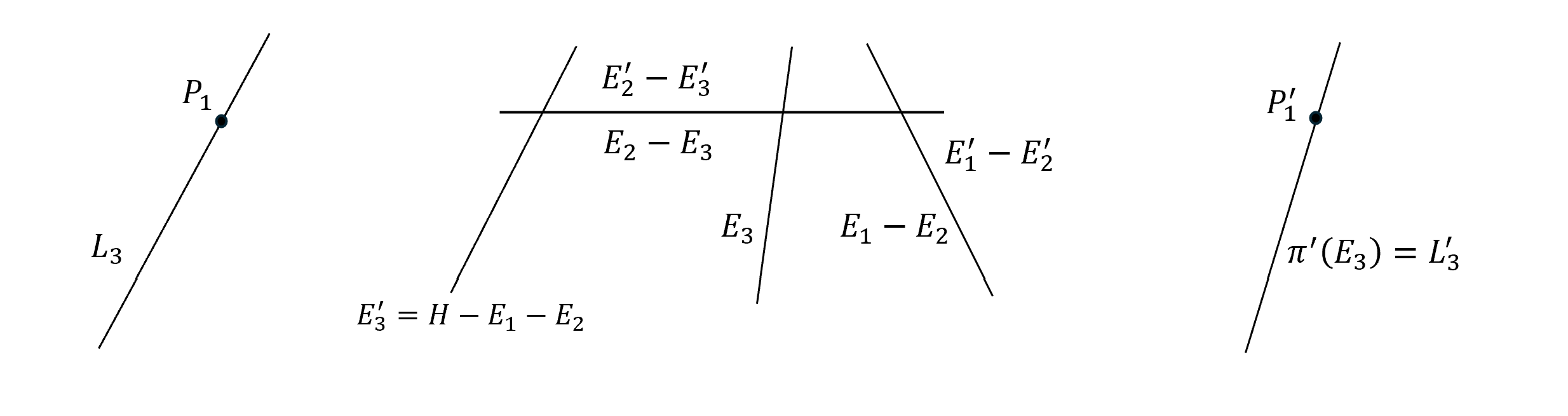}
   \caption*{Figure 3(3)}
\end{figure}

Now suppose we are given a complex arrangement $\hat{D}$ in $\C\P^2$ (i.e., a union of irreducible curves in 
$\C\P^2$). Furthermore, let $\tilde{X}$ be a rational surface with birational morphisms 
$\tilde{\pi}, \tilde{\pi}^\prime: \tilde{X}\rightarrow \C\P^2$ such that $\Psi=\tilde{\pi}^\prime\circ \tilde{\pi}^{-1}$.
For each $D\subset \tilde{X}$ which is a union of irreducible curves $\{C_k\}$, whose direct image under
$\tilde{\pi}: \tilde{X}\rightarrow \C\P^2$ is the given a complex arrangement $\hat{D}$, we obtain a complex
arrangement $\hat{D}^\prime$ in $\C\P^2$ which is the direct image of $D$ under 
$\tilde{\pi}^\prime: \tilde{X}\rightarrow \C\P^2$. Note that the correspondence between $\hat{D}$ and $\hat{D}^\prime$, to be denoted by $\Psi_D: \hat{D}\mapsto \hat{D}^\prime$, not only depends on the Cremona transformation $\Psi$, but also on the choice of $D\subset \tilde{X}$; in particular, $D$ is not unique and some of 
the irreducible curves $C_k$ in $D$ can be exceptional under $\tilde{\pi}: \tilde{X}\rightarrow \C\P^2$. Note that
algebraically, $\Psi_D: \hat{D}\mapsto \hat{D}^\prime$ is given by the correspondence $(\vec{v}_k)\mapsto 
(\vec{v}_k^\prime):=R(\gamma)(\vec{v}_k)$, where $R(\gamma)$ is the reflection associated to the Cremona transformation $\Psi$ and $(\vec{v}_k), (\vec{v}_k^\prime)$ are the sets of vectors given by the coefficients of the 
irreducible curves $\{C_k\}$ in $D$ with respect to the standard bases associated to the birational morphisms 
$\tilde{\pi}, \tilde{\pi}^\prime: \tilde{X}\rightarrow \C\P^2$  (cf. Lemma 2.3(1)). In particular, the combinatorial types of $\hat{D},\hat{D}^\prime$ are encoded in $(\vec{v}_k), (\vec{v}_k^\prime)$. In a nutshell, our main technical result, Theorem 1.8, describes a symplectic analog of the correspondence  
$\Psi_D: \hat{D}\mapsto \hat{D}^\prime$. 

\vspace{2mm}

Now back to the study of the symplectic configuration $D$. Let $(\vec{v}_k)\in\Omega(D)$ be an element which is realized by some $\omega\in Z(\vec{\delta})$, such that the successive blowing-down procedure can be performed to the final stage of $\C\P^2$, resulting a symplectic arrangement $\hat{D}$ in $\C\P^2$.
Let $H,E_1,E_2,\cdots,E_N$ be the reduced basis, with respect to which the $a$, $b_i$-coefficients of the class of $F_k$ are given by the entries in the vector $\vec{v}_k$. Then as we have shown earlier, there is a partial order 
$\leq$ of infinitely-nearness on the set of the $E_i$-classes, and a combinatorial type associated to the symplectic arrangement $\hat{D}$, all depending only on $(\vec{v}_k)$. 

With the preceding understood, let $E_r, E_s,E_t$ be three distinct $E_i$-classes, and let 
$\gamma:=H-E_r-E_s-E_t$. Set $(\vec{v}_k^\prime):=R(\gamma)(\vec{v}_k)$. 
Then observe that if we let $H^\prime, E_1^\prime,E_2^\prime,\cdots E_N^\prime$
be the image of $H,E_1,E_2,\cdots,E_N$ under the reflection $R(\gamma)$, and write
$\vec{v}_k=(a_k, b_{k1},\cdots,b_{kN})$ and $\vec{v}_k^\prime=(a_k^\prime, b_{k1}^\prime,\cdots,b_{kN}^\prime)$, then 
$$
a_kH-\sum_{i=1}^N b_{ki} E_i=a_k^\prime H^\prime-\sum_{i=1}^N b_{ki}^\prime E_i^\prime.
$$
In particular, $\vec{v}_k^\prime$ encodes the coefficients of the class of $F_k$ with respect to the basis $H^\prime, E_1^\prime,E_2^\prime,\cdots E_N^\prime$, which is only a standard basis.

\begin{lemma}
Let $(\vec{v}_k^\prime):=R(\gamma)(\vec{v}_k)$, where $\gamma:=H-E_r-E_s-E_t$. 
Then $(\vec{v}_k^\prime)\in\Omega(D)$, i.e., each 
$\vec{v}_k^\prime$ is admissible, if and only if the following conditions hold: for any $k$,
\begin{itemize}
\item [{(i)}] if $a_k=1$, then $b_{kr}+b_{ks}+b_{kt}\leq 2$, and 
\item [{(ii)}] if $a_k=0$, then $b_{kr}+b_{ks}+b_{kt}\leq 0$. 
\end{itemize}
Moreover, the first entry $a_k^\prime$ in each $\vec{v}_k^\prime$ is non-negative. 
\end{lemma}

\begin{proof}
To simplify the notations, we shall drop the index $k$ in $\vec{v}_k$ and $\vec{v}_k^\prime$, simply write it as 
$\vec{v}=(a,b_{1},b_{2}, \cdots, b_{N})$ and $\vec{v}^\prime=(a^\prime,b_{1}^\prime,b_{2}^\prime, \cdots, 
b_{N}^\prime)$.

With this understood, we let $A=aH-\sum_{i=1}^N b_iE_i$ and $A^\prime=a^\prime H-\sum_{i=1}^N b_i^\prime E_i$
where $A^\prime=R(\gamma)(A)=A+(\gamma\cdot A)\gamma$. Note that 
$\gamma\cdot A=a-(b_r+b_s+b_t)$, from which it follows easily that
$$
a^\prime=2a-(b_r+b_s+b_t), \; b_r^\prime=a-(b_s+b_t), \; b_s^\prime=a-(b_r+b_t), \; b_t^\prime=a-(b_r+b_s),
$$
and $b_i^\prime=b_i$ for any $i\neq r,s,t$.

With this understood, consider first the case where $a=1$. Since by assumption $b_{r}+b_{s}+b_{t}\leq 2$, 
it follows easily that $a^\prime\geq 0$, with $a^\prime= 0$ iff $b_{r}+b_{s}+b_{t}=2$. Since $a=1$, it is easy to see
that $b_{r}+b_{s}+b_{t}=2$ if and only if exactly two of $b_{r}, b_{s},b_{t}$ equal to $1$ and the third one equals $0$. It follows easily that $\vec{v}^\prime$ is admissible in this case. 

Next, assume $a=0$. In this case, we have  $b_{r}+b_{s}+b_{t}\leq 0$ by the assumption, which gives 
$a^\prime\geq 0$ immediately,  with $a^\prime= 0$ iff $b_{r}+b_{s}+b_{t}=0$ and $a^\prime>0$  
iff $b_{r}+b_{s}+b_{t}=-1$. It follows easily that $\vec{v}^\prime$ is admissible in this case as well.

Finally, we consider the case where $a\geq 2$. Note that under the condition of $a\geq 2$, one has $a>b_i$ 
for any $i$ (see the proof of Lemma 2.4). The assertion that $\vec{v}^\prime$ is admissible and $a^\prime>0$
follows immediately from the following claim:

\vspace{1mm}

{\bf Claim:} {\it Assume $a\geq 2$. Then for any $i\neq j$, $a\geq b_{i}+b_{j}$ holds true.}

\vspace{1mm}

For a proof, we let $F$ be the component of $D$ whose homological expression is given by the vector 
$\vec{v}$, and let $\hat{F}$ denote the corresponding component in $\hat{D}$, which is a $\hat{J}$-holomorphic curve of degree $a$. 

We begin by recalling the following formula for computing local intersection numbers. Let $C$ be a germ of holomorphic curves (not necessarily irreducible) at $0\in \C^2$, and let $L$ be a germ of embedded holomorphic disk intersecting $C$ only at $0\in \C^2$. We blow up at $0\in \C^2$ and let $E$ be the exceptional $(-1)$-sphere. Let $C^\prime$, $L^\prime$ denote the proper transforms of $C$, $L$ respectively. Then one has
$$
L\cdot C=E\cdot C^\prime+L^\prime\cdot C^\prime. 
$$
In particular, $L\cdot C\geq E\cdot C^\prime$ as $L^\prime\cdot C^\prime\geq 0$. 

Now back to the proof of the claim that $a\geq b_{i}+b_{j}$ for any $i\neq j$. First consider the case where
there exist two distinct minimal classes $E_m$, $E_n$ such that $E_m\leq E_i$, $E_n\leq E_j$.
Let $L$ be the unique degree $1$ $\hat{J}$-holomorphic sphere in $\C\P^2$ passing through 
the points $\hat{E}_m$, $\hat{E}_n$. Then on the one hand, $a=L\cdot \hat{F}$ as $\hat{F}$ is of
degree $a$, and on the other hand, the local intersection number of $L$ with $\hat{F}$ at
$\hat{E}_m$, $\hat{E}_n$ is bounded from below by $b_m, b_n$ respectively. This is because if 
$F^\prime$ denotes the proper transform of $\hat{F}$ after blowing up at $\hat{E}_m$ (resp. 
$\hat{E}_n$), then $b_m=E_m\cdot F^\prime$ (resp. $b_n=E_n\cdot F^\prime$). It follows easily that
$a\geq b_m+b_n\geq b_i+b_j$ by Lemma 3.4(4). 

It remains to consider the case where there is only one minimal class $E_m$ such that $E_m\leq E_i$,
$E_m\leq E_j$. Since $E_i,E_j$ are distinct, there must be a class $E_{i_\alpha}$ which is infinitely near 
to $E_m$ of order $1$ such that $E_{i_\alpha}\leq E_j$ (or $E_i$). With this understood, it suffices to
show that $a\geq b_m+b_{i_\alpha}$. To see this, let $L$ be the unique degree $1$ $\hat{J}$-holomorphic 
sphere in $\C\P^2$ passing through the point $\hat{E}_m$ and tangent to the line $L_\alpha$ in the $4$-ball $B(\hat{E}_m)$ determined by the class $E_{i_\alpha}$ (cf. Description of $\hat{D}$ near the points $\hat{E}_i$, where $E_i\in\E(D)$, i.e., $E_i$ is minimal). Then the local intersection number of $L$ with $\hat{F}$ at the 
point $\hat{E}_m$ equals $b_m+L^\prime\cdot F^\prime$, where $L^\prime,F^\prime$ are the proper transforms of $L$, $\hat{F}$ after blowing up at $\hat{E}_m$. Since $L$ is tangent to $L_\alpha$, $L^\prime$ must be an embedded disk containing 
the center $\hat{E}_{i_\alpha}$ of the $4$-ball $B(\hat{E}_{i_\alpha})$. Consequently, 
$L^\prime\cdot F^\prime\geq b_{i_\alpha}$, and the proof is finished. 
\end{proof}

\begin{remark}
Note that from the proof of Lemma 3.8, for any component $F_k$ of $D$, if its $a$-coefficient $a_k\geq 2$, then for any of its $b_i$-coefficients $b_{ki}, b_{kj}$ where $i,j$ are distinct, one must have
$a_k\geq b_{ki}+b_{kj}$. This is a necessary condition on the homological expression of  $D$, in order for $D$ to be blown-down to a symplectic arrangement in $\C\P^2$ under the successive blowing-down procedure associated to the given homological expression. For a counterexample, note that 
$$
A=5H-3E_1-3E_2-E_3-E_4-\cdots-E_{11}
$$
is an admissible class for a symplectic $(-2)$-sphere, which violates this condition. 

Using a similar, but more elaborate, argument, one can show that when $a_k\geq 3$, the inequality 
$2a_k\geq b_{ki_1}+b_{ki_2}+\cdots +b_{ki_5}$ holds true for any distinct indices $i_1,i_2,\cdots,i_5$.
Note that this inequality does not follow from the one in Lemma 3.8, i.e., $a_k\geq b_{ki}+b_{kj}$ for any 
$i\neq j$. Consider the following admissible class of a symplectic $(-2)$-sphere
$$
B=7H-3E_1-3E_2-3E_3-3E_4-3E_5-E_6-\cdots-E_{11}.
$$
The class $B$ is the result of applying $R(\gamma)$, where $\gamma=H-E_3-E_4-E_5$, to
$$
A=5H-3E_1-3E_2-E_3-E_4-E_5-E_6-\cdots-E_{11}.
$$
Note that the class $B$ obeys the inequality $a\geq b_{i}+b_{j}$ for any $i\neq j$, but it does not obey the inequality 
$2a\geq b_{i_1}+b_{i_2}+\cdots +b_{i_5}$ for any distinct $i_1,i_2,\cdots,i_5$.
\end{remark}

\vspace{2mm}

\noindent{\bf Proof of Theorem 1.8:}

\vspace{2mm}

It is easy to see that the assumptions on the classes $E_r,E_s,E_t$ in Theorem 1.8 imply that the assumptions in Lemma 3.8 hold true. Consequently, for each $k$, $\vec{v}_k^\prime$ is admissible and $a_k^\prime\geq 0$. Thus it remains to show that each $\vec{v}_k^\prime$ is positive with respect to an order of the standard basis 
$H^\prime, E_1^\prime, E_2^\prime,\cdots, E_N^\prime$ to conclude that $F_k\mapsto a_k^\prime H^\prime-\sum_{i=1}^N b_{ki}^\prime E_i^\prime$ is a virtual homological expression of $D$. 

To proceed further, we shall identify smoothly the successive symplectic blowing-down of $(X_N,D)$ to 
$(\C\P^2,\hat{D})$ with the successive blowing-down reversing the successive blowing-up 
$\pi: (\tilde{X}_N,\tilde{D})\rightarrow (\C\P^2,\hat{D})$. The advantage is that the blowing down of the classes $E_i$ does not need to follow strictly the total order given by the reduced basis $H,E_1,E_2,\cdots,E_N$, but rather it is only governed by the partial order defined by the relation of ``infinitely near", which depends only on $(\vec{v}_k)$. We shall regard $H,E_1,E_2,\cdots,E_N$ as classes in $\tilde{X}_N$, but note that it is only a standard, naturally ordered basis now. 

With the preceding understood, it suffices to show that there is a natural partial order of ``infinitely near" 
on the set $\{E_i^\prime\}$, which extends to a total order on the standard basis $H^\prime,E_1^\prime,E_2^\prime,\cdots,E_N^\prime$. Moreover, one can blow down $\tilde{X}_N$ along the classes $E_i^\prime$ according to the partial order, which transforms $\tilde{D}$ to a symplectic arrangement $\hat{D}^\prime$ in $\C\P^2$. As we have seen in the proof of Lemma 2.3, this would imply that each $(\vec{v}_k^\prime)$ is positive with respect to the order of $H^\prime,E_1^\prime,E_2^\prime,\cdots,E_N^\prime$, so that $F_k\mapsto a_k^\prime H^\prime-\sum_{i=1}^N b_{ki}^\prime E_i^\prime$ is a virtual homological expression of $D$. It is clear that its virtual combinatorial type is realized by the symplectic arrangement $\hat{D}^\prime$.

We begin by noting that for any $E_i\neq E_r,E_s$ or $E_t$, $E_i^\prime=E_i$. Moreover, it is easy to see that in any of the cases (1)-(3), for any $E_i,E_j$ not equal to $E_r,E_s$ or $E_t$ such that $E_i<E_j$,
there exists no $E_k\in \{E_r,E_s,E_t\}$, such that $E_i<E_k$ and $E_k<E_j$. As a consequence,
we can define a partial order $\leq$ of ``infinitely near" (resp. a total order) on the subset of $E_i^\prime$ where
$E_i^\prime=E_i$ by restricting the partial order of ``infinitely near" (resp. the natural total order) on the set $\{E_i\}$ to it. Furthermore, when blowing down $\tilde{X}_N$ along the classes $E_i^\prime$, we can first blow
down those $E_i^\prime$ where $E_i^\prime=E_i\neq E_r,E_s$ or $E_t$. Denote by $\check{X}$
the resulting $4$-manifold and by $\check{D}$ the descendant of $\tilde{D}$ in $\check{X}$. 
Then it is easy to see that $(\check{X},\check{D})$ is the result of applying Lemma 3.1 to
$(\C\P^2, \hat{D})$ successively at the points $\hat{E}_r$, $\hat{E}_s$ and $\hat{E}_t$. As such, 
$\check{X}$ has a natural symplectic structure $\check{\omega}$ and an $\check{\omega}$-compatible almost complex structure $\check{J}$, such that $\check{D}$ is $\check{J}$-holomorphic. 
Finally, we note that $H^\prime,E_r^\prime,E_s^\prime,E_t^\prime$ is naturally a standard basis of
$(\check{X},\check{\omega})$, and it remains to extend the partial order $\leq$ to 
$E_r^\prime,E_s^\prime,E_t^\prime$, and to describe how to successively blow down $\check{X}$
along the classes $E_r^\prime,E_s^\prime,E_t^\prime$, transforming $\check{D}$ to a symplectic 
arrangement in $\C\P^2$. The virtual combinatorial type of $(\vec{v}_k^\prime)$, as well as its realization by the symplectic arrangement, will follow automatically. 

With the preceding understood, consider case (1) first, where the classes $E_r,E_s,E_t$ are all minimal. 
In this case, $\hat{E}_r$, $\hat{E}_s$ and $\hat{E}_t$ are points in $\C\P^2$, and moreover, for each pair of them,
i.e., $\hat{E}_s$ and $\hat{E}_t$, $\hat{E}_r$ and $\hat{E}_t$, and $\hat{E}_s$ and $\hat{E}_r$,
there is a unique degree $1$ $\hat{J}$-holomorphic sphere passing through the pair of points, which will be denoted 
by $L_r, L_s$ and $L_t$ respectively. The proper transforms of $L_r, L_s$ and $L_t$ in $\check{X}$,
denoted by $C_r, C_s$ and $C_t$, are $\check{J}$-holomorphic $(-1)$-spheres representing the classes $E_r^\prime$, $E_s^\prime$, and $E_t^\prime$ respectively. (See Figure 2(1) for an illustration.) With this understood, we need to explain how to extend the partial order $\leq$ to the classes $E_r^\prime,E_s^\prime,E_t^\prime$ and how to blow down $C_r, C_s$ and $C_t$. It is clear that each of $E_r^\prime$, $E_s^\prime$, and $E_t^\prime$ is minimal. The question is whether it is also maximal, and how to blow down the corresponding $(-1)$-sphere. There are two cases we need to discuss separately. For simplicity, we shall focus on $C_r$ without loss of generality.

First, consider the case where $C_r$ is not a component of $\check{D}$. In this case, $E_r^\prime$
is not the leading class of any components of $D$ with respect to the basis $H^\prime, E_1^\prime,E_2^\prime,\cdots E_N^\prime$. Hence $E_r^\prime$ should be maximal in the partial order (cf. Lemma 3.4(2)). To blow down $C_r$, we shall perturb it slightly so that it intersects each component of $\check{D}$ transversely and positively (note that $C_r, C_s$ and $C_t$ are disjoint, 
so the perturbation of $C_r$ can be done without interference with $C_s,C_t$), and then symplectively blow down $\check{X}$ along the perturbed $C_r$ as described in \cite{C1}, Section 4. 

Next, consider the remaining case where $C_r$ is a component of  $\check{D}$. In this case, the class
$E_r^\prime$ will not be maximal, and the question is to determine which classes $E_j^\prime$, $j\neq s,t$, are infinitely near to $E_r^\prime$ of order $1$. 

Let $S$ be the component of $D$ whose descendant in $\check{D}$ is $C_r$. Then it is easy to see that the homological expression of $S$ (in the basis $H,E_1,E_2,\cdots, E_N$) takes the form
$$
S=H-E_s-E_t-E_{j_1}-E_{j_2}-\cdots-E_{j_n}, \mbox{ where } j_\beta\neq r, s,t, \forall \beta.
$$
In particular, note that $E_{j_\beta}^\prime=E_{j_\beta}$ for each $\beta$. With this understood, we declare a class $E_j^\prime$, where $j\neq s,t$, is infinitely near to $E_r^\prime$ of order $1$ if and only if $E_j^\prime=E_{j_\beta}$ for some $\beta$ and $E_j^\prime$ is minimal among the classes $E_i^\prime$ where $i\neq r,s,t$. 
 
 Now it comes to the question of how to blow down the $(-1)$-sphere $C_r$. Since it is a component 
 of $\check{D}$, we can no longer perturb it before blowing it down. With this understood, recall that
 in the process of successive blowing down, at each stage $m$, after blowing down the class $E_m$,
 the resulting standard symplectic $4$-ball $B(\hat{E}_m)$ in the next stage $X_{m-1}$ has a special
 complex coordinate $z_1,z_2$ such that each component of the descendant of $D$ inside 
$B(\hat{E}_m)$ is given by an equation of the form $z_1^p=cz_2^q$ for some $c\neq 0$ (here $p=q=1$ is allowed), or $z_1=0$, or $z_2=0$. This said, the key issue for blowing down $C_r$ is to be able to arrange, at each stage $m$ of the successive blowing down of the classes $E_i$, where $i\neq r,s,t$, such that the descendant of the component $S$ is given by either $z_1=0$ or $z_2=0$ in the $4$-ball $B(\hat{E}_m)$. The assumptions (a) and (b) are designed to ensure this for the components of $D$ whose $a$-coefficient is zero. However, $S$ does not have zero $a$-coefficient in the basis $H,E_1,E_2,\cdots, E_N$, 
so the assumptions (a), (b) on the homological expression given by $(\vec{v}_k)$ do not apply to the component $S$. This is where the assumptions (a), (b) on the virtual homological expression 
$F_k\mapsto a_k^\prime H^\prime-\sum_{i=1}^N b_{ki}^\prime E_i^\prime$ are needed. 

For an illustration, suppose $E_i$ is a maximal class in the partial order $\leq$ for the basis
$H,E_1,E_2,\cdots, E_N$, and let the following be the portion of the linear chain associated to
$E_i$ which lies in the complement of $E_r,E_s,E_t$, such that the minimal element $E_{i_1}$ of the portion is contained in the expression of $S$, i.e., $E_{i_1}=E_{j_\beta}$ for some $\beta$:
$$
E_{i_1}<E_{i_2}<\cdots < E_{i_{m-1}}<E_{i_m}<E_i.
$$
If we denote by $E_\alpha$ the first class in the chain above (from right to left) which appears in the expression of $S$, then before the class $E_\alpha$, the successive blowing down procedure does not involve the component $S$. Since this is only for an illustration, we will assume for simplicity that 
$E_\alpha=E_i$. With this understood, let $S_1$ be the component of $D$ whose leading class is
$E_{i_m}$. Then the expression of $S_1$ takes the form
$$
S_1=E_{i_m}-E_i-E_{k_1}-\cdots-E_{k_p}. 
$$
Note that $S\cdot S_1\geq 0$ implies that $E_{i_m}$ must appear in the expression of $S$, and 
$S\cdot S_1=0$ must be true. The same argument shows that all the classes $E_{i_1}, \cdots,
E_{i_{m-1}}$ must also appear in the expression of $S$. Furthermore, it is easy to see that there
is no class $E_j$ with $E_j<E_{i_{m-1}}$, such that $E_{i_m}$ appears in the expression of the component $S^\prime$ of $D$ whose leading class is $E_j$, because by the same argument $E_{i_{m-1}}$ also has to appear in the expression of $S^\prime$, which implies $S\cdot S^\prime<0$, a contradiction. 
What this means is that regarding the assumptions (a), (b), for the class $E_{i_m}$ it falls to the case (b). With this understood, it is possible that one of the classes $E_{k_1}, \cdots, E_{k_p}$, say $E_{k_1}$ without loss of generality, has the following property: $E_{k_1}$ does not appear in the expression of the component $S_2$ of $D$ whose leading class is $E_{i_{m-1}}$ and there is a component $F$ of $D$
such that $E_{k_1}\cdot F>1$. In this scenario, when we blow down $E_{i_m}$, the axes 
$z_1=0$ and $z_2=0$ in the $4$-ball $B(\hat{E}_{i_m})$ have to be occupied by the descendant of
$S_2$ and $F$ respectively, and there is no room for the descendant of $S$. With this understood, we have to eliminate the possibility of $F$ in order to make room for the descendant of $S$. But since 
in the basis $H^\prime,E_1^\prime, \cdots,E_N^\prime$, $S$ has zero $a$-coefficient, with leading class 
$E_{r}^\prime$, so when we apply the assumptions (a), (b) for the virtual homological expression 
$F_k\mapsto a_k^\prime H^\prime-\sum_{i=1}^N b_{ki}^\prime E_i^\prime$ to the class 
$E_{i_m}=E_{i_{m}}^\prime$, we are in the case (a) and therefore the component $F$ is not allowed 
under the assumption. This ensures that the descendant of $S$ in the $4$-ball $B(\hat{E}_{i_m})$
is given by one of the axes $z_1=0$ or $z_2=0$. This proves that at each stage $m$ of the successive blowing down of the classes $E_i$, where $i\neq r,s,t$, the descendant of the component $S$ is given by either 
$z_1=0$ or $z_2=0$ in the $4$-ball $B(\hat{E}_m)$, so that the $(-1)$-sphere $C_r$ can be blown down properly.
This completes the discussions on case (1). 

Next, we consider case (2). In case (2), only $\hat{E}_r$, $\hat{E}_s$ are points in
$\C\P^2$. Let $L_t$ be the unique degree $1$ $\hat{J}$-holomorphic sphere passing through
$\hat{E}_r$, $\hat{E}_s$, and let $L_r$ be the unique degree $1$ $\hat{J}$-holomorphic sphere
passing through $\hat{E}_s$ whose proper transform after blowing up at $\hat{E}_s$ contains the point $\hat{E}_t$ which is infinitely near to $\hat{E}_s$. (Note that by the assumption in case (2), $L_r\neq L_t$.) 
It is easy to see that the proper transforms of $L_t,L_r$ in
$\check{X}$, denoted by $C_t, C_r$ respectively, are $\check{J}$-holomorphic $(-1)$-spheres representing the classes $E_t^\prime=H-E_r-E_s$ and $E_r^\prime=H-E_s-E_t$. On the other hand, since $E_t$ is infinitely near to $E_s$ of order $1$, there is a component $\check{S}$ in $\check{D}$ which is $\check{J}$-holomorphic $(-2)$-sphere representing the class $E_s-E_t=E_s^\prime-E_t^\prime$. Note that the corresponding component 
$S$ in $D$ has a zero $a$-coefficient, with $E_s$ being the leading class of $S$. Finally, note that $C_t$ intersects $\check{S}$ transversely at one point, and $C_r$ is disjoint from $\check{S}$. (See Figure 2(2) for an illustration.)

With the preceding understood, assume first that none of $C_t,C_r$ is a component of $\check{D}$.
In this case, we shall slightly perturb each of them if necessary so that they intersect $\check{D}$ 
transversely and positively. Then we blow down successively the perturbed $C_t$, then the 
descendant of $\check{S}$ which has class $E_s^\prime$, then the perturbed $C_r$, reaching
$\C\P^2$. We should point out that since $\check{S}$ is a component of $\check{D}$, we need to use the assumptions (a), (b) for the virtual homological expression 
$F_k\mapsto a_k^\prime H^\prime-\sum_{i=1}^N b_{ki}^\prime E_i^\prime$ to ensure that the descendant 
of $\check{S}$ can be blown down properly. To see this, simply note that the corresponding component $S$ in $D$
has the following expression in the reduced basis $H,E_1,E_2,\cdots, E_N$:
$$
S=E_s-E_t-E_{j_1}-E_{j_2}-\cdots-E_{j_n}, \mbox{ where } j_\beta\neq r, s,t, \forall \beta.
$$
It follows easily that, in the basis $H^\prime,E_1^\prime,E_2^\prime,\cdots, E_N^\prime$, we have
$$
S=E_s^\prime-E_t^\prime-E_{j_1}^\prime-E_{j_2}^\prime-\cdots-E_{j_n}^\prime, \mbox{ where } j_\beta\neq r, s,t, \forall \beta,
$$
and the assumptions (a), (b) ensure that at each stage $m$ of the successive blowing down of the classes $E_i$, where $i\neq r,s,t$, the descendant of the component $S$ is given by either 
$z_1=0$ or $z_2=0$ in the $4$-ball $B(\hat{E}_m)$, as $E_i=E_i^\prime$, and at the last step blowing down 
the perturbed $C_t$, it is also given by $z_1=0$ or $z_2=0$ in the $4$-ball $B(\hat{E}_t^\prime)$. Finally,
we note that under the assumption that none of $C_t,C_r$ is a component of $\check{D}$,
both $E_t^\prime,E_r^\prime$ are maximal. Moreover, $E_s^\prime,E_r^\prime$
are both minimal, and $E_t^\prime$ is infinitely near to $E_s^\prime$ of order $1$. So to extend the partial order
$\leq$ to the classes $E_r^\prime,E_s^\prime,E_t^\prime$, we only need to decide which classe $E_j^\prime$,
for $j\neq r,s,t$, is infinitely near to $E_s^\prime$ of order $1$. It is clear that $E_j^\prime$ is infinitely near to $E_s^\prime$ of order $1$ if and only if $E_j^\prime=E_{j_\beta}$ for some $\beta$ in the expression of the component $S$ and $E_j^\prime$ is minimal among the classes $E_i^\prime$ where $i\neq r,s,t$. This finishes the
case where none of $C_t,C_r$ is a component of $\check{D}$.

If any of $C_t,C_r$, say $C_r$, is a component of $\check{D}$, we shall proceed exactly the same way as in case (1). More precisely, in this case, the class $E_r^\prime$ may not be maximal, and we need to make sure the $(-1)$-sphere $C_r$ can be blown down properly. The assumptions (a), (b) ensure that  $C_r$ can be blown down properly, and we can determine which class $E_j^\prime$ is infinitely near to $E_r^\prime$ of order $1$, by exactly the same arguments as in case (1). The discussion concerning the class $E_s^\prime$ and the blowing-down of the corresponding $(-1)$-sphere remain the same. We leave the details to the reader. 

Finally, consider case (3). In this case, only $\hat{E}_r$ is a point of $\C\P^2$. There is a unique degree 
$1$ $\hat{J}$-holomorphic sphere passing through $\hat{E}_r$, denoted by $L_t$, whose proper transform after blowing up at $\hat{E}_r$ contains the point $\hat{E}_s$ which is infinitely near to $\hat{E}_r$ of order $1$. Let $C_t$ be the proper transform of $L_t$ in $\check{X}$. Then $C_t$ is a $\check{J}$-holomorphic $(-1)$-sphere representing the class $H-E_r-E_s=E_t^\prime$. On the other hand, there are two $\check{J}$-holomorphic $(-2)$-spheres $\check{S}_1$, $\check{S}_2$ in $\check{D}$ representing the classes $E_r-E_s=E_r^\prime-E_s^\prime$,
$E_s-E_t=E_s^\prime-E_t^\prime$ respectively. (See Figure 2(3) for an illustration.) 
Blowing down $C_t$ (if $C_t$ is not a component of
$\check{D}$ we perturb it slightly so that it intersects $\check{D}$ transversely and positively), and
then the descendant of $\check{S}_2$ which has class $E_s^\prime$, then the descendant of $\check{S}_1$
which has class $E_r^\prime$, we reach $\C\P^2$. The justification that these $(-1)$-spheres can be blown down
properly by appealing to the assumptions (a), (b), and the extension of the partial order $\leq$ to the classes $E_r^\prime, E_s^\prime,E_t^\prime$, are done in exactly the same way as in the cases (1) and (2). 

The proof of Theorem 1.8 is completed.

\section{Gromov theory and holomorphicity of symplectic arrangements}

In this section, we give a proof of Theorem 1.11. Our proof  is based on Gromov's theory of pseudoholomorphic curves \cite{G}, see also \cite{Bar, HLS, IvaS, Shev, ST}. In particular, for the structure of the moduli space of pseudoholomorphic curves, we shall adapt the approach of Ivashkovich and Shevchishin in \cite{IvaS}, and for automatic transversality of the moduli space, it is based on the work of Hofer, Lizan, and Sikorav \cite{HLS}. 

\subsection{Structure of moduli spaces and automatic transversality}
We begin by reviewing the relevant work in \cite{IvaS} concerning the structure of the moduli space of pseudoholomorphic curves. To this end, consider a symplectic $4$-manifold $(X,\omega)$ and a nonzero homology class $A\in H_2(X)$, and fix a compact oriented surface $S$ of genus $g$. Fix an
$0<\alpha<1$, and let $\J$ be the Banach manifold of $\omega$-tame, $C^{1,\alpha}$-almost complex structures on $X$, and $\J_S$ be the Banach manifold of $C^{1,\alpha}$-almost complex structures on $S$ (compatible with the orientation of $S$). Finally, fix a $2<p<\infty$ and consider the Banach
manifold of $L^{1,p}$-maps from $S$ to $X$
$$
\S:=\{u\in L^{1,p}(S,X)| [u(S)]=A\}. 
$$
With this understood, let $\p$ be the subset of $\S\times \J_S\times \J$, where
$$
\p:=\{(u,J_S,J)\in \S\times \J_S\times \J | du+J\circ du\circ J_S=0\},
$$
and denote by ${\bf pr}_\J: \p\rightarrow \J$ the projection to the third factor. 

For each $(u,J_S,J)\in \p$, the map $u$ is a $(J_S,J)$-holomorphic map, and the image $u(S)$ in
$X$ is called a $J$-holomorphic curve. We shall be interested in the case where the map $u$ is
{\bf simple}, i.e., the map $u: S\rightarrow u(S)$ is generically one to one, which then defines a
parametrization of the $J$-holomorphic curve $C:=u(S)$. In this case, $C:=u(S)$ is called a
{\bf genus-$g$ $J$-holomorphic curve carrying a homology class $A$}. We denote by $\M_{A,g,J}$ the set of all such $J$-holomorphic curves $C:=u(S)$ where $u$ is simple. With this understood, we are interested in 
the structure of each $\M_{A,g,J}$, the space $\M_{A,g}:=\sqcup_{J\in\J} \M_{A,g,J}$, and the natural projection 
$\M_{A,g}\rightarrow \J$. 

To this end, let $\G$ be the Banach group of $C^{2,\alpha}$-diffeomorphisms of $S$ which preserve the orientation. Then there is a natural action of $\G$ on $\p$ with a natural projection $\overline{{\bf pr}}_\J:\p/\G\rightarrow \J$.
Moreover, $\G$ acts freely near each $(u,J_S,J)\in \p$ where $u$ is simple, and 
$\M_{A,g}$ is an open subset of $\p/\G$. The projection $\M_{A,g}\rightarrow \J$ is simply the restriction of $\overline{{\bf pr}}_\J:\p/\G\rightarrow \J$ to the open subset $\M_{A,g}$, and the moduli space 
$\M_{A,g,J}=\overline{{\bf pr}}_\J^{-1}(J)\subset \M_{A,g}$ for each $J\in\J$.

With the preceding understood, let $(u,J_S,J)\in\p$ where $u$ is simple. Let $E:=u^\ast(TX)$ be the pull-back bundle over $S$, and fix a torsion-free connection $\nabla$ on $TX$. Then the linearization of the equation $du+J\circ du\circ J_S=0$ defines an elliptic operator 
$D_{u,J}: L^{1,p}(S,E)\rightarrow L^p(S, \Lambda^{0,1}S \otimes E)$, which is given by
$$
D_{u,J}(v):=\frac{1}{2}(\nabla v+J\circ\nabla v\circ J_S+ (\nabla_v J)\circ (du\circ J_S)), \; \forall 
v\in L^{1,p}(S,E).
$$
Furthermore, $D_{u,J}=\bar{\partial}_{u,J}+R$ where $\bar{\partial}_{u,J}$ is the $J$-linear part and is
an operator of Cauchy-Riemann type, and $R$ is of zero order. With this understood, it was shown
in \cite{IvaS} that $\bar{\partial}_{u,J}$ defines a holomorphic structure on $E:=u^\ast(TX)$, and with
that understood, $du: \O(TS)\rightarrow \O(E)$ is an injective analytic morphism of analytic sheaves. The quotient
sheave $\N:=\O(E)/du(\O(TS))=\O(N_0)\oplus \N_1$, where $N_0$ is a holomorphic line bundle over
$S$ and $\N_1=\oplus_i \C^{n_i}_{a_i}$. Here $\C^{n_i}_{a_i}$ denotes the sheave which is supported 
at the critical points $a_i$ of $du: \O(TS)\rightarrow \O(E)$ and has a stalk $\C^{n_i}$ where $n_i$
is the order of zero of $du$ at $a_i$. 

The operator $D_{u,J}$ induces a so-called {\bf Gromov operator} $D_{u,J}^N:
L^{1,p}(S,N_0)\rightarrow L^p(S,\Lambda^{0,1} S\otimes N_0)$. Furthermore, 
$D_{u,J}^N=\bar{\partial}+R$ where $\bar{\partial}$ is the Cauchy-Riemann operator for the holomorphic line bundle $N_0$ and $R$ is of zero order. With this understood, we introduce the
$D$-cohomologies:
$$
H^0_D(S,N_0):=\ker D, \;\; H^1_D(S,N_0):=\text{coker } D, \mbox{ where } D:=D_{u,J}^N.
$$

\begin{lemma}
(cf. \cite{IvaS}) The map $({\bf pr}_\J)_\ast: T_{(u,J_S,J)}\p \rightarrow T_J \J$ is surjective if and only if $H^1_D(S,N_0)=0$. Moreover, with $H^1_D(S,N_0)=0$, a neighborhood of $C:=u(S)$ in
$\M_{A,g,J}$ is a smooth manifold with $T_C \M_{A,g,J}= H^0_D(S,N_0)\oplus H^0(S,\N_1)$.
\end{lemma}

For the condition $H^1_D(S,N_0)=0$, we recall the ``automatic transversality" theorem in \cite{HLS},
i.e., if $c_1(N_0)(S)>2g-2$ where $g$ is the genus of $S$, then $H^1_D(S,N_0)=0$. 

We remark that in the special case where all the $J$-holomorphic curves in $\M_{A,g,J}$, $J\in\J$,
are smoothly immersed, $N_0$ is simply the normal bundle and the sheave $\N_1$ is trivial. Moreover, one has the adjunction formula $c_1(N_0)(S)-(2g-2)=c_1(TX)(A)$. Thus in this case, under the topological condition 
$c_1(TX)(A)>0$, each space $\M_{A,g,J}$ is a smooth manifold of dimension
$$
\dim_\R \M_{A,g,J}=2(c_1(N_0)(S)+(1-g))=2(c_1(TX)(A)+(g-1))\geq 0
$$
by the Riemann-Roch Theorem. Furthermore, $\M_{A,g}$ is a Banach manifold and the projection 
$\overline{{\bf pr}}_\J:\M_{A,g}\rightarrow \J$ is a submersion of Banach manifolds. 

We summarize the discussion in the following lemma.

\begin{lemma}
(cf. \cite{HLS}) Assume each $J$-holomorphic curve in $\M_{A,g,J}$ is smoothly immersed, and moreover,
$c_1(TX)(A)>0$. Then the space $\M_{A,g,J}$ is a smooth manifold of dimension
$$
\dim_\R \M_{A,g,J}=2(c_1(TX)(A)+(g-1))\geq 0.
$$
Moreover, $\M_{A,g}$ is a Banach manifold and the projection $\overline{{\bf pr}}_\J:\M_{A,g}\rightarrow \J$ is a submersion of Banach manifolds. 
\end{lemma}

For the purpose in this paper, we shall be interested in the case where $X=\C\P^2$, with $\omega$ a K\"{a}hler form. Moreover, $S$ is of genus $0$. More precisely, for $d=1,2$, we shall consider
the space $\M(d,J)$ of degree $d$ $J$-holomorphic spheres in $\C\P^2$. Such spheres are always smoothly embedded, and the condition $c_1(TX)(A)=3d>0$ is satisfied. As a consequence, each 
$\M(d,J)$ is a smooth manifold of dimension 
$$
\dim_\R\M(d,J)=6d-2, \mbox{ where } d=1,2.
$$
Moreover, $\overline{{\bf pr}}_\J:\M(d):=\sqcup_{J\in\J} \M(d,J)\rightarrow \J$ is a submersion of Banach manifolds. 

With this understood, let $\{x_i\}$ be a finite set of distinct points in $X=\C\P^2$. For each $i$, we 
assign an integer $a_i>0$ to $x_i$, called the {\bf multiplicity}, and for each $i$ where $a_i>1$, we fix a $J$-invariant plane $T_i$ in the tangent space $T_{x_i} X$. We denote by ${\bf x}$ the data set
$\{(x_i,a_i,T_i)\}$. With this understood, we let $\M(d,J, {\bf x})$ be the subset of $\M(d,J)$, which consists of $J$-holomorphic spheres $C$ of degree $d$, such that for each $i$, $x_i\in C$, and furthermore, if $a_i>1$, we require the tangent space $T_{x_i} C$ intersects the $J$-invariant plane $T_i$ with a tangency of order at least $a_i$. We shall call $\{x_i\}$ the {\bf fixed points} associated to $\M(d,J, {\bf x})$ and $T_i$ the {\bf fixed tangent plane} at $x_i$.

Concerning the structure of the subset $\M(d,J, {\bf x})$, 
let $C\in \M(d,J, {\bf x})$, and $u: S\rightarrow X$ be a $J$-holomorphic map parametrizing $C$.
Let $z_i:=u^{-1}(x_i)$ for each $i$. Then the Gromov operator describing the structure of 
$\M(d,J, {\bf x})$ takes the form 
$$
D_{u,J}^N: L^{1,p}(S,N_0\otimes [A])\rightarrow L^p(S, \Lambda^{0,1}S \otimes (N_0\otimes [A])),
$$
where $[A]:=\sum_i -a_i [z_i]$ is a divisor on $S$ (cf. \cite{Bar, Shev}). With this understood, 
the ``automatic transversality" condition in this case is
given by $H^1_D(S, N_0\otimes [A])=0$, and with that, $\M(d,J, {\bf x})$ is a smooth manifold of dimension equaling $\dim_\R H^0_D (S, N_0\otimes [A])$. More specifically, one has the following lemma. 

\begin{lemma}
(cf. \cite{Shev})
Let ${\bf x}=\{(x_i,a_i,T_i)\}$, and assume that $3d-\sum_i a_i>0$, which means $\sum_i a_i\leq 2$ for $d=1$ 
and $\sum_i a_i\leq 5$ for $d=2$. Then $\M(d,J, {\bf x})$ is a smooth manifold of dimension 
$$
\dim_\R\M(d,J, {\bf x})= 6d-2-\sum_i 2a_i\geq 0.
$$
\end{lemma}
For convenience we shall introduce the following terminology: under the above conditions, we shall say that the moduli space $\M(d,J, {\bf x})$ is {\bf transversely cut-out}. Note that in this case, the space
$\M(d,{\bf x}):=\sqcup_{J\in\J} \M(d,J, {\bf x})$ is a Banach manifold and the projection 
$\overline{{\bf pr}}_\J:\M(d,{\bf x})\rightarrow \J$ is a submersion.

\vspace{2mm}

In order to analyze the moduli spaces of symplectic arrangements formed out of $J$-holomorphic spheres in
$\M(d,J, {\bf x})$, we shall add marked points to the $J$-holomorphic curves in $\M(d,J, {\bf x})$, which is done as follows. For simplicity, we shall only look at the case of adding a single marked point in details; the general case of multiple points is completely analogous and one can simply repeat the procedure. To this end, let 
$C\in \M(d,J, {\bf x})$, and let $u:S\rightarrow X$ be a $J$-holomorphic parametrization of $C$, where 
$z_i:=u^{-1}(x_i)\in S$ are the existing marked points on $S$. For any $z\in S\setminus \{z_i\}$, we add $z$ 
as a new marked point on $S$, and consider the map $u: (S,z_i,z)\rightarrow X$ from a marked two-sphere to $X$. The space of equivalence classes of such $u$ modulo reparametrizations of the marked two-sphere $S$ will be denoted by $\M(d,J, {\bf x}; 1)$, which is a smooth manifold of dimension 
$$
\dim_\R \M(d,J, {\bf x}; 1)=\dim_\R \M(d,J, {\bf x})+2.
$$ 
Furthermore, there is a well-defined 
evaluation map $ev: \M(d,J, {\bf x}; 1)\rightarrow X\setminus \{x_i\}$, sending $[u]$ to $u(z)$. Note that
for any $y\in X\setminus \{x_i\}$, the pre-image $ev^{-1}(y)$, if nonempty, can be identified with the
space $\M(d,J, {\bf x}^\prime)$, where ${\bf x}^\prime$ is the data set obtained from ${\bf x}$ by
adding $y$ to the set $\{x_i\}$ and giving it with a multiplicity $1$, i.e., 
$$
ev^{-1}(y)= \M(d,J, {\bf x}^\prime), \mbox{ where } {\bf x}^\prime:=\{y, (x_i,a_i,T_i)\}. 
$$
With this understood, it is easy to see that $y$ is a regular value of the evaluation map 
$ev: \M(d,J, {\bf x}; 1)\rightarrow X\setminus \{x_i\}$
if and only if $\M(d,J, {\bf x}^\prime)$ is transversely cut-out, which is guaranteed if 
$\dim_\R \M(d,J, {\bf x}^\prime)\geq 0$, or equivalently, $\dim_\R \M(d,J, {\bf x})\geq 2$. 
More generally, we denote by $\M(d,J, {\bf x}; k)$ the corresponding space of 
$J$-holomorphic curves with $k$ distinct marked points being added, which is a smooth manifold of dimension 
$$
\dim_\R \M(d,J, {\bf x}; k)=\dim_\R \M(d,J, {\bf x})+2k.
$$
Moreover, there is a well-defined evaluation map $ev$ from $\M(d,J, {\bf x}; k)$ to the $k$-fold product of $X\setminus \{x_i\}$, which is a submersion when $\dim_\R \M(d,J, {\bf x})\geq 2k$. 

The moduli space of marked $J$-holomorphic curves $\M(d,J, {\bf x}; k)$ allows us to describe the space of arrangements of $J$-holomorphic curves with a prescribed, transverse, intersection pattern
in a convenient way. For the simplest situation, consider, for $j=1,2$, the moduli spaces 
$\M(d_j,J, {\bf x}_j)$, which are transversely cut-out. Let $\M_J$ be the space of pairs 
$(C_1,C_2)\in \M(d_1,J, {\bf x}_1)\times \M(d_2,J, {\bf x}_2)$ where $C_1,C_2$ intersect transversely at one point which lies in the complement of the fixed points associated to the spaces 
$\M(d_j,J, {\bf x}_j)$, $j=1,2$. If we consider the spaces of marked $J$-curves $\M(d_j,J,{\bf x}_j; 1)$, with evaluation map $ev_j: \M(d_j,J,{\bf x}_j; 1) \rightarrow X$, and let $\Delta$ denote the diagonal of $X\times X$, 
then it is easy to see that $\M_J$ can be regarded as a subset of $(ev_1\times ev_2)^{-1}(\Delta)$.
Moreover, the transversality condition on $\M_J$ is fulfilled if the map $ev_1\times ev_2$ is transversal to $\Delta$ at the points in $ev_1\times ev_2(\M_J)$, in which case $\M_J$ is a smooth manifold, transversely cut-out, of dimension 
$$
\dim_\R \M_J=\sum_{j=1}^2 \dim_\R \M(d_j,J,{\bf x}_j).
$$
See Lemma 4.4 below. 

For our purpose in this paper, we also need to consider a slightly more general situation. Let 
$C\subset X$ be a given embedded $J$-holomorphic curve. Let $\M_J(C)$ be the subset of $\M_J$
consisting of pairs $(C_1,C_2)$, where the intersection point of $C_1,C_2$ lies in $C$, and the corresponding triple intersection point of $C_1,C_2$, and $C$ is a transversal intersection pairwise. It follows easily that
$\M_J(C)$ may be regarded as a subset of $(ev_1\times ev_2)^{-1}(\Delta\cap (C\times C))$. 

We observe the following lemma.

\begin{lemma}
(1) For $d_j=1$ or $2$, assume both of $\M(d_j,J, {\bf x}_j)$, $j=1,2$, are transversely cut-out. Then so is $\M_J$,
and $\dim_\R \M_J=\sum_{j=1}^2 \dim_\R \M(d_j,J,{\bf x}_j).$

(2) If in addition, at least one of the spaces $\M(d_j,J, {\bf x}_j)$, $j=1,2$, has a positive dimension, then the map $ev_1\times ev_2$ is transversal to $\Delta\cap (C\times C)$ at the points in $ev_1\times ev_2(\M_J(C))$, and 
$\M_J(C)$ is a smooth manifold of dimension 
$$
\dim_\R \M_J(C)=\sum_{j=1}^2 \dim_\R \M(d_j,J,{\bf x}_j)-2.
$$
\end{lemma}

\begin{proof}
For (1), note first that since $d_j=1$ or $2$, the $J$-curves in $\M(d_j,J, {\bf x}_j)$, $j=1,2$, are smoothly 
embedded. As a consequence, for any $J$-curve $C_j\in \M(d_j,J, {\bf x}_j)$, and any point $y\in C_j$,
if we regard $C_j$ naturally as an element of $\M(d_j,J, {\bf x}_j;1)$, then the map
$(ev_j)_\ast: T_{C_j} \M(d_j,J, {\bf x}_j;1)\rightarrow T_yC_j$ is surjective, where the vectors in 
$T_{C_j} \M(d_j,J, {\bf x}_j;1)$ sending nontrivially onto $T_yC_j$ are given by the variation of the marked point 
$z\in S\setminus \{z_i\}$ in the domain of the $J$-holomorphic parametrization $u: S\rightarrow X$ of $C_j$.
With this understood, assume $(C_1,C_2)\in\M_J$ and $y$ is the transverse intersection point of $C_1,C_2$.
We consider the homomorphism $\rho: T_y X\times T_y X\rightarrow T_y X$ defined by $\rho(v_1,v_2)=v_1-v_2$,
which is surjective with the kernel being precisely $T_{(y,y)}\Delta$. It follows easily that the map
$ev_1\times ev_2$ is transversal to $\Delta$ at $(y,y)\in \Delta$ if and only if the composition 
$\rho\circ (ev_1\times ev_2)_\ast$ is surjective. But this is easily seen to be true, as $y$ is the transverse intersection point of $C_1,C_2$, so that for any $v\in T_y X$, there is a unique pair 
$(v_1,v_2)\in T_y C_1\times T_y C_2$ such that $v=v_1-v_2=\rho(v_1,v_2)$, while on the other hand, as we have shown already that $(ev_j)_\ast: T_{C_j} \M(d_j,J, {\bf x}_j;1)\rightarrow T_yC_j$ is surjective for each $j=1,2$.
This finishes the proof for  (1).

For (2), the argument is similar, where instead of considering $\rho: T_y X\times T_y X\rightarrow T_y X$, we shall
replace it by $\tilde{\rho}: T_y X\times T_y X\rightarrow T_y X\times (\nu_C)_y$, where $\nu_C$ is the normal bundle of the embedded $J$-holomorphic curve $C$ in $X$, and 
$\tilde{\rho}(v_1,v_2)=(v_1-v_2, \pi(v_1)+\pi(v_2))$, with
$\pi: T_y X\rightarrow (\nu_C)_y$ being the projection. With this understood, note that the kernel of $\tilde{\rho}$
is precisely $T_{(y,y)}(\Delta\cap (C\times C))$. Now without loss of generality, assume $\M(d_1,J, {\bf x}_1)$ has a 
positive dimension. Then we note that $(ev_1)_\ast: T_{C_1} \M(d_1,J, {\bf x}_1;1)\rightarrow T_y X$ is 
surjective. With this understood, it is easy to see that $\M_J(C)$ is transversely cut-out, and is a smooth manifold of dimension
$$
\dim_\R \M_J(C)=\sum_{j=1}^2 \dim_\R \M(d_j,J,{\bf x}_j)-2,
$$
if $\tilde{\rho}: T_y X\times T_y C_2\rightarrow T_y X\times (\nu_C)_y$ is surjective. To see that  $\tilde{\rho}: T_y X\times T_y C_2\rightarrow T_y X\times (\nu_C)_y$ is surjective, we let $(v,w)\in T_y X\times (\nu_C)_y$ be any
given element. Then since by assumption the triple intersection point $y$ of $C_1,C_2,C$ is a transversal 
intersection, the projection $\pi: T_y C_2\rightarrow (\nu_C)_y$ is an isomorphism, which implies that there is a unique 
$v_2\in  T_y C_2$ such that $\pi(v_2)=\frac{1}{2}(w-\pi(v))$. With this understood, we let $v_1:=v+v_2$. Then
$\pi(v_1)=\pi(v)+\pi(v_2)=\frac{1}{2}(w+\pi(v))$, from which it follows easily that $\tilde{\rho}(v_1,v_2)=(v,w)$.
This finishes the proof for  (2), and the proof of the lemma is completed.

\end{proof}

\begin{remark}
(1) It is easy to see that Lemma 4.4(2) can be generalized from a pair of curves $C_1,C_2$ to $k$ distinct curves 
$C_1,C_2,\cdots, C_k$: let $\M_J(C)$ be the moduli space of $k$-tuples $(C_1,C_2,\cdots,C_k)$ where the intersection point of $C_1,C_2,\cdots,C_k$ lies in $C$ and the intersection is transversal. Then $\M_J(C)$
is transversely cut-out if the moduli spaces of any $k-1$ of the $k$ curves have positive dimension, and in this case, 
$\dim_\R \M(C)$ is $2(k-1)$ less of the sum of the dimensions of the moduli spaces (i.e., there is a $2(k-1)$ drop in
the total dimension).

(2) Lemma 4.4 suggests that in a pseudoholomorphic deformation argument for a symplectic arrangement, one needs to balance the following two conflicting aspects: on the one hand, if we impose that a $k$-tuple transversal intersection be a ``fixed point" throughout the deformation, the total dimension of moduli spaces will drop by $2k$, while on the other hand, if we allow the $k$-tuple intersection point ``floating" during the deformation,  the total dimension will only drop by $2k-4$, saving $4$ dimensions, but the ``floating" intersection point may become the cause of non-compactness during the deformation.

(3) In the proof of Lemma 4.4(2), if both spaces $\M(d_j,J, {\bf x}_j)$, $j=1,2$, have zero dimension, then the argument breaks down as the map $(ev_1)_\ast: T_{C_1} \M(d_1,J, {\bf x}_1;1)\rightarrow T_y X$ is no longer surjective. 
However, if we allow deformation of $J$, then the map $(ev_1)_\ast: T_{C_1} \M(d_1,J, {\bf x}_1;1)\rightarrow T_y X$ continues to be surjective, and the argument goes through. What this means is that the total moduli space
$\M(C):=\sqcup_{J\in\J} \M_J(C)$ is a Banach manifold, but since $\M_J(C)$ has a virtual dimension $-2$, the 
projection $\overline{{\bf pr}}_\J: \M(C)\rightarrow \J$ is a Fredholm map of index $-2$ between Banach manifolds,
and consequently, the set $\{J\in\J |\M_J(C)\neq\emptyset\}$ is a subspace of co-dimension $2$ in $\J$. More generally, if we fix a combinatorial type of symplectic arrangements of $J$-curves and assume the corresponding moduli space has a negative virtual dimension, then the set of $J$ for which such a symplectic arrangement of $J$-curves exists is a subspace of $\J$ of co-dimension at least $2$; in particular, its complement in $\J$ is path-connected.

\end{remark}

\subsection{Pseudoholomorphic deformations of symplectic arrangements}
With these preparations, we now give a proof of Theorem 1.11.
To simplify the notations, we shall rename the intersection points $\hat{E}_{s}$ in Figure 1 by 
$x_s$ ($s\neq 2$). On the other hand, we denote by $T_{\hat{J}}$ the tangent plane of the degree $2$ 
$\hat{J}$-spheres $\hat{F}_1$, $\hat{F}_2$ at $x_1$ (note that $\hat{F}_1$, $\hat{F}_2$ are tangent at $x_1$), and denote by $L_{\hat{J}}$ the unique degree $1$ 
$\hat{J}$-sphere passing through $x_1$ whose tangent plane is $T_{\hat{J}}$. We continue to 
denote the space of $\omega$-tame almost complex structures by $\J$; in particular, $\hat{J}\in\J$. 

With this understood, we begin the proof with the following definition. 

\begin{definition}
We fix distinct points $x_{0,3},x_{0,4}\neq x_s$, $\forall s$, on the components $\hat{F}_3$, $\hat{F}_4$ of $\hat{D}$ in Figure 1, such that none of $x_{0,3},x_{0,4}$ is contained in the degree $1$ $\hat{J}$-sphere  $L_{\hat{J}}$. With this understood, 
let $\J_0\subset \J$ denote the subset which consists of $J$ satisfying the following condition: there is a subset of $3$ distinct points in $\{x_1,x_7,x_8,x_{0,3},x_{0,4}\}$ which is contained in a degree $1$ $J$-sphere.
\end{definition}

With this understood, we observe 

\begin{lemma} 
One can choose $x_{0,3},x_{0,4}$ such that there is a $J_0\in\J\setminus \J_0$ which is integrable.
\end{lemma}

\begin{proof}
It suffices to show that there is a $J_0\in \J$ which is integrable such that $x_1,x_7,x_8$ do not lie in a
degree $1$ $J_0$-sphere. It is easy to see that one can choose $x_{0,3},x_{0,4}$ such that the rest of the condition  is satisfied so that $J_0\in\J\setminus \J_0$. 

To this end, we pick an integrable $J_0\in \J$ and suppose $x_1,x_7,x_8$ are contained in a degree $1$ 
$J_0$-sphere $L$. Then we shall modify $J_0$ as follows. Pick a diffeomorphism 
$\phi:  \C\P^2\rightarrow  \C\P^2$ which is
supported in a small neighborhood of $x_8$, with $\phi-Id$ being $C^\infty$-small, such that $x_1,x_7$ are still contained in $\phi(L)$ but $x_8$ is not contained in $\phi(L)$. With this understood, note that $\phi^\ast J_0\in
\J$ and is integrable, and $\phi(L)$ is the unique degree $1$ $\phi^\ast J_0$-sphere containing both $x_1,x_7$. 
It follows easily that $x_1,x_7,x_8$ do not lie in a degree $1$ $\phi^\ast J_0$-sphere, so that we can modify $J_0$ by replacing it with $\phi^\ast J_0$. 

\end{proof}

We shall fix the integrable almost complex structure $J_0\in\J\setminus \J_0$. Moreover, for each $J\in\J$,
we fix $J$-invariant plane in $T_{x_1}\C\P^2$, denoted by $T_J$, such that the unique degree $1$ $J$-sphere 
$L_J$, passing through $x_1$ with tangent plane $T_J$, does not contain any of the points in 
$\{x_7,x_8,x_{0,3},x_{0,4}\}$, extending the definition for $\hat{J}$ to the entire space $\J$. It is clear that we can assume $T_J$ depends on $J$ smoothly. 

With the preceding understood, we introduce the following data sets: 
$$
{\bf x}:= \{(x_1,2,T_J), x_7, x_8\}, {\bf x}_{0,3}:=\{x_7,x_{0,3}\}, {\bf x}_{0,4}:=\{x_7, x_{0,4}\}, {\bf y}:=\{x_8\}, 
$$
and consider the moduli spaces $\M(2,J,{\bf x})$, $\M(1,J,{\bf x}_{0,3})$, $\M(1,J, {\bf x}_{0,4})$ and 
$\M(1,J,{\bf y})$. Note that $\M(1,J,{\bf x}_{0,3})$, $\M(1,J, {\bf x}_{0,4})$ are of zero dimension, each containing a unique element which we denote by $C_3(J),C_4(J)$ respectively, and we shall assume $J$ has the property that $C_3(J)\neq C_4(J)$, for instance, $J=\hat{J}$ or $J\in \J\setminus\J_0$. Next let $C_1(J),C_2(J)$
be two distinct elements of $\M(2,J,{\bf x})$, which has dimension $2$, such that each of $C_3(J),C_4(J)$
intersects with $C_1(J),C_2(J)$ transversely at $(x_3)_J,(x_5)_J$ and $(x_4)_J,(x_6)_J$ respectively. Finally,
let $C_5(J),C_6(J)$ be two distinct elements of $\M(1,J,{\bf y})$, which has dimension $2$, such that 
$C_5(J),C_6(J)$ contains  $(x_3)_J,(x_6)_J$ and $(x_4)_J,(x_5)_J$ respectively. See Figure 4. We let 
$\M(J)$ denote the space of such arrangements of $J$-curves $\cup_{k=1}^6 C_k(J)$. It is clear that the symplectic arrangement $\hat{D}\in\M(\hat{J})$, and by Lemma 4.4, $\M(J)$ is transversely cut-out, and is of zero dimension. 

\begin{figure}[h]
   \centering
   \includegraphics[width=0.5\textwidth]{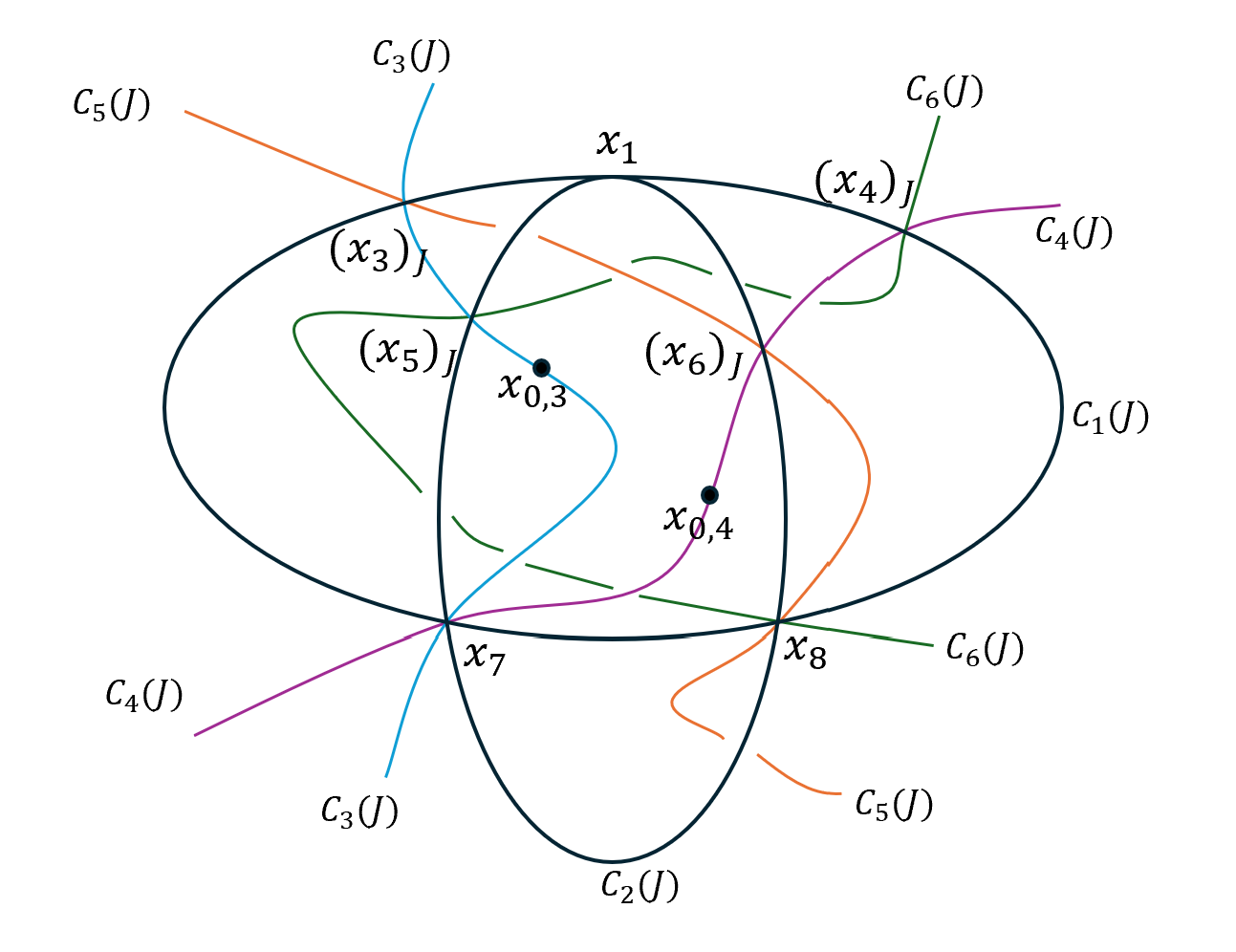}
   \caption*{Figure 4}
\end{figure}

We shall also need to consider certain arrangements of $J$-curves which have a combinatorial type equivalent to the one shown in either Figure 5(1) or Figure 5(2). (We note that in Figure 5(2), $C_3(J), C_1(J)$ and 
$C_4(J), C_2(J)$ are tangent at $x_7$.) 
We denote by $\M_1(J)$, $\M_2(J)$ the space of such arrangements of $J$-curves respectively. By introducing suitable moduli spaces of $J$-spheres and applying 
Lemma 4.4, it is easy to see that $\M_1(J)$, $\M_2(J)$ has a virtual dimension $-2$. With this understood, 
we introduce
$$
\J_i:=\{J\in\J|\M_i(J)\neq \emptyset\}, \mbox{ where } i=1,2.
$$
We note that $\J_0,\J_1,\J_2$ are subspaces of co-dimension $2$; in particular, $\J\setminus (\J_0\cup\J_1\cup\J_2)$ is path-connected (cf. Remark 4.5(3)).

\begin{figure}[h]
   \centering
   \begin{subfigure}[b]{0.45\textwidth}
      \centering
      \includegraphics[height=4.8cm]{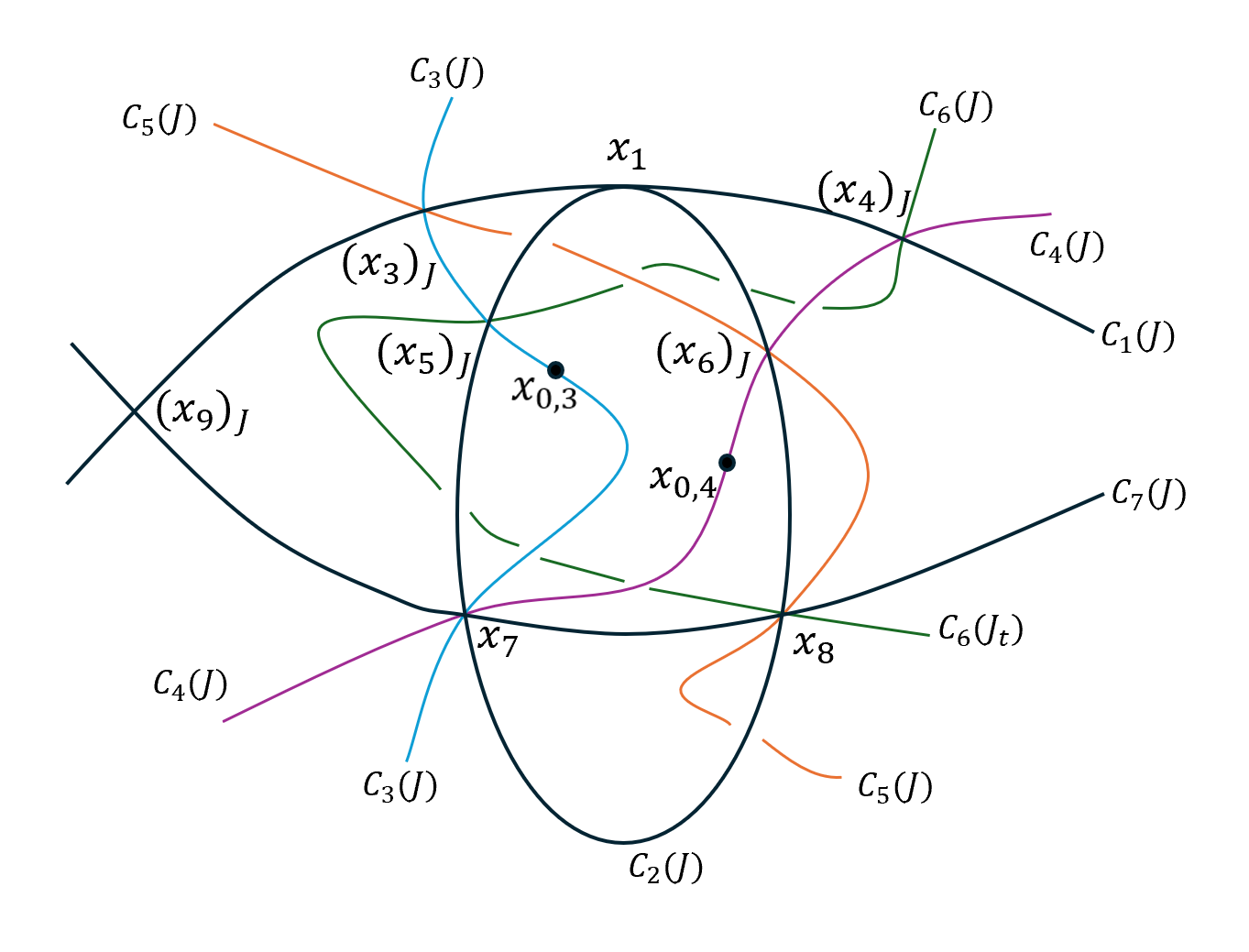}
      \caption*{Figure 5(1)}
   \end{subfigure}
   \hfill
   \begin{subfigure}[b]{0.45\textwidth}
      \centering
      \includegraphics[height=4.8cm]{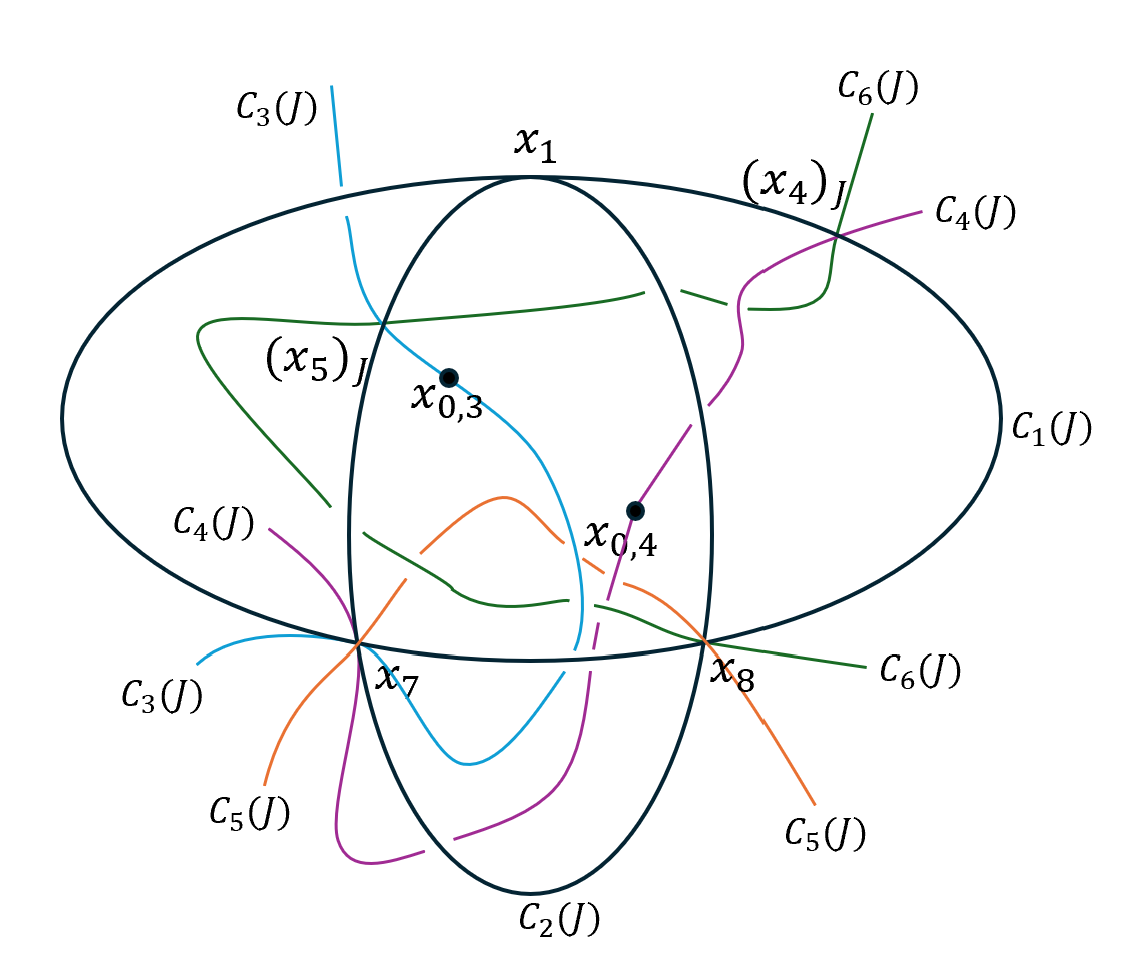}
      \caption*{Figure 5(2)}
   \end{subfigure}
\end{figure}

\vspace{2mm}

With the preceding understood, we fix a smooth path of $\omega$-tame almost complex structures $J_t$,
$t\in [0,1]$, such that $J_t=\hat{J}$ for $t=1$ and $J_t=J_0$ for $t=0$, and for $t\in (0,1)$, $J_t\in \J\setminus (\J_0\cup\J_1\cup\J_2)$. Observe that for $t\in [0,1)$, i.e., $t\neq 1$, $J_t\in\J\setminus \J_0$.

Since $\M(J)$ is transversely cut-out and $0$-dimensional, $\hat{D}\in \M(\hat{J})$ is an isolated point. It follows easily that there is an $\epsilon>0$ such that we can lift the smooth path $J_t$, $t\in (1-\epsilon,1]$, to a smooth path of $J_t$-holomorphic arrangements $\hat{D}_t\in \M(J_t)$, for $t\in (1-\epsilon,1]$, such that 
$\hat{D}_t=\hat{D}$ at $t=1$. We would like to extend the path $\hat{D}_t\in \M(J_t)$ from $t\in (1-\epsilon,1]$ to
the whole interval $[0,1]$, and this is done by a standard continuity argument, by investigating the limit of 
$\hat{D}_t$ as $t$ converges to a $t_0\in [0,1)$.

To simplify the notations, we shall denote the components $C_k(J_t)$ in $\hat{D}_t$ by $C_k(t)$, $k=1,2,\cdots,6$, and the corresponding triple intersection points $(x_s)_{J_t}$ by $(x_s)_t$, $s=3,4,5,6$. As $t\rightarrow t_0$,
we let $(x_s)_{t_0}$ be the limiting point of $(x_s)_t$. There are two cases:

\begin{itemize}
\item [{(1)}] The points $(x_s)_{t_0}$, $s=3,4,5,6$, are distinct from $x_7$.
\item [{(2)}] One of the points $(x_s)_{t_0}$, $s=3,4,5,6$, is $x_7$.
\end{itemize}

In the discussions which follow, it is important to obverse that $J_{t_0}\in \J\setminus \J_0$. Moreover, note that
as $t\rightarrow t_0$, $C_3(t),C_4(t)$ converge to the unique element $C_3(t_0)\in \M(1,J_{t_0}, {\bf x}_{0,3})$, $C_4(t_0)\in \M(1,J_{t_0}, {\bf x}_{0,4})$ respectively, which are distinct because $x_7, x_{0,3}$, and $x_{0,4}$ do not lie in a degree $1$ $J_{t_0}$-sphere (as $J_{t_0}\in \J\setminus \J_0$). With this understood, Theorem 1.11
follows immediately from the following lemma by a standard continuity argument, by the assumption 
that for $t\in (0,1)$, $J_t\in \J\setminus (\J_0\cup\J_1\cup\J_2)$. 

\begin{lemma}
Assume $\lim_{t\rightarrow t_0} \hat{D}_t$ does not lie in $\M(J_{t_0})$. Then 

(1) If we are in case (1), i.e., the points $(x_s)_{t_0}$, $s=3,4,5,6$, are distinct from $x_7$, then 
$\lim_{t\rightarrow t_0} \hat{D}_t\in \M_1(J_{t_0})$.

(2) If we are in case (2), i.e., one of the points $(x_s)_{t_0}$, $s=3,4,5,6$, is $x_7$, 
then $\lim_{t\rightarrow t_0} \hat{D}_t\in \M_2(J_{t_0})$.
\end{lemma}

\begin{proof}
Consider case (1), where the points $(x_s)_{t_0}$, $s=3,4,5,6$, are distinct from $x_7$. Since $C_3(t_0),
C_4(t_0)$ are distinct, it follows easily that $(x_3)_{t_0}, (x_5)_{t_0}\neq (x_4)_{t_0}, (x_6)_{t_0}\in C_4(t_0)$,
as $(x_3)_{t_0}, (x_5)_{t_0}\in C_3(t_0)$. Furthermore, we claim that $(x_3)_{t_0}\neq (x_5)_{t_0}$ and 
$(x_4)_{t_0}\neq (x_6)_{t_0}$. To see this, assume to the contrary that $(x_3)_{t_0}=(x_5)_{t_0}$. Then 
the limit of the degree $1$ $J_t$-sphere $C_6(t)$, which contains $(x_4)_t, (x_5)_t, x_8$, would contain the points 
$(x_4)_{t_0}, (x_3)_{t_0}, x_8$, which are on the limit of the degree $2$ $J_t$-sphere $C_1(t)$. If $(x_4)_{t_0}, (x_3)_{t_0}, x_8$ are distinct, then $C_1(t)$ must converge to a union of  degree $1$ $J_{t_0}$-spheres, one of which contains the $3$ points $(x_4)_{t_0}, (x_3)_{t_0}, x_8$. On the other hand, by the fact that $J_{t_0}\in \J\setminus \J_0$, it is easily seen that the pair of degree $1$ $J_{t_0}$-spheres as the limit of $C_1(t)$ must be 
$L_{J_{t_0}}$, the one passing through $x_1$ with tangent plane $T_{J_{t_0}}$, and the degree $1$ 
$J_{t_0}$-sphere which contains $x_7,x_8$. It is clear that none of them can contain all $3$ points 
$(x_4)_{t_0}, (x_3)_{t_0}, x_8$ because of $J_{t_0}\in \J\setminus \J_0$. This contradiction shows that 
$(x_4)_{t_0}, (x_3)_{t_0}, x_8$ cannot be distinct, hence one of $(x_4)_{t_0}, (x_3)_{t_0}$ must be
$x_8$. But this is also a contradiction because $x_{0,4},x_7,x_8$ or $x_{0,3},x_7,x_8$ are not contained in a degree $1$ $J_{t_0}$-sphere by the assumption that $J_{t_0}\in \J\setminus \J_0$. In summary, $(x_s)_{t_0}$, $s=3,4,5,6$, are distinct points, and none of them is $x_7,x_8$ or $x_1$. By a similar argument, the limit of 
$C_k(t)$ for $k=3,4,5,6$ are distinct as well. 

Next, we observe that if the limit of both $C_1(t),C_2(t)$ under $t\rightarrow t_0$ are of degree $2$, they must remain distinct. This is because if the limits of $C_1(t),C_2(t)$ were to coincide, then the $3$ distinct points 
$(x_3)_{t_0}, (x_5)_{t_0},x_7$ would lie in a degree $2$ $J_{t_0}$-sphere, which would then imply that the
degree $2$ $J_{t_0}$-sphere would intersect the degree $1$ $J_{t_0}$-sphere $C_3(t_0)$ at $3$ distinct points, a
contradiction. On the other hand, since by assumption $\lim_{t\rightarrow t_0} \hat{D}_t$ does not lie in 
$\M(J_{t_0})$, one of $C_1(t),C_2(t)$ must converge to a union of degree $1$ $J_{t_0}$-spheres. Without loss of
generality, we assume $C_1(t)$ converges to a union of degree $1$ $J_{t_0}$-spheres, which, as we have seen earlier, must be $L_{J_{t_0}}$, the degree $1$ $J_{t_0}$-sphere passing through $x_1$ with tangent plane $T_{J_{t_0}}$, and the degree $1$ $J_{t_0}$-sphere which contains $x_7,x_8$. With this understood, note that $C_2(t)$ cannot converge to a union of degree $1$ $J_{t_0}$-spheres, because if otherwise, the limit of $C_2(t)$ would be the same pair of degree $1$ $J_{t_0}$-spheres, and as a result, $(x_5)_{t_0}$ would be lying in the degree $1$ $J_{t_0}$-sphere $L_{J_{t_0}}$, and the degree $1$ $J_{t_0}$-sphere $C_3(t_0)$ would intersect
 $L_{J_{t_0}}$ at $2$ distinct points $(x_3)_{t_0}, (x_5)_{t_0}$, which is a contradiction. Now it is clear that 
 $\lim_{t\rightarrow t_0} \hat{D}_t\in \M_1(J_{t_0})$.
 
 It remains to consider case (2), where one of the points $(x_s)_{t_0}$, $s=3,4,5,6$, is $x_7$. Without loss of
 generality, assume $(x_3)_{t_0}=x_7$. Then note that the limit of $C_5(t)$, which contains $(x_3)_t,(x_6)_t, x_8$,
 must be the degree $1$ $J_{t_0}$-sphere which passes through $x_7,x_8$. Furthermore, with 
 $J_{t_0}\in \J\setminus \J_0$, we must have $(x_6)_{t_0}=x_7$ as well. With this understood, we claim that
 under $t\rightarrow t_0$, the limit of both $C_1(t),C_2(t)$ must remain to be a degree $2$ $J_{t_0}$-sphere. 
 To see this, assume to the contrary that $C_1(t)$ converges to a union of degree $1$ $J_{t_0}$-spheres, which must be $L_{J_{t_0}}$ and $C_5(t_0)$. (Here $C_5(t_0)$ denotes the limit of $C_5(t)$ which is the degree $1$ $J_{t_0}$-sphere passing through $x_7,x_8$.) Then by Gromov compactness, there is a simple closed loop on $C_1(t)$ which is pinched to a point under the convergence, and on the complement of any fixed open regular neighborhood of the loop, the convergence of $C_1(t)$ is in $C^\infty$ after a suitable parametrization of the pseudoholomorphic maps. In particular, the disc on $C_1(t)$ which contains the point $(x_3)_t$ will converge in
$C^\infty$ as an embedding to a disc on $C_5(t_0)$. In the same way, the convergence of $C_3(t)$ to 
$C_3(t_0)$ is also in $C^\infty$ after a suitable parametrization of the pseudoholomorphic maps. With this understood, the fact that the $2$ distinct points $(x_3)_t, x_7$ both lie on $C_1(t)$ and $C_3(t)$ and $(x_3)_t$
converges to $x_7$ implies that the limits of $C_1(t)$ and $C_3(t)$ should have
a tangent plane at $x_7$ with the same slope (as complex lines in $(T_{x_7} \C\P^2, J_{t_0})$). But this 
contradicts the fact that $C_5(t_0)$ and $C_3(t_0)$ are intersecting transversally at $x_7$. Hence the limit of 
both $C_1(t),C_2(t)$ remains to be a degree $2$ $J_{t_0}$-sphere, to be denoted by $C_1(t_0),C_2(t_0)$ respectively. With this understood, note that the same argument shows that $C_3(t_0),C_1(t_0)$ and $C_4(t_0),C_2(t_0)$ are tangent at $x_7$. Since  $C_3(t_0),C_4(t_0)$ are distinct, so must be $C_1(t_0),C_2(t_0)$. Finally, $C_3(t_0)$ (resp. $C_4(t_0)$) intersects with $C_2(t_0)$ (resp. $C_1(t_0)$) transversely at $x_7$, so $(x_5)_{t_0}$ (resp. $(x_4)_{t_0}$) must be the other transverse intersection point, distinct from $x_8,x_1$. It follows easily that $\lim_{t\rightarrow t_0} \hat{D}_t\in \M_2(J_{t_0})$, and the proof of Lemma 4.8 is completed.

\end{proof}

\vspace{2mm}

{\Small University of Massachusetts, Amherst.\\
{\it E-mail:} wch@umass.edu

\end{document}